\documentclass[12pt]{amsart}

\usepackage[T1]{fontenc}
\usepackage{lmodern}

\usepackage{amscd, amssymb, amsmath, amsthm}
\usepackage{enumerate, latexsym, mathrsfs}
\usepackage{xypic}

\usepackage[pdftex,lmargin=1in,rmargin=1in,tmargin=1in,bmargin=1in]{geometry}
\usepackage[
		bookmarks=true, bookmarksopen=true,%
    bookmarksdepth=3,bookmarksopenlevel=2,%
    colorlinks=true,%
    linkcolor=blue,%
    citecolor=blue]{hyperref}
\newtheorem{theorem}{Theorem}[section]
\newtheorem{lemma}[theorem]{Lemma}
\newtheorem{proposition}[theorem]{Proposition}
\newtheorem{corollary}[theorem]{Corollary}
\newtheorem{conjecture}[theorem]{Conjecture}

\theoremstyle{definition}
\newtheorem{definition}[theorem]{Definition}
\newtheorem{notation}[theorem]{Notation}
\newtheorem{convention}[theorem]{Convention}

\newtheorem{problem}[theorem]{Problem}

\newtheorem{example}[theorem]{Example}

\newtheorem{remark}[theorem]{Remark}

 \numberwithin{equation}{section}

\newcommand{\Natural}{{\mathbb N}}
\newcommand{\Real}{{\mathbb R}}
\newcommand{\Rational}{{\mathbb Q}}
\newcommand{\Complex}{{\mathbb C}}
\newcommand{\Integral}{{\mathbb Z}}

\newcommand{\ud}{{\mathrm{d}}}
\newcommand{\imunit}{{\mathbf{i}}}

\newcommand{\fabcover}{{\#}}

\newcommand{\interior}{{\mathrm{int}}}
\newcommand{\tdigraph}{{T}} 

\title[Virtual homological spectral radii]{Virtual homological spectral radii\\for automorphisms of surfaces}

\author[Yi Liu]{%
        Yi Liu} 
\address{%
        Beijing International Center for Mathematical Research, Peking University\\
				Beijing 100871, China P.R.} 
\email{%
    liuyi@math.pku.edu.cn}


\begin{document}

\begin{abstract} 
		In this paper, it is shown 
		that any surface automorphism of positive mapping-class entropy
		possesses a virtual homological eigenvalue 
		which lies outside the unit circle of the complex plane.
\end{abstract}

\maketitle

\section{Introduction}\label{Sec-introduction}
	Let $S$ be a connected compact orientable surface. 
	By an \emph{automorphism} of $S$, we mean an orientation-preserving homeomorphism
	$f\colon S\to S$.	
	For any automorphism $f$ of $S$, the \emph{homological spectral radius}
	is defined to be the spectral radius of the induced linear automorphism of $f$
	on the first homology of $S$ with complex coefficients,
	namely, the greatest modulus for all the complex eigenvalues
	of $f_*\colon H_1(S;\Complex)\to H_1(S;\Complex)$.
	Given any connected finite cover $\tilde{S}$ of $S$,
	we say that an automorphism $\tilde{f}\colon\tilde{S}\to\tilde{S}$ \emph{lifts} $f$
	if the following diagram of maps commutes:
	$$\xymatrix{
	\tilde{S} \ar[r]^{\tilde{f}} \ar[d]_{\mathrm{cov.}}& \tilde{S} \ar[d]^{\mathrm{cov.}}\\
	S \ar[r]^{f} & S
	}$$
	In this case, the homological spectral radius for $\tilde{f}$ is said to be
	a \emph{virtual homological spectral radius} for $f$.
	
	Given any automorphism $f$ of a connected compact orientable surface $S$,
	it is evident that any virtual homological spectral radius for $f$ is 
	greater than or equal to $1$,
	and less than or equal to the exponential of the mapping-class entropy of $f$.
	In general, 
	the mapping-class entropy of $f$ is defined to be 
	the infimum of the topological entropy for all the surface automorphisms 
	that are isotopic to $f$; 
	the exponential of that quantity is known to be 
	the dilatation of the pseudo-Anosov part of $f$,
	with respect to the Nielsen--Thurston classification,
	(see \cite[Corollary 10]{Kojima}).
	It is shown by C.~T.~McMullen that any virtual homological spectral radius for 
	a pseudo-Anosov surface automorphism $f$ is strictly less than the dilatation 
	if the invariant foliations for $f$ have prong singularities of odd order \cite{McMullen-entropy}.
	Whereas it is generally impossible to recover the mapping-class entropy 
	using virtual homological spectral radii,
	it is anticipated that 
	nontriviality of the mapping-class entropy may be detected.
	This leads to the following well-known conjecture due to McMullen,
	(first raised as a question for pseudo-Anosov automorphisms;
	see \cite{Koberda-NT_classification}):
	
	\begin{conjecture}\label{conjecture-vhsr}
		Given any automorphism $f$ of a compact connected orientable surface $S$,
		there exists a virtual homological spectral radius for $f$ which is strictly greater than $1$
		if and only if the mapping-class entropy of $f$ is strictly greater than $0$.
	\end{conjecture}
	%
	
	The first progress toward the conjecture is made
	by T. Koberda, who shows
	that any isotopically nontrivial surface automorphism 
	possesses a virtual homological eigenvalue other than $1$
	 \cite{Koberda-NT_classification}.
	The Nielsen--Thurston type of any given surface automorphism
	is also known to be determined
	by some decidable virtual homological action \cite{Koberda-NT_classification,Koberda--Mangahas}.
	For pseudo-Anosov automorphisms of surfaces with nonempty boundary,
	A.~Hadari shows that some virtual homological action must have infinite order \cite{Hadari-order}.
	A criterion for the existence of 
	virtual homological eigenvalues outside the unit circle
	is provided by H.~Sun,
	in terms of Mahler measure of multivariable Alexander polynomials \cite{Sun-vhsr}.
	
	The goal of this paper is to confirm Conjecture \ref{conjecture-vhsr}:
	
	\begin{theorem}\label{main-vhsr}
		The statement of Conjecture \ref{conjecture-vhsr} holds true.
	\end{theorem}
	
	In the literature, there already exist many interesting works
	which study virtual properties of fibered $3$--manifolds
	through surface automorphisms, or vice versa, 
	such as \cite{Calegari--Sun--Wang,Cooper--Long--Reid,Koberda-largeness,Masters}.	
	Every irreducible compact $3$--manifold with empty or tori boundary and of positive simplicial volume
	is known to be finitely covered by a fibered one,
	(see \cite[Theorem 9.2]{Agol-VHC}, \cite[Corollary 16.10]{Wise-book}, 
	and \cite[Corollary 1.3]{PW-mixed}).	
	Combining with other known implications of Conjecture \ref{conjecture-vhsr}
	(see \cite[Theorem 1.2]{Sun-vhsr}, \cite[Conjecture 6.1]{Silver--Williams}, and \cite[Theorem 5]{Le-abelian}), 
	we obtain the following consequences of Theorem \ref{main-vhsr}:
	
	\begin{corollary}\label{corollary-Mahler-measure}
		Every irreducible compact $3$--manifold with empty or tori boundary and of positive simplicial volume
		admits a regular finite cover 
		whose multivariable Alexander polynomial 
		is not constant zero and has Mahler measure strictly greater than $1$.
	\end{corollary}
	
	\begin{corollary}\label{corollary-torsion}
		Every irreducible compact $3$--manifold with empty or tori boundary and of positive simplicial volume
		admits a finite cover whose first integral homology contains nontrivial torsion. 
	\end{corollary}
	
	In particular, Corollary \ref{corollary-torsion} 
	answers affirmatively \cite[Question 1.8]{Sun-torsion}
	(see also \cite[Question 7.5.3]{AFW-group}).
	For closed hyperbolic $3$--manifolds,	
	the existence of virtual homological torsion is a theorem due to H.~Sun \cite{Sun-torsion}
	(compare \cite{L--Sun}).
	In that case,
	Sun actually shows that any finite abelian group
	can be realized as a direct summand of the first homology of some finite cover.
	Corollary \ref{corollary-torsion} works for more general $3$--manifolds,
	but tells less about the pattern of the virtual homological torsion part.
	
	
	
	In the rest of this section, we discuss our proof of Conjecture \ref{conjecture-vhsr}.
	The essential case is when the surface is closed and when the automorphism is pseudo-Anosov.
	We study pseudo-Anosov automorphisms of closed surfaces through their mapping tori.
	From this perspective, periodic trajectories of the suspension flow 
	play an important role in connecting the topological and the dynamical aspects.
	We develop some early work of D.~Fried \cite{Fried-sections,Fried-pptc}	on homology directions
	in the case of pseudo-Anosov suspension flows, 
	and relate it more closely with finite covers and essential submanifolds in $3$--manifold topology.
	The framework of our main proof is inspired by recent works of A.~Hadari \cite{Hadari-shadows,Hadari-order},
	(see Remark \ref{dominant-enfeoffed-remark}).
	Virtual specialization techniques for hyperbolic $3$--manifolds,
	as developed by D. Wise \cite{Wise-book},
	and I.~Agol, D.~Groves, and J.~Manning \cite{Agol-VHC,AGM-MSQT,GM-filling},
	lie in the heart of our main proof.
	
	It is not particularly easy to outline our actual proof of Conjecture \ref{conjecture-vhsr} at this point,
	as the proof evolved from several intuitive ideas into a sophisticated final form.
	In what follows, we provide an overview of the structure of this paper,
	and then illustrate some key ingredients by considering a related sample problem.
	A description of our actual strategy can be found in Section \ref{Sec-detection}, 
	immediately after we state our main working criterion (Theorem \ref{criterion-enfeoffed}).
	We put it there because Section \ref{Sec-detection} 
	is the actual beginning of our proof of Conjecture \ref{conjecture-vhsr}.
		
	\subsection*{Organization of this paper}
	This paper can be divided roughly into four parts, as listed below.
	A reader who would like to see the intuitive ideas is suggested to read through Part I,
	only skipping the appendix to Section \ref{Sec-preliminary}.
	Our complete argument for proving Conjecture \ref{conjecture-vhsr} (Theorem \ref{main-vhsr})
	occupies Parts II--IV, which relies on the preliminary materials of Section \ref{Sec-preliminary}.
		
	\begin{itemize}
	\item 
	Part I (Sections \ref{Sec-introduction}--\ref{Sec-picture}): \emph{Background and motivations.}	
	This part includes an introduction to Conjecture \ref{conjecture-vhsr}, 
	and an exposition of the key ideas in their initial form.
	In Section \ref{Sec-preliminary}, we review several natural structures 
	on mapping tori of pseudo-Anosov automorphisms.
	In Section \ref{Sec-picture}, 
	we investigate a simplified situation to motivate our actual plan.
	This section also contains a description of Fried's work on homology directions.
	\item 
	Part II (Section \ref{Sec-detection}): \emph{Criterion and plan.}
	For pseudo-Anosov automorphisms of closed surfaces, 
	we provide a criterion for the existence of a virtual homological eigenvalue outside the unit circle.
	The criterion is formulated in terms of so-called homology direction hulls
	and Mahler measure of multivariable Alexander polynomials.
	Our strategy for proving Conjecture \ref{conjecture-vhsr} is explained based on the criterion.
	\item 
	Part III (Sections \ref{Sec-perspective_of_Markov_partitions}--\ref{Sec-covering}): \emph{Covering clusters and their homology direction hulls.} 
	We revisit Fried's work on homology directions and extend the theory to the so-called covering cluster setting.
	In Section \ref{Sec-perspective_of_Markov_partitions}, we obtain the polytope description of homology direction hulls
	through a Markov partition approach. 
	In Section \ref{Sec-cluster}, we introduce what we call clusters and 
	extend the treatments of Section \ref{Sec-perspective_of_Markov_partitions} to those objects.
	In Section \ref{Sec-rcp}, we introduce and study reciprocal characteristic polynomials for clusters,
	which are analogous to the multivariable Alexander polynomials for mapping tori.
	In Section \ref{Sec-covering}, we extend the above theory to covering clusters.	
	\item 
	Part IV (Sections \ref{Sec-diversity_of_dominant_virtual_faces}--\ref{Sec-main_proof}):	\emph{Core arguments of the main proof.}
	We prove Conjecture \ref{conjecture-vhsr}, based on the results of Part II and Part III.
	In Section \ref{Sec-diversity_of_dominant_virtual_faces}, we essentially prove Conjecture \ref{conjecture-vhsr}
	for pseudo-Anosov automorphisms of closed surfaces. This case is summarized in Lemma \ref{vhsr-pA-closed}.
	In Section \ref{Sec-main_proof}, 
	we derive the general case,
	and complete the proof of Conjecture \ref{conjecture-vhsr} (Theorem \ref{main-vhsr}).
	\end{itemize}

	\subsection*{Illustration of key ingredients}
	We explain a central problem that we address in the course of our proof.
	This is roughly the first half of our core argument (Section \ref{Subsec-vertices-diversity}).
	We reformulate the problem below to avoid technical terms.
	For the moment, the reader may simply take it as a method-oriented illustration.
	It should become more clear in Section \ref{Sec-picture}
	how the problem is relevant to Conjecture \ref{conjecture-vhsr}.
	
	For any pseudo-Anosov automorphism $f$ of a connected closed orientable surface $S$,
	the mapping torus $M_f$ is naturally equipped with a suspension flow.
	The homology classes of the periodic trajectories
	is a collection of integral points in $H_1(M_f;\Real)$.
	The radial rays in $H_1(M_f;\Real)$ through those points
	form a dense subset of a unique closed linear cone $C$ in $H_1(M_f;\Real)$.
	Fried shows that $C$ is convex polyhedral, and has codimension $0$ in $H_1(M_f;\Real)$.
	In fact, $C$ is dual to the fibered cone in $H^1(M_f;\Real)$ determined by the fibration with fiber $S$,
	(see Section \ref{Sec-picture} for more details). 
	For any regular finite cover $M'$ of $M_f$ with a deck transformation group $\Gamma'$,
	there is a lifted suspension flow on $M'$, and the similarly defined cone $C'$ in $H_1(M';\Real)$
	is convex polyhedral, of codimension $0$, and is invariant under the induced action of $\Gamma'$.
	We are interested in the number of the $\Gamma'$--orbits of $1$--dimensional faces of $C'$,
	namely, the number of the \emph{extreme-ray orbits}.
	
	\begin{problem}\label{problem-vertex-diversity}
	For any given positive integer $n$, construct some regular finite cover $M'$ of $M_f$,
	such that the cone $C'$ has at least $n$ extreme-ray orbits.
	\end{problem}
	
	\begin{proof}[Outline of Solution]
	To construct $M'$ by induction, we may assume that a regular finite cover $M'_n$ of $M_f$ is given 
	with at least $n$ extreme-ray orbits in its cone $C'_n$. We construct $M'_{n+1}$ 
	with at least one more orbit in $C'_{n+1}$.
	After choosing auxiliary basepoints, we may identify the fundamental group $\Pi'_n$ of $M'_n$
	as a finite-index normal subgroup of the fundamental group $\Pi$ of $M_f$.
	The periodic trajectories in $M'_n$ can be thought of as conjugacy classes in $\Pi'_n$.
	For any quotient group $G'_n$ of $\Pi'_n$ which contains a finite-index free abelian subgroup $G''_n$,
	the preimage of $G''_n$ in $\Pi'_n$ is a finite-index subgroup $\Pi''_n$,
	and there is a finite-index subgroup $\Pi'_{n+1}$ which is contained in $\Pi''_n$ and normal in $\Pi$.
	Denote by $M''_n$ and $M'_{n+1}$ the corresponding finite covers of $M_f$.
	The above construction is of course rather general.
	We have to impose additional conditions on $\Pi'_n\to G'_n$
	to make $M'_{n+1}$ as desired.
	
	Let us assume in addition that $G'_n$ is infinite,
	and that the periodic trajectories in $M'_n$ carried by $\partial C'_n$ 
	(that is, representing homology classes of $\partial C'_n$,)
	are all sent to conjugacy classes of finite-order elements in $G'_n$. 
	Note that there are induced surjective linear maps $C'_{n+1}\to C''_n$ and $C''_n\to C'_n$.
	The additional assumption ensures that the preimage of $\partial C'_n$ in $C''_n$
	is sent to the origin under $H_1(M''_n;\Real)\to H_1(G''_n;\Real)$.
	If all the extreme rays of $C''_n$ were mapped into $\partial C'_n$ under $C''_n\to C'_n$,
	then $C''_n$ would be mapped to the origin under $H_1(M''_n;\Real)\to H_1(G''_n;\Real)$.	
	However, this is impossible because $C''_n$ has codimension $0$ in $H_1(M''_n;\Real)$,
	and $H_1(M''_n;\Real)\to H_1(G''_n;\Real)$ is surjective with positive-dimensional image.
	Therefore, some extreme ray of $C''_n$ must project a ray in the interior of $C'_n$ (plus the origin).
	It follows,
	that under the additional assumption,
	$C'_{n+1}$ must have at least one more extreme-ray orbit than $C'_n$.
	
	To construct some $\Pi'_n\to G'_n$ satisfying the additional assumption,
	we invoke the deep fact that $\Pi$ is word-hyperbolic, nonelementary, and virtually compact special \cite{Agol-VHC}.
	Suppose that there are finitely many infinite-index quasi-convex subgroups $H'_1,\cdots,H'_s$ of $\Pi'_n$
	with the following property:
	For every periodic trajectory $\gamma$ in $M'_n$ carried by $\partial C'_n$,
	the corresponding conjugacy class in $\Pi'_n$ has nonempty intersection with some $H'_i$.
	Then we can apply Wise's Special Quotient Theorem \cite[Theorem 12.7]{Wise-book} 
	 to obtain a nonelementary, word-hyperbolic, virtually compact special quotient
	$\Pi'_n\to G^*_n$, such that the image of every $H'_i$ is finite,
	(see also the argument of Lemma \ref{filling-Q} for some technical clarification).
	In particular, $G^*_n$ has virtually positive first Betti number,
	so there exists an infinite virtually free abelian quotient $G^*_n\to G'_n$.
	Then $\Pi'_n\to G'_n$ is as desired.
	
	It remains to construct a collection of subgroups $H'_i$ as above.
	We describe our construction as follows.
	We take a Markov partition $\mathcal{R}$ of $S$ with respect to $f$,
	and obtain the associated transition graph $T_{f,\mathcal{R}}$,
	which is a finite directed graph.
	In terms of symbolic dynamics,
	the dynamical cycles of $T_{f,\mathcal{R}}$
	(that is, the directly immersed loops) encode the periodic trajectories of $M_f$.
	There is a naturally induced regular finite cover $T'_n$ of $T_{f,\mathcal{R}}$,
	whose dynamical cycles encode the periodic trajectories of $M'_n$.
	With this combinatorial data, 
	we are able to construct a finite collection of (irreducible) subgraphs
	$V'_1,\cdots, V'_s$ of $T'_n$,	
	whose dynamical cycles altogether encode, and only encode, the periodic trajectories carried by $\partial C'_n$.
	(Each $V'_i$ encodes some of the periodic trajectories carried by a closed face of $C'_n$.)
	Besides, there is an induced map $T'_n\to M'_n$, which is canonical up to homotopy.
	To speak properly of fundamental groups,
	fix auxiliary basepoints of $T'_n$ and $V'_i$, and join them by a path.
	Then the subgroup $H'_i$ 
	is constructed as the image of the composite homomorphism 
	$\pi_1(V'_i)\to \pi_1(T'_n)\to \Pi'_n$.
	The detail of this construction 
	is the content of Part III (Sections \ref{Sec-perspective_of_Markov_partitions}--\ref{Sec-covering}).
	The constructed subgroups $H'_i$ are quasiconvex of infinite-index in $\Pi'_n$.
	The basic reason is that
	they lie in the kernel of some nonfibered cohomology class $\psi'_i$
	in $H^1(M'_n;\Integral)\cong \mathrm{Hom}(\Pi'_n,\Integral)$.
	In fact, $\psi'_i$ can be chosen on the boundary of the fibered cone in $H^1(M'_n;\Real)$
	that is dual to $C'_n$ in $H_1(M'_n;\Real)$.
	\end{proof}
	
	If we take a covering $Q'_i\to M'_n$ 
	that corresponds to $H'_i$ in the above construction,
	the group-theoretic conditions on $H'_i$ can be translated as follows:
	Every $Q'_i$ is a geometrically finite hyperbolic $3$--manifold of infinite volume,
	and moreover,	every periodic trajectory $\gamma$ carried by $\partial C'_n$ lifts to some $Q'_i\to M'_n$.
	One may think of each $Q'_i$ as a $\pi_1$--injectively immersed submanifolds of $M'_n$, 
	so the periodic trajectories in $M'_n$ carried by faces of $C'_n$
	are freely homotopic to loops in these submanifolds.
	Intuitively, each $Q'_i$ holds together 
	a collection of periodic trajectories,
	which are encoded by an irreducible subgraph of the covering transition graph.
	In a suitable sense, $Q'_i$ is also minimal with such property.

	This interpretation inspires us to introduce \emph{clusters},
	\footnote{The name does not suggest a connection with cluster algebra, or any other homonymous concepts in existing mathematics.}
	which formalizes the above construction,
	(see Section \ref{Sec-cluster}).
	The above $Q'_i$ are examples of \emph{covering clusters},
	as they arise in a regular finite cover of $M_f$.
	
	In dynamical and homological aspects,
	(covering) clusters behave very much like (covering) pseudo-Anosov mapping tori.
	For example, the periodic trajectories in clusters also give rise to codimension--$0$ convex polyhedral cones
	in the first real homology.
	Moreover, we can extract further interesting clusters from closed faces of a cluster's cone.
	All the clusters form a finite partial-order system,
	parallel to the system of irreducible subgraphs of the transition graph.
	This feature is particularly useful in the second half
	of our core argument, 
	because we can start from the simplest clusters
	and build up inductive arguments.
	
	Just as the Thurston norm is related to the dual fibered cone of a pseudo-Anosov mapping torus,
	the Alexander norm is related to a convex polyhedral cone in the first real homology,
	possibly of positive codimension. 
	The picture is described in Section \ref{Sec-picture} with more details.
	Upon suitable interpretation,
	the second half of our core argument (Section \ref{Subsec-dominant-diversity})
	is essentially about solving a problem 
	analogous to Problem \ref{problem-vertex-diversity} for the cone related to the Alexander norm.
	The construction is based on the result of the first half.
	It invokes separability of quasiconvex subgroups
	in virtually special word-hyperbolic groups.
	
	\subsection*{Change of language}
	Whereas homological convex polyhedral cones are quite intuitive for
	an expository purpose,
	they are not always the most convenient to work with.
	(For example, arguing with such objects would make the exposition of 
	Sections \ref{Sec-detection} and \ref{Sec-rcp} unnecessarily complicated.)
	For this reason,
	we only talk about those objects in Part I (Sections \ref{Sec-introduction}--\ref{Sec-picture}).
	In Parts II--IV (Sections \ref{Sec-detection}--\ref{Sec-main_proof}),
	we adopt some systematically introduced terminology, 
	which is more direct for our arguments.
	The reader may consult Section \ref{Sec-picture}
	for switching between different perspectives.

\subsection*{Note added in proof} 
The author just learned that Asaf Hadari had a completely independent proof of Conjecture \ref{conjecture-vhsr}
for surfaces with nonempty boundary \cite{Hadari-vhsr}. 
Using very different techniques from ours
for constructing regular finite covers,
Hadari's proof works also 
for fully irreducible outer automorphisms 
of finitely generated free groups,
and the asserted covers there are solvable.

\subsection*{Acknowledgement} The author would like to thank 
Ian Agol, Boju Jiang, Thomas Koberda, Curtis T.~McMullen,
Yi Shi, and Hongbin Sun for valuable comments and suggestions.	
The author also thanks the anonymous referees for many great suggestions
that help improving the exposition of this paper.

	%
	
		%

\section{Mapping tori of pseudo-Anosov automorphisms}\label{Sec-preliminary}
	In this preliminary section, 
	we review pseudo-Anosov automorphisms of closed orientable surfaces
	from the perspective of their mapping tori.
	We collect facts that are relevant to our discussion, 
	and set up notations that we adopt.
	Our general references include 
	Fathi--Laudenbach--Po\'enaru \cite{FLP} for Thurston's work on surfaces, 
	and Jiang \cite{Jiang-book} for Nielsen's fixed point theory,
	and Turaev \cite{Turaev-combinatorial} for combinatorial torsions.
	In the literature, 
	it is very common to assume that the surface is connected.
	We restrict ourselves to connected surfaces for this section.
	However, disconnected surfaces arise naturally 
	when we consider finite covers of mapping tori, 
	so we introduce an extension of our terminology 
	to facilitate subsequent discussion.
	This is Convention \ref{covering-mapping-torus}.
	
	Let $f$ be a pseudo-Anosov automorphism of a connected closed orientable surface $S$.
	In other words, $S$ is required to have genus at least $2$, and
	$f\colon S\to S$ is an orientation-preserving homeomorphism,
	and moreover,
	there exist a constant $\lambda>1$ and a pair of measured foliations 
	$(\mathscr{F}^{\mathtt{s}},\mu^{\mathtt{s}})$ and $(\mathscr{F}^{\mathtt{u}},\mu^{\mathtt{u}})$
	of $S$ with the property
	$$f\cdot(\mathscr{F}^{\mathtt{u}},\mu^{\mathtt{u}})=(\mathscr{F}^{\mathtt{u}},\lambda\mu^{\mathtt{u}})
	\textrm{ and }
	f\cdot(\mathscr{F}^{\mathtt{s}},\mu^{\mathtt{s}})=(\mathscr{F}^{\mathtt{s}},\lambda^{-1}\mu^{\mathtt{s}}).$$
	The constant $\lambda$ is usually called the \emph{stretching factor} of $f$.
	The measured foliations $(\mathscr{F}^{\mathtt{s}},\mu^{\mathtt{s}})$ and $(\mathscr{F}^{\mathtt{u}},\mu^{\mathtt{u}})$
	are called the \emph{stable} and the \emph{unstable} measured foliations of $f$, respectively.
	The underlying invariant foliations are transverse to each other,
	except at finitely many common singular points.
	At each singular point, 
	both of the invariant foliations have a $k$--prong singularity,
	for some and the same positive integer $k\geq3$.
		
	The \emph{mapping torus} $M_f$ of $f$ 
	can be constructed as the quotient of the product $S\times \Real$
	by the equivalence relation $(x,r+1)\sim (f(x),r)$.
	(Caution: We follow 
	the dynamicists' convention rather than the topologists' convention
	$(x,r)\sim (f(x),r+1)$.)
	There exists a distinguished (forward) \emph{suspension flow}
	$$\theta_t\colon M_f\to M_f,$$
	which is parametrized by $t\in\Real$.
	It is naturally induced by the parametrized forward flow on $S\times \Real$
	of constant unit velocity in the $\Real$ direction,
	namely, $(x,r)\mapsto(x,r+t)$.
	The product fibration of $S\times\Real$ over $\Real$, $(x,r)\mapsto r$,
	induces a distinguished fibration of $M_f$ over the circle $\Real/\Integral$ 
	and a distinguished inclusion of $S$ as a fiber of $M_f$,
	via the composite map $S\to S\times\{0\}\to S\times \Real\to M_f$.
	Therefore, the fibration map $M_f\to \Real/\Integral$
	represents a distinguished primitive cohomology class
	$$\phi_f\in H^1(M_f;\Integral),$$
	as we naturally identify $H^1(M_f;\Integral)\cong[M_f,\Real/\Integral]$.
	In other words, $\phi_f$ is the first integral cohomology class
	that is the Poincar\'e dual of $[S]\in H_2(M_f;\Integral)$.
	The cohomology class $\phi_f$ induces 
	a distinguished $\Integral$--grading of the commutative group algebra
	$\Integral H_1(M_f;\Integral)_{\mathtt{free}}$,
	where $H_1(M_f;\Integral)_{\mathtt{free}}$ 
	stands for the quotient of $H_1(M_f;\Integral)$ 
	by the submodule of torsion elements $H_1(M_f;\Integral)_{\mathtt{tors}}$.
	To be precise, any element $h\in H_1(M_f;\Integral)_{\mathtt{free}}$
	is assigned with a degree $\phi_f(h)\in\Integral$,
	as we naturally identify $H^1(M_f;\Integral)\cong\mathrm{Hom}(H_1(M_f;\Integral)_{\mathtt{free}},\Integral)$.
	Then any homogeneous element in $\Integral H_1(M_f;\Integral)_{\mathtt{free}}$
	of $\phi_f$--graded degree $m\in\Integral$ is a $\Integral$--linear combination
	of finitely many degree--$m$ elements in $H_1(M_f;\Integral)_{\mathtt{free}}$.
	
	Major algebraic topological information about $M_f$
	can be extracted from the homological action of $f$. 
	We can read off the homology and the cohomology of the mapping torus $M_f$
	from the following split short exact sequences of integral modules:
	\begin{equation}\xymatrix{
	0 \ar[r]& \mathrm{Cofix}\left(H_1(S;\Integral)\stackrel{f_*}\longrightarrow H_1(S;\Integral)\right)
	\ar[r]& H_1(M_f;\Integral) \ar[r]^-{\phi_f}& \Integral \ar[r]& 0
	}\end{equation}
	and
	\begin{equation}\xymatrix{
	0 \ar[r]& \Integral \ar[r]^-{1\mapsto \phi_f}& H^1(M_f;\Integral) \ar[r]& 
	\mathrm{Fix}\left(H^1(S;\Integral)\stackrel{f^*}\longrightarrow H^1(S;\Integral)\right) \ar[r]& 0.
	}\end{equation}
	Here $\mathrm{Fix}(f^*)$ stands for the submodule of $H^1(S;\Integral)$ 
	which is fixed under the induced action $f^*$ of $f$, 
	and $\mathrm{Cofix}(f_*)$
	stands for the quotient of $H_1(S;\Integral)$ 
	by the submodule generated by all the elements 
	$h-f_*(h)$, for all $h\in H^1(S;\Integral)$.	
	The single variable Alexander polynomial 
	of $M_f$ dual to the primitive cohomology class $\phi_f$,
	denoted as $\Delta^{\phi_f}_{M_f}(t)$,
	is determined by the reciprocal characteristic polynomial of the first homological action 
	$f_*\colon H_1(S;\Integral)\to H_1(S;\Integral)$,
	namely:
	$$\Delta^{\phi_f}_{M_f}(t)\doteq \mathrm{det}_{\Integral[t]}\left(\mathbf{1}-t\,f_*\right).$$
	Here the dotted equal symbol stands for an equality in the integral Laurent polynomial ring
	$\Integral[t,t^{-1}]$ up to a unit, 
	namely, a factor of the form $\pm t^n$ with $n\in\Integral$.
	Modulo the indeterminacy, 
	$\Delta^{\phi_f}_{M_f}$ is a monic palindromic polynomial of degree $b_1(S)$.
	
	The $\phi_f$--graded degree of the multivariable Alexander polynomial $\Delta^\fabcover_{M_f}$
	is also determined on the homological action level.
	In this paper, \emph{the multivariable Alexander polynomial} of $M_f$
	refers to the order of the finitely generated $\Integral H_1(M_f;\Integral)_{\mathtt{free}}$--module
	$H_1(M^\fabcover_f;\Integral)$, 
	where $M^\fabcover_f$ stands for (any fixed model of) the maximal free abelian covering space of $M_f$,
	on which $H_1(M_f;\Integral)_{\mathtt{free}}$ acts by deck transformations.
	The order of any finitely generated module over a Noetherian unique factorization domain (UFD)
	is known as the generator of the smallest principal ideal 
	that contains the zeroth elementary ideal	of that module,
	well defined up to a unit.
	By slightly abusing the notation,
	we denote the multivariable Alexander polynomial of $M_f$ using	any of its representatives
	$$\Delta^\fabcover_{M_f}\in \Integral H_1(M_f;\Integral)_{\mathtt{free}}.$$
	Throughout this paper,
	we treat abelian groups as multiplicative groups, rather than additive groups,
	whenever we talk about their group algebras, 
	(such as $H_1(M_f;\Integral)_{\mathtt{free}}$ in $\Integral H_1(M_f;\Integral)_{\mathtt{free}}$).
		
	The \emph{$\phi_f$--graded degree} of $\Delta^\fabcover_{M_f}$ refers to 
	the degree difference between the highest and the lowest homogeneous part of $\Delta^\fabcover_{M_f}$,
	with respect to the $\phi_f$--grading of $\Integral H_1(M_f;\Integral)_{\mathtt{free}}$.
	In other words, if $\Delta^\fabcover_{M_f}$ has a representative $\sum_h a_h h$,
	where $h$ ranges over $H_1(M_f;\Integral)_{\mathtt{free}}$ and 
	where $a_h\in \Integral$ is nonzero for only finitely many $h$,
	the $\phi_f$--graded degree of $\Delta^\fabcover_{M_f}$ is defined as
	$\max\{\phi_f(h)\colon a_h\neq0\}-\min\{\phi_f(h)\colon a_h\neq0\}$.
	This degree is clearly independent of the representative chosen.
	The $\phi_f$--graded degree of $\Delta^\fabcover_{M_f}$
	is completely determined by the Euler characteristic of $S$ and the first Betti number
	of $M_f$:
	\begin{equation}\label{deg-mAp}
	\mathrm{deg}_{\phi_f}\left(\Delta^\fabcover_{M_f}\right)=
		\begin{cases} -\chi(S) & \mbox{if } b_1(M_f)>1 \\ 
		-\chi(S)+2 & \mbox{if } b_1(M_f)=1 \end{cases}
	\end{equation}
	(See \cite[Theorem 10]{FV-survey} or \cite[Theorem 14.12]{Turaev-combinatorial}.)
	
	The multivariable Alexander polynomial $\Delta^\fabcover_{M_f}$ itself 
	reflects a deeper level	of the dynamics of $f$.
	It is intimately related with the periodic points and their indices.
	This is more precisely Theorem \ref{MAP-formula} below.
	We need some notations to state it.
	
	For any positive integer $m\in\Natural$,
	denote by $\mathrm{Per}_m(f)$ the set of the $m$--periodic points of $f$.
	For any $p\in\mathrm{Per}_m(f)$, we denote by
	\begin{equation}\label{ind-m}
	\mathrm{ind}_m(f;p)\in\Integral
	\end{equation}
	the $m$--periodic point index of $f$ at $p$.	
	
	The particular definition of periodic point index is not so much important in this paper,
	because explicit formulas are available for pseudo-Anosov automorphisms.
	However, we recall the following description for the reader's convenience:
	For any fixed point $p\in S$ of $f$, 
	let $U\subset S$ be an open neighborhood of $p$,
	and $\varphi\colon U\to \Complex$ be a local coordinate chart	with $\varphi(p)=0$.
	Denote by $\partial D_\epsilon\subset\Complex$ 
	the circle of radius $\epsilon$ centered at $0$.
	Denote by $g_\epsilon\colon \partial D_\epsilon\to \partial D_\epsilon$
	the `Gaussian' map
	$g_\epsilon(z)=\epsilon\cdot(z-(\varphi\circ f\circ\varphi^{-1})(z))/|z-(\varphi\circ f\circ\varphi^{-1})(z)|$.
	As $f$ has only isolated fixed points,
	$g_\epsilon$ is defined for all sufficiently small $\epsilon>0$.
	The \emph{fixed point index} of $f$ at $p$
	can be defined as the mapping degree of $g_\epsilon$.
	It depends only on $(f,p)$ when $\epsilon$ is sufficiently small.
	For any $m\in\Natural$,
	an \emph{$m$--periodic point} of $f$ refers to a fixed point $p\in S$ of $f^m$.
	The \emph{$m$--periodic point index} of $f$ at an $m$--periodic point $p$
	refers to the fixed point index of $f^m$ at $p$.
	There are generalized versions that works in more general settings,
	see Remark \ref{fixed-point-index} for some information.
	
	As $f$ is pseudo-Anosov, it is known that
	$\mathrm{ind}_m(f;p)$ can be expressed explicitly 
	in terms of either of the invariant foliations:
	For any $p\in\mathrm{Per}_m(f)$
	of a $k$--prong singularity, $k\geq 3$,
	we have $\mathrm{ind}_m(f;p)=1-k$ if $f^m$ 
	preserves every prong of the foliation,	or otherwise $\mathrm{ind}_m(f;p)=1$. 
	For any $p\in\mathrm{Per}_m(f)$ which is regular,
	we have $\mathrm{ind}_m(f;p)=-1$ if $f^m$ preserves any orientation of the leaf
	through $p$, or otherwise $\mathrm{ind}_m(f;p)=1$.
	To put together, the $m$--periodic index $\mathrm{ind}_m(f;p)$ equals
	$1$ minus the number of the prongs at $p$ which are invariant under $f^m$.
	Here we treat the regular case as an artificial $2$--prong singularity.
	(See Remark \ref{fixed-point-index} for some references.)



	\begin{remark}\label{fixed-point-index}
		For continuous self-maps of compact PL spaces in general,
		the \emph{fixed point index} is a (nonzero) integer 
		associated to any (essential) fixed point class,
		(see \cite{Jiang-book} for an expository introduction).
		For the standard form of a pseudo-Anosov automorphism,
		which we have been considering,
		every essential fixed point class corresponds to a unique fixed point,
		(see \cite[Corollary 2.3]{JG} and \cite[Chapter I, Definition 4.1]{Jiang-book}).
		The index formula for pseudo-Anosov periodic points as we explained above
		can be found in \cite[Lemma 2.1]{JG}, 
		(see the two types $(p,0)^+$ and $(p,k)^+,p\nmid k$ for 
		the interior fixed point case thereof).
	\end{remark}
	
	For any $m$--periodic point $p\in\mathrm{Per}_m(f)$,
	we associate a closed path $\gamma_m(f;p)$ of $M_f$ based at $p$,
	which is formed by running along the suspension flow for $m$ time units.
	Namely, $\gamma_m(f;p)(t)=\theta_t(p)$ for all $t\in[0,m]$.
	Therefore, the free abelianization class 
	\begin{equation}\label{gamma-m}
	[\gamma_m(f;p)]\in H_1(M_f;\Integral)_{\mathtt{free}}
	\end{equation}
	satisfies $\phi_f([\gamma_m(f;p)])=m$
	with respect to the grading by $\phi_f$.
	Note that (\ref{gamma-m}) and (\ref{ind-m}) are both invariant 
	for $m$--periodic points of the same $f$--iteration orbit. 

	With the above notations (\ref{ind-m}) and (\ref{gamma-m}),
	we introduce \emph{the multivariable Lefschetz zeta function} of $f$ using the expression
	\begin{equation}\label{mLzeta}
	\zeta^\fabcover_f=\exp\left(\sum_{m\in\Natural}\sum_{p\in\mathrm{Per}_m(f)} \frac{\mathrm{ind}_m(f;p)}{m}\cdot[\gamma_m(f;p)]\right),
	\end{equation}
	where $\exp(a)$ stands for the formal series $\sum_{k=0}^\infty \frac{a^k}{k!}$.
	We consider $\zeta^\fabcover_f$
	as, {\it a priori}, a formal series living in the positive half completion 
	of the commutative group algebra
	$\Rational H_1(M_f;\Integral)_{\mathtt{free}}$
	with respect to the $\phi_f$--grading.
	
	\begin{theorem}\label{MAP-formula}
		Let $S$ be a connected closed orientable surface of genus at least $2$,
		and $f$ be a pseudo-Anosov automorphism of $S$. Denote by $M_f$ the mapping torus of $f$
		and $\phi_f\in H^1(M_f;\Integral)$ the cohomology class of the distinguished fibration
		of $M_f$ over the forward oriented circle $\mathbb{R}/\mathbb{Z}$.
		Then the multivariable Alexander polynomial $\Delta^\fabcover_{M_f}$ of $M_f$ can be computed
		by the following formula:
		\begin{equation*}
			\Delta^\fabcover_{M_f}\doteq\begin{cases} \zeta^\fabcover_{f} & \mbox{if } b_1(M_f)>1 \\ 
			\zeta^\fabcover_{f}\cdot(1-t)^2 & \mbox{if } b_1(M_f)=1 \end{cases}
		\end{equation*}
		In either case,
		the right-hand side is a finite sum of terms in $H_1(M_f;\Integral)_{\mathtt{free}}$ with integral coefficients.
		Hence the dotted equality symbol stands for an equality in $\Integral H_1(M_f;\Integral)_{\mathtt{free}}$
		up to a unit.
		In the case $b_1(M_f)=1$,	
		the notation $t$ stands for the generator of $H_1(M_f;\Integral)_{\mathtt{free}}$ with $\phi_f(t)=1$.
	\end{theorem}
	
	In fact, the formula in Theorem \ref{MAP-formula}
	can be derived from well-known connections between
	twisted Alexander polynomials, twisted Reidemeister torsions,
	and twisted Lefschetz zeta functions.
	An exposition of its proof is included in the appendix at the end of this section,
	for the reader's reference.
	
	We mention another formula for the multivariable Lefschetz zeta function
	in terms of the primitive periodic trajectories.
	To any (unparametrized and basepoint-free) primitive periodic trajectory $\gamma$ 
	of the suspension flow, we associate two positive integers,
	\begin{equation}\label{pn-po}
	\mathrm{pn}(\gamma),\mathrm{po}(\gamma)\in\Natural,
	\end{equation}
	as follows.
	For any periodic point $p\in S$, 
	the set of prongs at $p$ is partitioned into orbits, 
	with respect to the action of the smallest power of $f$ that fixes $p$.
	Given any primitive periodic trajectory $\gamma$, 
	take a periodic point $p\in \gamma\cap S$.
	We define $\mathrm{pn}(\gamma)$ to be the number of prongs at $p$, 
	using either of the invariant foliations.
	(Count $\mathrm{pn}(\gamma)$ as $2$ for any regular periodic point.)
	We define $\mathrm{po}(\gamma)$ to be the number of prongs in any prong orbit at $p$,
	as explained above.
	Note that the choice of the periodic point or the prong orbit does not affect 
	the numbers $\mathrm{pn}(\gamma),\mathrm{po}(\gamma)$.
	We also observe that $\mathrm{po}(\gamma)$ always divides $\mathrm{pn}(\gamma)$,
	since $\mathrm{pn}(\gamma)/\mathrm{po}(\gamma)$ equals the number of prong orbits
	at any periodic point in $\gamma\cap S$.
	
	The following formula can be derived by some obvious manipulation of the expression (\ref{mLzeta}):
	\begin{equation}\label{mLzeta-product}
	\zeta^\fabcover_{f}=\prod_{\gamma\textrm{ primitive}}
	\left(1-[\gamma]\right)^{-1}\left(1-[\gamma]^{\mathrm{po}(\gamma)}\right)^{{\mathrm{pn}(\gamma)}/{\mathrm{po}(\gamma)}}
	\end{equation}
	where $(1-a)^{-1}$ stands for the formal series $\sum_{k=0}^\infty a^k$ and where the infinite product is taken over
	all the primitive periodic trajectories $\gamma$.
	
	In fact, the factors in (\ref{mLzeta-product}) that involve a primitive periodic trajectory $\gamma$
	correspond to the exponential of those summands in (\ref{mLzeta}) 
	that involve the periodic points in $\gamma\cap S$. 
	The $m$--periodic point set $\mathrm{Per}_m(f)$ has empty intersection with $\gamma\cap S$ 
	if $\phi_f(\gamma)$ does not divide $m$.
	Otherwise, $\mathrm{Per}_m(f)$ contains $\gamma\cap S$.
	In the latter case, for any $p\in\gamma\cap S$,
	the $m$--index
	$\mathrm{ind}_m(f;p)$ equals $1-\mathrm{pn}(\gamma)$ if $\phi_f(\gamma)\times\mathrm{po}(\gamma)$ divides $m$,
	or $1$ otherwise.
	The cardinality of $\gamma\cap S$ equals $\phi_f(\gamma)$.
	From the above observations,
	we see that the sum of those summands in (\ref{mLzeta})
	that involve periodic points of $\gamma\cap S$
	is given by
	\begin{eqnarray*}
	& &\sum_{m\in\Natural}\sum_{p\in\mathrm{Per}_m\left(f;\gamma^{\Natural}\right)} \frac{\mathrm{ind}_m(f;p)}{m}\cdot[\gamma_m(f;p)]\\
	&=&\phi_f(\gamma)\cdot\left(
	\sum_{j\in\Natural}\frac{1}{j\times\phi_f(\gamma)}\cdot[\gamma]^{j}
	-\sum_{j\in\Natural}\frac{\mathrm{pn}(\gamma)}{j\times\phi_f(\gamma)\times\mathrm{po}(\gamma)}\cdot[\gamma]^{j\times\mathrm{po}(\gamma)}\right)\\
	&=&
	-\log(1-[\gamma])+\frac{\mathrm{pn}(\gamma)}{\mathrm{po}(\gamma)}\cdot\log\left(1-[\gamma]^{\mathrm{po}(\gamma)}\right),
	\end{eqnarray*}
	where $\mathrm{Per}_m(f;\gamma^{\Natural})$ 
	stands for the intersection of $\mathrm{Per}_m(f)$ with $\gamma\cap S$,
	and where	$\log(1-a)$ stands for the formal series $\sum_{k=1}^\infty \frac{a^k}{-k}$.
	Take product over all primitive periodic trajectories $\gamma$,
	then we obtain (\ref{mLzeta-product}).

	In this paper,
	we frequently need to extend the terminology of this section to the covering setting.
	We pose the following Convention \ref{covering-mapping-torus} for subsequent reference.
	Most facts of this section can be suitably generalized, 
	(for example, see Lemma \ref{MAP-formula-covering}).
	However, to avoid confusion,
	we do not assume any unstated generalizations throughout this paper.	
	
	\begin{convention}\label{covering-mapping-torus}\
	Let $f$ be a pseudo-Anosov automorphism of a connected closed orientable surface $S$.
	Denote by $M_f$ the mapping torus of $f$
	and by $\phi_f\in H^1(M_f;\Integral)$ the distinguished cohomology class,
	and	by $\theta_t\colon M_f\to M_f$ the suspension flow with the distinguished parametrization.
	
	\begin{enumerate}
	\item
	Suppose that $\tilde{M}$ is any connected finite covering space of the mapping torus $M_f$.
	Denote by $\tilde{\theta}_t\colon \tilde{M}\to \tilde{M}$ the lifted parametrized suspension flow.
	Let $\tilde{S}$ be the preimage of the distinguished fiber $S$ in $\tilde{M}$,
	and $\tilde{f}\colon\tilde{S}\to \tilde{S}$ be the restriction of $\tilde{\theta}_1$	to $\tilde{S}$.
	By declaring $\tilde{M}$ as a \emph{covering mapping torus} over $M_f$,
	we agree to identify $\tilde{M}$ with the mapping torus $M_{\tilde{f}}$ 
	by the canonical homeomorphism
	so that $\tilde{\theta}_t$	matches 
	with the distinguished parametrized suspension flow of $M_{\tilde{f}}$.
	\item
	Note that for a covering mapping torus $\tilde{M}$ as above, the distinguished fiber
	$\tilde{S}$ is an orientable closed surface, possibly disconnected with finitely many components.
	The lifted automorphism $\tilde{f}$ 
	is an orientation-preserving homeomorphism of $\tilde{S}$
	which permutes the components cyclically and transitively.
	We agree to extend our usual terminology about pseudo-Anosov automorphisms to this setting.
	The measured foliations are understood via pull-back,
	with the stretching factor unchanged.	
	The distinguished cohomology class $\tilde{\phi}$
	in $H^1(\tilde{M};\Integral)$ is understood as the pull-back of $\phi_f$,
	which is dual to $\tilde{S}$.
	Note that $\tilde{\phi}$ is not necessarily primitive, but has divisibility $b_0(\tilde{S})$.
	\end{enumerate}
	\end{convention}

	\subsection*{Appendix to Section \ref{Sec-preliminary}}
	In this appendix,
	we provide a proof of Theorem \ref{MAP-formula}
	based on results of \cite{Turaev-combinatorial} and \cite{Fried-pptc}.
	Let $S$ be a connected closed orientable surface of genus at least $2$,
	and $f$ be a pseudo-Anosov automorphism of $S$.
	
	We start by relating the multivariable Alexander polynomial $\Delta^\fabcover_{M_f}$
	with the multivariable Reidemeister torsion $\tau^\fabcover(M_f)$.
	We mention a quick definition of the latter
	for the reader's reference.
	Given any connected finite cell complex $X$,
	and for any dimension $i\in \Integral$,
	\emph{the $i$--th multivariable Alexander polynomial}
	$X$	is defined to be 
	the order of the $\Integral H_1(X;\Integral)_{\mathtt{free}}$--module
	$H_i(X^\fabcover;\Integral)$, denoted by any representative
	$\Delta^\fabcover_{X,i}\in \Integral H_1(X;\Integral)_{\mathtt{free}}$.
	The notation $X^\fabcover$ stands for the maximal free abelian cover of $X$ as before.
	(It is customary to use $i=1$ as the default dimension,
	so $\Delta^\fabcover_{M_f}$ agrees with $\Delta^\fabcover_{M_f,1}$.)
	Under the hypothesis $\Delta^\fabcover_{X,i}\neq0$ of all $i$,
	\emph{the multivariable Reidemeister torsion} 
	of $X$ can be characterized as the following alternating product:
	$$\tau^\fabcover(X)=\frac{\prod_{i\textrm{ odd}} \Delta^\fabcover_{X,i}}{\prod_{i\textrm{ even}} \Delta^\fabcover_{X,i}}.$$
	It lives in the field of fractions $\mathrm{Frac}(\Integral H_1(X;\Integral)_{\mathtt{free}})$
	up to a unit of $\Integral H_1(X;\Integral)_{\mathtt{free}}$.
	By convention, we assign $\tau^\fabcover(X)$ to be $0$ if the hypothesis does not hold. 
	Note that for any particular $i\in\Integral$, 
	we have
	$\Delta^\fabcover_{X,i}\neq0$ if and only if 
	$H_i(X^\fabcover;\Integral)\otimes_{\Integral H_1(X;\Integral)_{\mathtt{free}}}
	\mathrm{Frac}(\Integral H_1(X;\Integral)_{\mathtt{free}})=0$,
	see \cite[Remark 4.5 (2)]{Turaev-combinatorial}.
	The multivariable Reidemeister torsion is actually naturally invariant 
	under homotopy equivalence of finite cell complexes	
	\cite[Theorem 11.3 and Corollary 11.4]{Turaev-combinatorial}.
	(Note that this invariant is called \emph{the Alexander function} 
	in \cite[Section 11.2]{Turaev-combinatorial}.)
	Therefore, it makes sense to speak of the multivariable Reidemeister torsion
	for any topological space that is homotopy equivalent to a connected finite cell complex.
	
	For (connected) closed $3$--manifolds, 
	the multivariable Reidemeister torsion is known to be completely determined
	by the multivariable Alexander polynomial.
	See \cite[Theorem 14.12]{Turaev-combinatorial} for the general formula.
	For mapping tori, the relation is explicitly as follows:
	\begin{equation}\label{tau_to_alex}
		\tau^\fabcover(M_f)\doteq 
		\begin{cases} \Delta^\fabcover_{M_f} & \mbox{if } b_1(M_f)>1 \\ 
		\Delta^\fabcover_{M_f}\cdot(1-t)^{-2} & \mbox{if } b_1(M_f)=1 \end{cases}
	\end{equation}
	where $t$ stands for the generator of $H_1(M_f)_{\mathtt{free}}$ with $\phi_f(t)=1$ in the $b_1(M_f)=1$ case.
		
	To calculate $\tau^\fabcover(M_f)$,
	we make some auxiliary choices and introduce some notations as follows.
	Choose an auxiliary element $t\in H_1(M_f;\Integral)_{\mathtt{free}}$
	with $\phi_f(t)=1$, (if $b_1(M_f)>1$). 
	Denote by $K$ the kernel of
	$\phi_f\colon H_1(M_f;\Integral)_{\mathtt{free}}\to\Integral$,
	which can be identified as
	the free abelian group 
	$\mathrm{Cofix}(f_*\colon H_1(S;\Integral)\to H_1(S;\Integral))_{\mathtt{free}}$
	of rank $b_1(M_f)-1$.
	We identify the group algebra $\Integral H_1(M_f;\Integral)_{\mathtt{free}}$
	canonically as the Laurent polynomial ring $(\Integral K)[t,t^{-1}]$
	over the commutative group algebra $\Integral K$.
	
	Choose an auxiliary cell decomposition of $S$ and 
	a cellular map $f'\colon S\to S$ that is homotopic to $f$.
	Then the mapping torus $M_{f'}$ of $f'$,
	defined as $S\times[0,1]/\sim$ with	$(x,1)\sim(f(x),0)$ for all $x\in S$,
	is furnished with a canonical cell decomposition.
	This cell decomposition of $M_{f'}$ is the induced 
	from the product cell decomposition of $S\times[0,1]$.
	The image of $S\times\{0\}$ is a distinguished copy of $S'$ embedded in $M_{f'}$,
	and every connected component of the preimage of $S'$ in $M^\fabcover_{f'}$
	is invariant under $K$ and covers $S'$ with deck transformation group exactly $K$.
	Choose a preimage component $\tilde{S}'$ of $S$ in $M^\fabcover_{f'}$.	
	
	The integral cellular chain complex 
	$(C_\bullet(M_{f'}),\partial_\bullet)$ has a canonical direct-sum decomposition
	$$C_\bullet\left(M_{f'}\right)=
	C_{\bullet,0}\left(M_{f'}\right)\oplus C_{\bullet-1,1}\left(M_{f'}\right),$$
	where the direct summands are	spanned by cells
	of the form $e\times\{0\}$ and $e\times(0,1)$, respectively.
	Choose orientation and ordering for the cells of $M_{f'}$ 
	so that they form a basis of $C_\bullet\left(M_{f'}\right)$.
	The integral cellular chain complex	
	$(C_\bullet(M^\fabcover_{f'}),\partial_\bullet)$
	is a finitely generated free $(\Integral K)[t,t^{-1}]$--module on each dimension.
	Moreover, there is a canonically induced direct-sum decomposition
	$$C_\bullet\left(M^\fabcover_{f'}\right)=
	C_{\bullet,0}\left(M^\fabcover_{f'}\right)\oplus C_{\bullet-1,1}\left(M^\fabcover_{f'}\right).$$
	Choose an auxiliary lift $\tilde{e}$ for each cell $e$ of $S'$ to $\tilde{S}'$,
	and lift the adjacent cell $e\times(0,1)\subset M_{f'}$ 
	to the unique cell adjacent to $\tilde{e}$ (on the suspension-flow forward side).
	Then, over the ordered basis formed by the lifted oriented cells,
	chains of $M^\fabcover_{f'}$ of each dimension are represented as column vectors, 
	and the boundary operators are represented as matrices over $(\Integral K)[t,t^{-1}]$.
	With respect to canonical direct-sum decomposition of $C_{\bullet}(M^\fabcover_{f'})$
	the matrices take the following block form:
	$$\xymatrix{
	{\partial_1=\left[\begin{array}{cc} * &  t\tilde{F}'_0-\mathbf{1}\end{array}\right]}, &
	{\partial_2=\left[\begin{array}{cc} * & \mathbf{1}-t\tilde{F}'_1\\ 0& *\end{array}\right]}, &
	{\partial_3=\left[\begin{array}{c} *\\ t\tilde{F}'_2-\mathbf{1}\end{array}\right]},
	}$$
	where $\tilde{F}'_0,\tilde{F}'_1,\tilde{F}'_2$ are square block matrices with entries in $\Integral K$.
	It follows from \cite[Theorem 2.2]{Turaev-combinatorial} and the homotopy invariance that
	the multivariable Reidemeister torsion can be computed by
	\begin{equation}\label{tau_to_det}
	\tau^\fabcover({M_f})\doteq\tau^\fabcover(M_{f'})\doteq\frac{\mathrm{det}_{(\Integral K)[t]}\left(\mathbf{1}-t\tilde{F}'_1\right)}
	{\mathrm{det}_{(\Integral K)[t]}\left(\mathbf{1}-t\tilde{F}'_2\right)\cdot\mathrm{det}_{(\Integral K)[t]}\left(\mathbf{1}-t\tilde{F}'_0\right)},
	\end{equation}
	understood as an equality in $(\Rational K)[t,t^{-1}]$ up to a unit.
	
	The right-hand side of (\ref{tau_to_det}) lives also in the formal series ring $(\Rational K)[[t]]$,
	(and indeed, in the multiplicative subgroup $1+(\Rational K)[[t]]t$).
	It satisfies the following equality in $(\Rational K)[[t]]$:
	\begin{equation}\label{det_expML}
	\frac{\mathrm{det}_{(\Integral K)[t]}\left(\mathbf{1}-t\tilde{F}'_1\right)}
	{\mathrm{det}_{(\Integral K)[t]}\left(\mathbf{1}-t\tilde{F}'_2\right)\cdot\mathrm{det}_{(\Integral K)[t]}\left(\mathbf{1}-t\tilde{F}'_0\right)}
	=\exp\left(\sum_{m\in \Natural} \frac{\mathcal{L}^\fabcover_m(f')}m\right),
	\end{equation}
	where $\exp(a)$ stands for the formal series $\sum_{k=0}^\infty \frac{a^k}{k!}$.
	The terms $\mathcal{L}^\fabcover_m(f')$ are defined by the expression
	\begin{equation}\label{MLN_formula}
	\mathcal{L}^\fabcover_m(f')=
	\left(\mathrm{tr}_{\Integral K}\left(\left(\tilde{F}'_2\right)^m\right)
	-\mathrm{tr}_{\Integral K}\left(\left(\tilde{F}'_1\right)^m\right)
	+\mathrm{tr}_{\Integral K}\left(\left(\tilde{F}'_0\right)^m\right)\right)\cdot t^m,
	\end{equation}
	in $(\Integral K)[t]$ for all $m\in\Natural$.
	(For other choices of the auxiliary component $\tilde{S}'$,
	or the basis of $C_\bullet(M^\fabcover_{f'})$,
	the matrices $\tilde{F}'_\bullet$ change only by conjugation
	with diagonal matrices of diagonal entries in $\{\pm1\}K$,
	so the terms $\mathcal{L}^\fabcover_m(f')$ depend only on $f'$.)
		
	As we have identified $(\Integral K)[t]\cong 
	\Integral H_1(M_{f'};\Integral)_{\mathtt{free}}
	\cong \Integral H_1(M_{f};\Integral)_{\mathtt{free}}$,
	the terms $\mathcal{L}^\fabcover_m(f')$ may also be regarded as
	homogeneous elements of $\Integral H_1(M_f;\Integral)_{\mathtt{free}}$
	of degree $m$, with respect to the $\phi_f$--grading.
	They can be actually interpreted as 
	certain version of multivariable periodic Lefschetz numbers.
	In particular, they ought to be naturally homotopy invariant  
	and	satisfy an index formula,
	in terms of periodic point indices when 
	$f$ has only isolated fixed points,
	(or in terms of indices for periodic orbit classes in general).
	Indeed, we have the identity
	\begin{equation}\label{MLN_index_formula}
	\mathcal{L}^\fabcover_m(f')=
	\sum_{p\in\mathrm{Per}_m(f)} \mathrm{ind}_m(f;p)\cdot[\gamma_m(f;p)],
	\end{equation}
	in $\Integral H_1(M_f;\Integral)_{\mathtt{free}}$ for all $m\in \Natural$.
	
	The identity (\ref{MLN_index_formula}) 
	is essentially
	a special case of
	Fried's equivariant Lefschetz formula \cite[Theorem 1.1]{Fried-pptc}.
	We explain the adaptation for the reader's reference.
	The forward suspension flow $\theta_s\colon M_{f'}\to M_{f'}$
	is a continuous family of continuous self-maps of $M_{f'}$ parametrized by $s\in[0,+\infty)$,
	and is defined as $\theta_s(x,r)=((f')^{\lfloor r+s\rfloor}(x),r+s-\lfloor r+s\rfloor)$,
	where points of $M_{f'}$ are uniquely represented their coordinates in $S\times[0,1)$.
	Denote by $\theta^\fabcover_s\colon M^\fabcover_{f'}\to M^\fabcover_{f'}$
	the lifted forward flow such that $\theta^\fabcover_0$ is the identity.
	Note that $\theta^\fabcover_s$ commutes with action 
	of the deck transformation group $H_1(M_f;\Integral)_{\mathtt{free}}$
	for all $s\in[0,+\infty)$.
	Let $\tilde{f}'\colon \tilde{S}'\to \tilde{S}'$
	be the restriction of
	composite map $t^{-1}\circ\theta^\fabcover_1\colon M^\fabcover_{f'}\to M^\fabcover_{f'}$
	to $\tilde{S}'$. 
	
	The self-map $\tilde{f}'$ of $\tilde{S}'$ is equivariant 
	with the deck transformation action of $K$.
	The integral chain modules $C_\bullet(\tilde{S}')$
	are identified with $C_{\bullet,0}(M^\fabcover_{f'})$,
	and the matrices of the induced action of $(\tilde{f}')^m$ 
	with respect to the chosen basis 
	are exactly $(\tilde{F}'_\bullet)^m$,	for all $m\in\Natural$.
	For any deck transformation $h\in K$, we obtain a subset 
	$\mathrm{Fix}(f';\tilde{f}',h)$	of $\mathrm{Fix}(f')$,
	defined as the image of $\mathrm{Fix}(h^{-1}\circ\tilde{f}')$
	under the covering projection $\tilde{S}'\to S'$.
	It follows from fixed point theory that 
	the subsets $\mathrm{Fix}(f';\tilde{f}',h)$
	are mutually isolated closed subsets of $S'$,
	and are nonempty for at most finitely many $h\in K$.
	There is a well-defined index	$\mathrm{ind}(f';\mathrm{Fix}(f';\tilde{f}',h))\in\Integral$,
	which is $0$ unless $\mathrm{Fix}(f';\tilde{f}',h)$ is nonempty.
	The equivariant Lefschetz formula of \cite[Theorem 1]{Fried-pptc}
	asserts
	\begin{equation*}
	\sum_{h\in K} \mathrm{ind}\left(f';\mathrm{Fix}\left(f';\tilde{f}',h\right)\right)\cdot h
	=
	\mathrm{tr}_{\Integral K}\left(\tilde{F}'_2\right)
	-\mathrm{tr}_{\Integral K}\left(\tilde{F}'_1\right)
	+\mathrm{tr}_{\Integral K}\left(\tilde{F}'_0\right),
	\end{equation*}
	so for $m=1$, (\ref{MLN_formula}) is transformed into 
	\begin{equation*}
	\mathcal{L}^\fabcover_1(f')=\sum_{h\in K} \mathrm{ind}\left(f';\mathrm{Fix}\left(f';\tilde{f}',h\right)\right)\cdot ht
	\end{equation*}
	Working with $(f')^m,(\tilde{f}')^m$ instead of $f',\tilde{f}'$ for any $m\in\Natural$,
	we obtain
	\begin{equation}\label{ELF_formula_m}
	\mathcal{L}^\fabcover_m(f')=\sum_{h\in K} \mathrm{ind}\left((f')^m;\mathrm{Fix}\left((f')^m;(\tilde{f}')^m,h\right)\right)\cdot ht^m,
	\end{equation}
	for all $m$.
	The right-hand side of (\ref{ELF_formula_m}) is in fact naturally invariant
	under homotopy, (as also pointed out by \cite[Section 1]{Fried-pptc},)
	so we can replace $f'$ with $f$ on the right-hand side of (\ref{ELF_formula_m}),
	(identifying $H_1(M_f;\Integral)_{\mathtt{free}}\cong H_1(M_f;\Integral)_{\mathtt{free}}$).
	Note that the involved terms are also defined for $f$ instead of $f'$.
	Since $f$ has only isolated periodic points,
	the indices become sum of indices at periodic points.
	To be precise, 
	that a periodic point $p\in\mathrm{Per}_m(f)$ lies in $\mathrm{Fix}(f^m;\tilde{f}^m,h)$ 
	means by definition $t^{-m}\cdot[\gamma_m(f;p)]=h$,
	or equivalently, $[\gamma_m(f;p)]=ht^m$.
	Moreover,
	$\mathrm{ind}(f^m;\mathrm{Fix}(f^m;\tilde{f}^m,h))$
	equals the sum of $\mathrm{ind}_m(f;p)$ 
	for all $p\in\mathrm{Per}_m(f)$ with $[\gamma_m(f;p)]=ht^m$.
	Therefore, we obtain
	\begin{eqnarray*}
	\mathcal{L}^\fabcover_m(f')&=&\sum_{h\in K} \mathrm{ind}\left(f^m;\mathrm{Fix}\left(f^m;f^m,h\right)\right)\cdot ht^m\\
	&=&\sum_{p\in\mathrm{Per}_m(f)} \mathrm{ind}_m(f;p)\cdot[\gamma_m(f;p)],
	\end{eqnarray*}
	for all $m\in\Natural$, which establishes (\ref{MLN_index_formula}).
	
	Plug (\ref{MLN_index_formula}) into the right-hand side of (\ref{det_expML}).
	Comparing with (\ref{mLzeta}), we obtain
	\begin{equation}\label{zeta_to_det}
		\zeta^\fabcover_f=\frac{\mathrm{det}_{(\Integral K)[t]}\left(\mathbf{1}-t\tilde{F}'_1\right)}
		{\mathrm{det}_{(\Integral K)[t]}\left(\mathbf{1}-t\tilde{F}'_2\right)\cdot\mathrm{det}_{(\Integral K)[t]}\left(\mathbf{1}-t\tilde{F}'_0\right)},
	\end{equation}
	in $(\Rational K)[[t]]$.
	We point out that (\ref{zeta_to_det}) 
	actually follows directly from \cite[Section 2, Theorem 2]{Fried-pptc}.
	Our above derivation only unwraps the proof thereof with notations of our context.
	
	The proof of Theorem \ref{MAP-formula} is complete
	by joining up the formulas (\ref{tau_to_alex}), (\ref{tau_to_det}), and (\ref{zeta_to_det}).

\section{A motivational picture}\label{Sec-picture}
	In this supplementary section, 
	we elaborate some idea behind our actual proof of Theorem \ref{main-vhsr}.
	We review some works of Thurston \cite{Thurston-norm}
	and Fried \cite{Fried-sections} 
	on the Thurston norm, fibered cones, and homology directions.
	See also \cite[Expos\'{e} 14]{FLP} for an exposition on these topics.
	We also mention some work of McMullen \cite{McMullen-norm} on the Alexander norm.
	Our exposition of this section 
	is specialized to the case of pseudo-Anosov mapping tori.
	We illustrate the idea in a simple, hypothetical situation.
	In fact, we prove Proposition \ref{criterion-cones},
	which serves as 
	a prototype of a main criterion (Theorem \ref{criterion-enfeoffed})
	that we use for proving Theorem \ref{main-vhsr}.
	We discuss how it leads to our actual plan of proof
	at the end of this section.
	
	This section is written mostly for an introductory purpose.
	A reader who has seen norms on the cohomology of $3$--manifolds
	might find the conceptual picture instructive
	for understanding the rest of this paper.
	To avoid confusion,
	let us make it clear that 
	our actual proof of Theorem \ref{main-vhsr} is formally independent of 
	the arguments or the notations appearing in this section.
		
	Let $f$ be a pseudo-Anosov automorphism of a connected closed orientable surface $S$.
	Denote by $M_f$ the mapping torus of $f$ and  by $\phi_f$
	the distinguished cohomology class of $M_f$.
		
	Denote by $\|\cdot\|_{\mathtt{Th}}$ the Thurston norm for the mapping torus $M_f$.
	Recall that for any connected compact orientable $3$--manifold $N$
	with empty or tori boundary, $\|\cdot\|_{\mathtt{Th}}$ 
	is a semi-norm on the real vector space	$H^1(N;\Real)$,
	and takes nonnegative integral values on the integral lattice $H^1(N;\Integral)$.
	When $N$ is closed and contains no essential spheres or tori,
	$\|\cdot\|_{\mathtt{Th}}$ is nondegenerate, (and hence a norm).
	(See \cite{Thurston-norm}.)
	In particular, this applies to the pseudo-Anosov mapping torus $M_f$,
	(see \cite[Chapter 1]{AFW-group}).
	
	Let $\mathcal{B}_{\mathtt{Th}}$ 
	be the unit ball of $\|\cdot\|_{\mathtt{Th}}$ in $H^1(M_f;\Real)$.
	It is a compact convex polyhedron (also known as a \emph{polytope}).
	It is symmetric about the origin and of codimension $0$ in $H^1(M;\Real)$.
	The dual ball $\mathcal{B}^*_{\mathtt{Th}}$
	refers to the subset of $H_1(M_f;\Real)$,
	such that $x\in\mathcal{B}^*_{\mathtt{Th}}$ holds if and only if 
	$|\psi(x)|\leq 1$ holds for all $\psi\in\mathcal{B}_{\mathtt{Th}}$.
	(One may recongnize $\mathcal{B}^*_{\mathtt{Th}}$ as the unit ball of the norm on $H_1(M_f;\Real)$
	dual to $\|\cdot\|_{\mathtt{Th}}$.)
	By duality, the $i$--dimensional faces of $\mathcal{B}^*_{\mathtt{Th}}$
	correspond bijectively to the codimension--$i$ faces of $\mathcal{B}_{\mathtt{Th}}$.
		
	Denote by $\mathcal{C}_{\mathtt{Th}}(\phi_f)$ the fibered cone that contains the distinguished cohomology class $\phi_f$.
	Recall that $\mathcal{C}_{\mathtt{Th}}(\phi_f)$ can be constructed 
	as the maximal open subset of $H^1(M_f;\Real)$ that contains $\phi_f$,
	such that the restriction of $\|\cdot\|_{\mathtt{Th}}$ to the closure of $\mathcal{C}_{\mathtt{Th}}(\phi_f)$ is linear.
	An integral cohomology class $\psi\in \mathcal{C}_{\mathtt{Th}}(\phi_f)\cap H^1(M_f;\Integral)$ 
	can be characterized by the property that $\psi$ is dual to an embedded closed subsurface
	which is transverse to the suspension flow everywhere and oriented in the flow direction.
	The closure of $\mathcal{C}_{\mathtt{Th}}(\phi_f)$ in $H^1(M_f;\Real)$ 
	is a convex polyhedral cone formed by radial rays
	(See \cite{Thurston-norm}, or \cite[Expos\'{e} 14]{FLP} for an exposition.)
	
	The dual cone $\mathcal{C}^*_{\mathtt{Th}}(\phi_f)$ refers to the subset of $H_1(M_f;\Real)$,
	such that	$x\in\mathcal{C}^*_{\mathtt{Th}}(\phi_f)$ holds 
	if and only if $\psi(x)>0$ holds for all $\psi\in\mathcal{C}_{\mathtt{Th}}(\phi_f)$.
	It follows that $\mathcal{C}^*_{\mathtt{Th}}(\phi_f)$ is a convex polyhedral closed cone.
	As $f$ is pseudo-Anosov,
	the dual cone $\mathcal{C}^*_{\mathtt{Th}}(\phi_f)$ can be characterized 
	as the smallest scaling-invariant convex closed subset of $H_1(M_f;\Real)$
	that contains all the homology classes of periodic trajectories of the suspension flow,
	\cite[Expos\'{e} 14]{FLP}.
	We also mention that
	the radial rays in $\mathcal{C}^*_{\mathtt{Th}}(\phi_f)$
	correspond to	the homology directions of the suspension flow.
	For any smooth nonsingular flow on a compact smooth manifold,
	a homology direction is roughly an accummulation point
	of normalized homology classes determined by long and nearly closed trajectories,
	as introduced by Fried \cite{Fried-sections}.
	
	The objects $\mathcal{B}_{\mathtt{Th}},\mathcal{B}^*_{\mathtt{Th}},\mathcal{C}_{\mathtt{Th}}(\phi_f),\mathcal{C}^*_{\mathtt{Th}}(\phi_f)$
	are related as follows:
	The fibered cone $\mathcal{C}_{\mathtt{Th}}(\phi_f)$ determines a unique top-dimensional open face
	of $\mathcal{B}_{\mathtt{Th}}$, such that
	$\mathcal{C}_{\mathtt{Th}}(\phi_f)$	is the union of the rays through that face with the origin deleted.
	This top-dimensional face is dual to a unique vertex $v_f$ of $\mathcal{B}^*_{\mathtt{Th}}$.
	In fact, $v_f$ is determined by the property 
	$$\|\psi\|_{\mathtt{Th}}=\psi(v_f)$$
	for all $\psi\in\mathcal{C}_{\mathtt{Th}}(\phi_f)$, (see \cite{Thurston-norm}).
	We identify the tangent space $T_{v_f}H_1(M_f;\Real)$ of $H_1(M_f;\Real)$ at $v_f$ 
	naturally with $H_1(M_f;\Real)$.
	Then the tangent vectors at $v_f$ that point into $\mathcal{B}^*_{\mathtt{Th}}$
	form a cone, which is exactly $\mathcal{C}^*_{\mathtt{Th}}(\phi_f)$ under the identification.
	
	There is a similar picture regarding the Alexander norm
	$\|\cdot\|_{\mathtt{A}}$ for the mapping torus $M_f$, as we describe below.
	For any $\psi\in H^1(M_f;\Real)$,
	$\|\cdot\|_{\mathtt{A}}$ is defined to be the $\psi$--degree of the multivariable Alexander polynomial $\Delta^\fabcover_{M_f}$.
	Namely,
	$$\|\psi\|_{\mathtt{A}}=\max\{\psi(h)\colon a_h\neq0\}-\min\{\psi(h)\colon a_h\neq0\},$$
	where $\sum_{h} a_h h\in \Integral H_1(M_f;\Integral)_{\mathtt{free}}$ 
	is any representative of $\Delta^\fabcover_{M_f}$.
	In general, $\|\cdot\|_{\mathtt{A}}$ is a semi-norm.
	If $b_1(M_f)>1$, the comparison
	$$\|\psi\|_{\mathtt{A}}\leq\|\psi\|_{\mathtt{Th}}$$
	holds for all $\psi\in H^1(M_f;\Real)$,
	by McMullen \cite{McMullen-norm},
	and the equality holds
	if $\psi$ equals $\phi_f$,
	by (\ref{deg-mAp}).
	(For $b_1(M_f)=1$,
	one simply obtains 
	$\|\phi_f\|_{\mathtt{A}}=\|\phi_f\|_{\mathtt{Th}}+2=b_1(S)$
	by (\ref{deg-mAp}).)
		
	Denote by $\mathcal{B}_{\mathtt{A}}$ the unit ball of $\|\cdot\|_\mathtt{A}$.
	Define the dual ball $\mathcal{B}^*_{\mathtt{A}}$ in $H_1(M_f;\Real)$ 
	by the property that	$x\in\mathcal{B}^*_{\mathtt{A}}$ holds if and only if
	$|\psi(x)|\leq 1$ holds for all $\psi\in\mathcal{B}_{\mathtt{A}}$.
	Under the assumption $b_1(M_f)>1$,
	it follows from the above norm comparison that 
	$\mathcal{B}^*_{\mathtt{A}}$ is contained in $\mathcal{B}^*_{\mathtt{Th}}$,
	and moreover,
	the norm equality for $\phi_f$
	implies that the vertex $v_f$ of $\mathcal{B}^*_{\mathtt{Th}}$
	dual to the fibered face determined by $\phi_f$
	is also a vertex of $\mathcal{B}^*_{\mathtt{A}}$.
	Let $\mathcal{C}_{\mathtt{A}}(\phi_f)$ 
	be the open cone of $H^1(M_f;\Real)$
	over the top-dimensional face of $\mathcal{B}_{\mathtt{A}}$
	that is dual to $v_f$.
	Let $\mathcal{C}^*_{\mathtt{A}}(\phi_f)$ 
	be the closed cone of tangent vectors of $H_1(M_f;\Real)$
	at $v_f$ which points into $\mathcal{B}^*_{\mathtt{A}}$.
	We naturally consider $\mathcal{C}^*_{\mathtt{A}}(\phi_f)$ as a closed cone
	in $H_1(M_f;\Real)$ as before.
	
	The dual ball $\mathcal{B}^*_{\mathtt{A}}$ is a polytope
	(possibly of positive codimension) in $H_1(M_f;\Real)$ and symmetric about the origin.
	In fact, 
	it follows from the definition that
	$\mathcal{B}^*_{\mathtt{A}}$
	agrees with the Newton polytope of the multivariable Alexander polynomial $\Delta^\fabcover_{M_f}$,
	up to an integral translation of $H_1(M_f;\Real)$ (as an affined linear space).
	Therefore, 
	$\mathcal{B}_{\mathtt{A}}$ is a (possibly noncompact) convex polyhedron in $H^1(M_f;\Real)$,
	symmetric about the origin.
	The cones $\mathcal{C}^*_{\mathtt{A}}(\phi_f)$ and $\mathcal{C}_{\mathtt{A}}(\phi_f)$ 
	are both convex polyhedral.
	For $b_1(M_f)>1$, 
	the logarithmic multivariable Lefschetz zeta function $\log\zeta^\fabcover_f$
	is related to $\mathcal{C}^*_{\mathtt{A}}(\phi_f)$
	like $\Delta^\fabcover_{M_f}$ to $\mathcal{B}^*_{\mathtt{A}}$.
	(In fact, Theorem \ref{MAP-formula} implies
	that $\mathcal{C}^*_{\mathtt{A}}(\phi_f)$	
	is the smallest scaling-invariant convex closed subset of $H_1(M_f;\Real)$
	that contains all the homology classes in the formal series expansion
	of $\log\zeta^\fabcover_f$ with nonzero coefficients.)
	
	The following proposition is an application 
	of the convex polyhedral cones and balls as mentioned above.
		
	\begin{proposition}\label{criterion-cones}
	Let $f$ be a pseudo-Anosov automorphism of a connected closed orientable surface $S$.
	Identify $\Integral H_1(M_f;\Integral)_{\mathtt{free}}$	
	as a multivariable polynomial ring $\Integral[z_1^{\pm1},\cdots,z_n^{\pm1}]$.
	
	Suppose that there exist at least $-\chi(S)+1$ closed cone faces 
	of the closed cone $\mathcal{C}^*_{\mathtt{Th}}(\phi_f)$,
	mutually disjoint except at the origin,
	and 
	suppose that the intersection of each of them with the closed cone $\mathcal{C}^*_{\mathtt{A}}(\phi_f)$ 
	contains a ray.
	Then (any representative of) the multivariable Alexander polynomial $\Delta^\fabcover_{M_f}$
	is not a product of monomials and cyclotomic polynomials evaluated on monomials.
	\end{proposition}
	
	\begin{proof}
	We observe that the corners of $\mathcal{C}^*_{\mathtt{Th}}(\phi_f)$ and $\mathcal{C}^*_{\mathtt{A}}(\phi_f)$
	at their vertices are isomorphic to the corners of $\mathcal{B}^*_{\mathtt{Th}}(\phi_f)$ and $\mathcal{B}^*_{\mathtt{A}}(\phi_f)$
	at the vertex $v_f$, respectively.
	In particular, $\mathcal{C}^*_{\mathtt{Th}}(\phi_f)$ contains $\mathcal{C}^*_{\mathtt{A}}(\phi_f)$,
	as $\mathcal{B}^*_{\mathtt{Th}}(\phi_f)$ contains $\mathcal{B}^*_{\mathtt{A}}(\phi_f)$.
	Then the condition implies that $\mathcal{C}^*_{\mathtt{A}}(\phi_f)$ 
	has at least $-\chi(S)+1$ closed faces which are mutually disjoint except at the origin.
	It follows that
	$\mathcal{B}^*_{\mathtt{A}}$ has at least $-\chi(S)+1$ edges ($1$--dimensional faces) adjacent to the vertex $v_f$.
	
	Assume on the contrary that $\Delta^\fabcover_{M_f}$ is represented by
	a multivariable Laurent polynomial $\pm w_0 P_1(w_1)\cdots P_r(w_r)$,
	where $w_i$ are monomials $z_1^{e_{i,1}}\cdots z_n^{e_{i,n}}$ and $P_i$ are cyclotomic polynomials.
	We observe that the Newton polytope of each $P_i(w_i)$ is a segment in $H_1(M_f;\Real)$ joining the origin and $w_i$.
	Note that $\mathcal{B}^*_{\mathtt{A}}$ is isomorphic to the Newton polytope of $\Delta^\fabcover_{M_f}$.
	Since $\phi_f$ has a unique minimum point $v_f$ on $\mathcal{B}^*_{\mathtt{A}}$,
	we obtain $\deg_{\phi_f}(P_i(w_i))>0$ for all $i$, 
	(otherwise the Newton polytope of $P_1(w_1)\cdots P_r(w_r)$ would have a closed face of positive dimension
	on which $\phi_f$ is minimized).
	Then we estimate
	$r\leq \deg_{\phi_f}(P_1(w_1))+\cdots+\deg_{\phi_f}(P_r(w_r))\leq \deg_{\phi_f}(\Delta^\fabcover_{M_f})\leq -\chi(S)$.
	However, it follows that the Newton polytope of $w_0 P_1(w_1)\cdots P_r(w_r)$
	has at most $-\chi(S)$ edges adjacent to the vertex $v_f$. This is a contradiction.	
	\end{proof}
	
	Under the assumptions of Proposition \ref{criterion-cones},
	it can be implied that some finite covering and automorphism lift $(S',f')$
	of $(S,f)$ has a homological spectral radius $>1$.
	This follows directly from Kronecker's theorem \cite[Theorem 3.10]{Everest--Ward}
	and a criterion of Sun \cite[Theorem 1.2]{Sun-vhsr}.
	(The former tells that $\Delta^\fabcover_{M_f}$ has Mahler measure $>1$,
	while the latter relates this property with the existence of a virtual homological eigenvalue
	outside the unit circle.)
	
	We make some comments 
	concerning our actual proof of Theorem \ref{main-vhsr}.
	Whereas Proposition \ref{criterion-cones} supplies a sufficient condition 
	for virtual spectral radii $>1$, the condition may very often fail to hold.
	We loosen the condition  (with some reformulation) to obtain a more realistic criterion,
	which works in the covering setting.
	This leads to Theorem \ref{criterion-enfeoffed}.
	The generalized criterion roughly asks for 
	a regularly covering mapping torus $M_{\tilde{f}}$ 
	with a deck transformation group $\Gamma$ as follows:
	With respect to the induced action of $\Gamma$,
	the dual cone $\mathcal{C}^*_{\mathtt{Th}}(\phi_{\tilde{f}})$
	should have at least $-\chi(S)+1$ closed cone face orbits,
	mutually disjoint except at the origin,
	and moreover,
	each of those orbits should meet 
	$\mathcal{C}^*_{\mathtt{A}}(\phi_{\tilde{f}})$ nontrivially.
	Proving the generalized criterion	is 
	mostly about convex geometry of multivariable Laurent polynomials,
	or their formal logarithms.
	This will become clear in Section \ref{Sec-detection} upon some suitable interpretation.
	Making use of the criterion, however,
	requires a systematic generalization of Fried's original work
	\cite{Fried-sections}	on homology directions
	to the covering setting,
	and also to a more technical cluster setting that we come up with. 
	One way to present our generalization of Fried's work 
	is to deal with covering clusters directly.
	Another possible option is to redo the basic case, 
	in a reasonably self-contained, sufficiently detailed, and easily extendable manner,
	and then, to specify necessary modifications for the general setting.
	In this paper, we take the latter approach,
	as it might be more accessible to a general reader.
	We point out related existing works in remarks
	for the reader's reference.

\section{A criterion for nontrivial Mahler measure}\label{Sec-detection}
	In this section, we provide 
	a criterion for a regularly covering pseudo-Anosov mapping torus
	to have a multivariable Alexander polynomial with nontrivial Mahler measure
	(Theorem \ref{criterion-enfeoffed}).
	We state the criterion after introducing 
	some terminology (Definition \ref{homology_directions}), 
	which is repeatedly used in subsequent sections.
	We explain our strategy for proving Conjecture \ref{conjecture-vhsr}
	after stating Theorem \ref{criterion-enfeoffed}.	
	The proof of Theorem \ref{criterion-enfeoffed} forms the body of this section.
	
	\begin{definition}\label{homology_directions}
	Let $f$ be a pseudo-Anosov automorphism of a connected orientable closed surface $S$.
	Denote by $M_f$ the mapping torus of $f$ and  by $\phi_f$
	the distinguished cohomology class of $M_f$.
	\begin{enumerate}
	\item	
	Denote by $\mathbf{P}(H_1(M_f;\Real))$ the projectivization of $H_1(M_f;\Real)$,
	whose points are considered to be 
	the real $1$--dimensional linear subspaces of $H_1(M_f;\Real)$.
	Denote by $\mathbf{A}(M_f,\phi_f)$
	the complement of the (projective) hyperplane $\mathbf{P}(\mathrm{Ker}(\phi_f))$.
	We furnish $\mathbf{A}(M_f,\phi_f)$ with the naturally induced affine linear space structure.
	A projective point $l\in\mathbf{A}(M_f,\phi_f)$ is called a \emph{periodic homology direction}
	for $M_f$
	if $l$ contains the real homology class of	a periodic trajectory $\gamma$ of the suspension flow.
	The \emph{homology direction hull} for $M_f$, denoted as
	$$\mathcal{D}(M_f,\phi_f)\subset \mathbf{A}(M_f,\phi_f),$$
	is defined to be 
	the affine linear convex hull of all the periodic homology directions.
	%
	\item We say that a subset $E$ of $\mathcal{D}(M_f,\phi_f)$
	is \emph{semi-extreme} 
	if there is some convex function $\mathbf{A}(M_f,\phi_f)\to\Real$
	and if $E$ consists of the maximum points
	of the function restricted to $\mathcal{D}(M_f,\phi_f)$.
	For any semi-extreme subset $E$ of $\mathcal{D}(M_f,\phi_f)$,
	the \emph{$E$--part of the multivariable Lefschetz zeta function}
	is defined as 
	$$\zeta^\fabcover_f[E]=1+\sum_{u\textrm{ over }E} \mathrm{coef}\left(\zeta^\fabcover_{f};u\right)\cdot u,$$
	where the summation is taken 
	over all $u\in H_1(M_f;\Integral)_{\mathtt{free}}$ 
	with $\phi_f(u)>0$ and $\Real u\in E$, and 
	where $\mathrm{coef}(\zeta^\fabcover_{f};u)\in\Rational$ 
	stands for 
	the coefficient of 
	the multivariable Lefschetz zeta function $\zeta^\fabcover_{f}$ 
	at $u$ in the formal series expansion,
	(see (\ref{mLzeta})).
	We consider $\zeta^\fabcover_f[E]$
	as living in the $\phi_f$--graded forward-half completion of $\Rational H_1(M_f;\Integral)_{\mathtt{free}}$.
	For any semi-extreme subset $E$ of $\mathcal{D}(M_f,\phi_f)$,
	we say that $E$ is \emph{dominant} if it satisfies
	$$\zeta^\fabcover_f[E]\neq1.$$
	\item More generally, for any connected finite cover $\tilde{M}$ of $M_f$,
	the \emph{homology direction hull} $\mathcal{D}(\tilde{M},\tilde{\phi})$
	for $\tilde{M}$	and \emph{dominant semi-extreme subsets} of $\mathcal{D}(\tilde{M},\tilde{\phi})$
	are defined in the same way as above,
	by decalaring $\tilde{M}$
	as a covering mapping torus with the distinguished cohomology class $\tilde{\phi}$
	according to Convention \ref{covering-mapping-torus}.
	\end{enumerate}
	\end{definition}
	
	\begin{remark}\label{remark_homology_directions}\
	\begin{enumerate}
	\item The complement of any projective hyperplane in a real projective $n$--space
	is naturally an affine linear $n$--space, (that is, a homogeneous space isomorphic
	to $\Real^n$ furnished with the action of affine linear transformations).
	In real affine linear spaces,
	it makes sense to speak of convex subsets.
	A (closed) \emph{convex polyhedron} in a real affine linear space refers to
	a closed convex subset whose sides (namely, maximal convex subsets on the boundary) are locally finite.
	A compact convex polyhedron is a \emph{polytope}.
	From the descriptions in Section \ref{Sec-picture},
	we know that the homology direction hull $\mathcal{D}(M_f,\phi_f)$
	agrees with Fried's set of homology directions,
	and is actually a polytope.
	(It is projective linearly isomorphic to 
	the cross-section of
	the dual cone $\mathcal{C}^*_{\mathtt{Th}}(\phi_f)$, or its reflection about the origin,
	with a generic affine hyperplane of $H_1(M_f;\Real)$.)
	The polytope description of the homology direction hull
	is reproved in Theorem \ref{D-description},
	for the purpose of further generalization,
	(see also Remark \ref{D-of-Fried}).
	\item 
	For any polytope in a real affine linear space, one may check that
	a semi-extreme subset is just a union of (various-dimensional) closed faces.
	Vertices are \emph{extreme} in the usual terminology of convex geometry.
	We introduce the notion of semi-extremity mostly 
	because our argument for Theorem \ref{criterion-enfeoffed}
	does not require a polytope description of the homology direction hull.
	\item
	For any covering mapping torus, every deck transformation must preserve the distinguished
	cohomology class, so it acts on the homology direction hull by an affine linear isomorphism.
	\end{enumerate}
	\end{remark}
	
	
	\begin{theorem}\label{criterion-enfeoffed}
		Let $f$ be a pseudo-Anosov automorphism of a connected closed orientable surface $S$.
		Let $\tilde{M}$ be any connected regular finite cover
		of the mapping torus $M_f$ with deck transformation group $\Gamma$.
		Declare $\tilde{M}$ as a covering mapping torus over $M_f$,
		(Convention \ref{covering-mapping-torus}).
		
		Suppose that the homology direction hull for $\tilde{M}$ contains 
		$-\chi(S)+1$
		or more	mutually disjoint $\Gamma$--invariant dominant semi-extreme subsets.
		Then the Mahler measure of 
		the multivariable Alexander polynomial $\Delta^\fabcover_{\tilde{M}}$ of $\tilde{M}$
		satisfies
		$$\mathbb{M}\left(\Delta^\fabcover_{\tilde{M}}\right)>1.$$
	\end{theorem}
	
	%
		%
	
	\begin{corollary}\label{criterion-enfeoffed-corollary}
		Under the same hypothesis as of Theorem \ref{criterion-enfeoffed},
		and in the same notations,
		there exist a connected finite cover $S'$ of the surface $S$
		and an automorphism lift $f'$ of $f$ to $S'$,
		and moreover,
		the homological spectral radius of $f'$ is strictly greater than $1$.
	\end{corollary}
	
	Corollary \ref{criterion-enfeoffed-corollary} follows from
	Theorem \ref{criterion-enfeoffed} and a characterization of Sun \cite[Theorem 1.2]{Sun-vhsr}.
	See Subsection \ref{Subsec-Mahler} for a brief review of Mahler measure.
	In fact, it can be shown that 
	provided with any covering mapping torus $\tilde{M}$ by Theorem \ref{criterion-enfeoffed},
	a covering mapping torus $M'$	for Corollary \ref{criterion-enfeoffed-corollary} is given by 
	a generic abelian finite cover of $\tilde{M}$.
	When $\tilde{M}$ is taken to be $M_f$, 
	Theorem \ref{criterion-enfeoffed} says essentially
	the same thing as Proposition \ref{criterion-cones}.
	
	Theorem \ref{criterion-enfeoffed} suggests 
	the following strategy for proving Conjecture \ref{conjecture-vhsr}:
	After some known reductions, it suffices to prove the essential case 
	for any pseudo-Anosov automorphism $f$ of a closed orientable surface $S$.
	Then we look for a regular finite cover $\tilde{M}$ of the mapping torus $M_f$ 
	which satisfies the hypothesis of Theorem \ref{criterion-enfeoffed}.
	Our plan consists of two steps:
	First we find some regular finite cover $M'$ of $M_f$
	so that the homology direction hull $\mathcal{D}(M',\phi')$
	possesses at least $-\chi(S)+1$ distinct orbits of vertices,
	which are not necessarily dominant.
	Then we find some further regular finite cover $\tilde{M}$,
	such that the preimage of the vertex orbits 
	are invariant dominant semi-extreme subsets	of $\mathcal{D}(\tilde{M},\tilde{\phi})$.
	Note that the covering projection induces an affine linear projection
	$\mathcal{D}(\tilde{M},\tilde{\phi})\to\mathcal{D}(M',\phi')$,
	so the preimage of any vertex orbit is a union of disjoint closed faces
	which is invariant under deck transformations.
	To carry out the plan, we must establish the polytope description
	of the homology direction hull,
	and moreover, 
	we must understand the topological, dynamical, or geometric	contents of closed faces.	
	In particular, 
	effective calculation for the multivariable Lefschetz zeta function 
	needs to be developed
	and behavior under finite coverings needs to be investigated.
	These materials are prepared by 
	Sections \ref{Sec-perspective_of_Markov_partitions}, \ref{Sec-cluster}, \ref{Sec-rcp}, 
	and \ref{Sec-covering}.
	Following the two-step plan (with a little technical modification),
	we produce the desired $\tilde{M}$
	by Theorem \ref{dominant-diversity} in Section \ref{Sec-diversity_of_dominant_virtual_faces}.
	
	\begin{remark}\label{dominant-enfeoffed-remark}
		In \cite{Hadari-order}, Hadari outlines a proof
		to show	that every pseudo-Anosov automorphism of 
		a compact orientable surface with boundary has virtually infinite order,
		and that the same holds for fully irreducible automorphisms
		of finitely generated free groups.
		Hadari's approach inspires
		Theorem \ref{criterion-enfeoffed} (compare \cite[Corollary 2.6]{Hadari-order})
		and our strategy for Conjecture \ref{conjecture-vhsr}.
		In particular, we point out that
		our dominance condition for semi-extreme subsets of
		the homology direction hull is essentially similar to
		Hadari's notion of enfeoffed vertices \cite[Definition 2.19]{Hadari-order},
		(compare Lemma \ref{dominant-enfeoffed}).
		
		Note that the automorphisms considered by Hadari are
		homologically trivial, possibly after passing to a finite iteration.
		So the first step of our two-step plan is a trivial fact there,
		(see \cite[Proposition 2.9]{Hadari-order}).
		However, for general pseudo-Anosov automorphisms
		the first step requires significantly more work, (compare Proposition \ref{vertices-diversity}).
		The second step in our setting also requires a different proof,
		since lower central series is no longer	suitable for 
		the induction argument,
		(compare Proposition \ref{vertex-dominant}).
	\end{remark}
	
	The rest of this section is devoted to the proof of Theorem \ref{criterion-enfeoffed}.
	Most part of the proof is purely algebraic about multivariable Laurent polynomials,
	using standard techniques.
	We refer the reader to Everest--Ward \cite{Everest--Ward} for the algebraic generalities.
	The only essential input from topology appears in Lemma \ref{dominant-Gamma},
	where indices are counted according to conjugacy classes of periodic points.
	
	\subsection{Mahler measure of multivariable Laurent polynomials}\label{Subsec-Mahler}
	For any complex Laurent polynomial in $n$ variables
	$q(z_1,\cdots,z_n)\in\Complex[z_1^{\pm1},\cdots,z_n^{\pm1}]$ other than zero,
	the (multiplicative) \emph{Mahler measure} of $q$ is defined by the expression
	\begin{equation}\label{Mahler}
	\mathbb{M}(q)=
	\exp\left(\frac{1}{(2\pi)^n}\,\int_0^{2\pi}\cdots\int_0^{2\pi}
	\log\left|q\left(e^{\imunit\theta_1},\cdots,e^{\imunit\theta_n}\right)\right|
	\ud \theta_1\cdots\ud \theta_n\right)
	\end{equation}
	The integral is known to converge absolutely so
	the value of $\mathbb{M}(q)$ is real and positive.
	By convention we set $\mathbb{M}(0)=0$.
	It is also known that $\mathbb{M}(q)\geq1$
	if $q$ is nonzero and over $\Integral$.
	For a complex polynomial $q(t)=Dt^n\prod_{i=1}^l(t-b_i)$
	in a single variable $t$, where $b_i\neq0$,
	the Mahler measure can be computed explicitly by the Jensen formula:
	\begin{equation}\label{Jensen}
	\mathbb{M}(q)=|D|\cdot\prod_{i=0}^l\max(1,|b_i|).
	\end{equation}
	See \cite[Chapters 1 and 3]{Everest--Ward}.	
	
	Given a finitely generated free abelian group, 
	it makes sense to speak of the Mahler measure 
	for multivariable Laurent polynomials 
	over the complex group algebra.
	This is obvious by treating 
	a basis of the free abelian group
	as extra independent variables.
	The resulting value of the Mahler measure
	is invariant under change of bases
	and multiplication by monomials with unit coefficients.
	In particular, it is valid to speak of the Mahler measure
	for multivariable Alexander polynomials,
	such as our $\Delta^\fabcover_{M_f}$.
	We give a characterization 
	for a polynomial in $t$ over 
	$\Integral[\Integral^n]$ with the constant term $1$ to have Mahler measure $1$:
	
	\begin{lemma}\label{algebraic-nontrivial-Mahler}
		Denote by $\Lambda$ the complex group algebra	$\Complex[\Integral^n]$.
		For any polynomial $q(t)$ over $\Lambda$ with the constant term $1$,
		denote by $(L_m(q)\in\Lambda)_{m\in\Natural}$
		the unique sequence	determined by the following equation
		in the formal series completion $\Lambda[[t]]$:
		\begin{equation}\label{L-m-q}
		\log q(t)=\sum_{m\in\Natural}\frac{L_m(q)}{m}\cdot t^m,
		\end{equation}
		where $\log(1-z)=\sum_{k=1}^\infty \frac{z^k}{-k}$.		
		
		Suppose that $q(t)$ is a polynomial over the subalgebra $\Integral[\Integral^n]$ 
		of $\Lambda$ with the constant term $1$.
		If the Mahler measure satisfies $\mathbb{M}(q)=1$,
		then for all $m\in \Natural$ the following inequality holds:
		$$\|L_m(q)\|_1\leq \mathrm{deg}_{t}(q).$$
		Here $\mathrm{deg}_{t}$ stands for the degree of the polynomial 
		with respect to $t$,
		and $\|\cdot\|_1$ stands for the standard $\ell^1$--norm on $\Lambda$,
		namely, $\|\sum_{x}a_xx\|_1=\sum_x|a_x|$ where the summations are taken 
		for all $x\in\Integral^n$.
	\end{lemma}
	
	\begin{proof}
		We identify $\Lambda$ with $\Complex[x_1^{\pm1},\cdots,x_n^{\pm1}]$.
		By Kronecker's theorem \cite[Theorem 3.10]{Everest--Ward}, any multivariable integral Laurent polynomial 
		$p\in\Integral[x_1^{\pm1},\cdots,x_n^{\pm1},t^{\pm1}]$
		with Mahler measure $\mathbb{M}(p)=1$ is a product of monomials
		and cyclotomic polynomials evaluated on monomials.
		If $q(t)$ is a polynomial in $t$ over $\Integral[x_1^{\pm1},\cdots,x_n^{\pm1}]$
		with the constant term $1$, and with $\mathbb{M}(q)=1$,
		we conclude that $q(t)$ must factorize in $\Lambda[t]$
		as a product
		$$q(t)=q_1(t)\times\cdots\times q_k(t)$$
		where each $q_j\in \Lambda[t]$ takes the form $1-\mu xt^l$ 
		for some monomial $x=x_1^{l_1}\cdots x_n^{l_n}$, some power $l\in\Natural$, 
		and some root of unity $\mu\in\Complex$.
		
		By the definition of $L_m$, we have the formula
		$$L_m(q)=L_m(q_1)+\cdots+L_m(q_k)$$
		for all $m\in\Natural$.
		Moreover, for each factor $q_j(t)=1-\mu xt^l$ and for all $m\in \Natural$,
		we evaluate explicitly that $L_m(1-\mu xt^l)$ equals $l\mu^{m/l} x^{m/l}$ if $l$ divides $m$,
		or $0$ otherwise.
		This yields the estimate
		$$\|L_m(q_j)\|_1\leq l= \mathrm{deg}_t(q_j),$$
		for each factor $q_j$ and for all $m\in\Natural$.
		Therefore, for all $m\in\Natural$,
		we have
		$$\|L_m(q)\|_1\leq\sum_{j=1}^k\|L_m(q_j)\|_1\leq\sum_{j=1}^k\mathrm{deg}_t(q_j)=\mathrm{deg}_t(q),$$
		as asserted.
		(Compare the proof of Proposition \ref{criterion-cones}.)
	\end{proof}

	\begin{lemma}\label{algebraic-nontrivial-polynomial}
		Adopting the same notations as of Lemma \ref{algebraic-nontrivial-Mahler},
		suppose that $q(t)$ is a polynomial over $\Lambda$ with the constant term $1$.
		If $q(t)$ is not the constant polynomial $1$,
		then there exists some positive integer $m_q\in\Natural$
		such that the inequality 
		$$L_m(q)\neq0$$
		holds in $\Lambda$ for all positive integral multiples $m$ of $m_q$.
	\end{lemma}
	
	\begin{proof}
		Suppose 
		$q(t)=1+c_1t+\cdots+c_{d-1}t^{d-1}+c_dt^d$,
		where $c_j\in\Lambda$, and $c_d\neq0$.
		We assume $q(t)\neq1$ so the degree $d$ is nonzero.
		We obtain a recurrence relation in $\Lambda$:
		\begin{equation}\label{recurrence-L}
		L_{m+d}(q)+c_1L_{m+d-1}(q)+\cdots+c_{d-1}L_{m+1}(q)+c_dL_m(q)=0
		\end{equation}
		for all $m\in\Natural$. 
		It also holds for $m=0$ if we define $L_0(q)=d$ by convention.
		This can be seen by a standard formal splitting argument,
		that is, working with $q(t)=(1-\lambda_1t)\times\cdots\times(1-\lambda_dt)$	in a splitting field over $\Lambda$.
		For all $m\in\Natural$, we observe $L_m(q)=-(\lambda_1^m+\cdots+\lambda_d^m)$.
		The left-hand side of (\ref{recurrence-L}) equals
		the sum of $-\lambda_i^{m+d}q(1/\lambda_i)$ for $i=1,\cdots,d$,
		which is $0$ because $q(1/\lambda_i)$ are all $0$.
		%
		%
		
		We also observe that $L_m(q)$ can be expressed as 
		a polynomial function of $c_1,\cdots,c_d$ with coefficients in $\Integral$.
		To be precise, 
		denote by $\sigma_i$ the $i$--th elementary symmetric polynomial in $d$ indeterminants
		and $p_i$ the $i$--th power sum polynomial in (the same) $d$ indeterminants.
		As a classical fact in polynomial algebra,
		there are universal polynomials $s_m$ over $\Integral$ in $m$ indeterminants,
		such that the Newton--Girard identity	$s_m(\sigma_1,\cdots,\sigma_m)=p_m$	holds
		regardless of $d$ or the indeterminants in $s_m$ and $p_m$.
		(See \cite[\S 16]{Milnor--Stasheff}, and 
		Problem 16-A thereof for an explicit formula for $s_m$.)
		Then we work out explicitly
		$L_m(q)=-p_m(\lambda_1,\cdots,\lambda_d)=
		-s_m(\sigma_1(\lambda_1,\cdots,\lambda_d),\cdots,
		\sigma_m(\lambda_1,\cdots,\lambda_d))
		=-s_m(-c_1,c_2,\cdots,(-1)^mc_m)$, 
		setting $c_j=0$ for all $j>d$.
		
		Denote by $A$ the $\Integral$--subalgebra of $\Lambda$	
		generated by the coefficients $c_1,\cdots,c_d\in \Lambda$. 
		Then for all $m\in\Natural$,
		$L_m(q)$ are indeed elements of $A$ and 
		(\ref{recurrence-L}) holds in $A$.
		
		As $A$ is an integral domain algebraically finitely generated over $\Integral$,
		there exists a maximal ideal $\mathfrak{m}$ of $A$ so that $A/\mathfrak{m}$ 
		is a field of some positive characteristic $p$, and moreover,
		we may require 
		$$dc_d\not\equiv 0\bmod\mathfrak{m}.$$
		In fact, for the special case when $A$ 
		is contained in $\Integral[\Integral^n]\cong\Integral[t_1^{\pm1},\cdots,t_n^{\pm1}]$,
		one may construct the maximal ideal $\mathfrak{m}$ of $A$ as follows:
		First take a ring homomorphism
		$\Integral[t_1^{\pm1},\cdots,t_n^{\pm1}]\to \Integral[\tau_1^{-1},\cdots,\tau_n^{-1}]$,
		such that each $t_i$ is sent to some nonzero $\tau_i\in\Integral$.
		Choose $\tau_i$ suitably so that $c_d(t_1,\cdots,t_n)$ has nonzero image $c_d(\tau_1,\cdots,\tau_n)$.
		Next take a suitable congruence quotient 
		$\Integral[\tau_1^{-1},\cdots,\tau_n^{-1}]\to \Integral/p\Integral$,
		requiring the prime $p\in\Natural$ not to divide any $\tau_i$, 
		or $d$, or the nominator of $c_d(\tau_1,\cdots,\tau_d)$.
		Then by composition we obtain a quotient $A\to\Integral/p\Integral$
		and the maximal ideal $\mathfrak{m}$ is given by the kernel.
		The general case is not necessary for proving our main theorem
		but it is a well-known fact in commutative algebra,
		(see \cite[Chapter 7, Theorem 7.6.6]{Ratcliffe-book}).
		Note that $A/\mathfrak{m}$ is finite 
		because it is a finitely generated field over $\Integral/p\Integral$.
		
		By the recurrence relation (\ref{recurrence-L}),
		the residual sequence $(L_m(q)\bmod\mathfrak{m})_{m\in\Natural}$ 
		must be eventually periodic,
		since every $d$ consecutive terms determine all their succeeding terms
		and since there are at most $|A/\mathfrak{m}|^d$ possible patterns for them.
		Every $d$ consecutive terms also determine all their preceding terms,
		by (\ref{recurrence-L}) and	by the required invertibility of $c_d\bmod\mathfrak{m}$ in $A/\mathfrak{m}$,
		so the residual sequence must be periodic starting from the term $L_0(q)\bmod\mathfrak{m}$.
		Take $m_q\in\Natural$ to be any positive period of the residual sequence.
		Then for all $m\in\Natural$ divisible by $m_q$, we have
		$$L_m(q)\equiv L_0(q)=d\bmod\mathfrak{m}.$$
		By the required invertibility of $d\bmod\mathfrak{m}$ in $A/\mathfrak{m}$,
		we conclude that $L_m(q)\neq0$ holds in $A$, and hence in $\Lambda$,
		for all positive integral multiples $m$ of $m_q$.
	\end{proof}
	
	\subsection{Dominant semi-extreme subsets}\label{Subsec-contribution}
	Let $M_f$ be the mapping torus of a pseudo-Anosov automorphism of
	a connected closed orientable surface $S$ of genus at least $2$.
	Let $\tilde{M}$ be a connected regular finite cover of $M_f$
	with the deck transformation group denoted by $\Gamma$.
	Declare $\tilde{M}$ as a covering mapping torus $M_{\tilde{f}}$ over $M_f$,
	(Convention \ref{covering-mapping-torus}).
	Denote by $\tilde{S}$ the distinguished fiber of $M_{\tilde{f}}$.
	Denote by $\phi_{\tilde{f}}\in H^1(M_{\tilde{f}};\Integral)$
	the distinguished cohomology class of $M_{\tilde{f}}$,
	which has divisibility $b_0(\tilde{S})$.
		
	Choose a homology class $t\in H_1(M_{\tilde{f}};\Integral)_{\mathtt{free}}$ 
	with $\phi_{\tilde{f}}(t)=b_0(\tilde{S})$,
	and choose a basis 
	for the $\Integral$--submodule of $H_1(M_{\tilde{f}};\Integral)_{\mathtt{free}}$ annihilated by $\phi_{\tilde{f}}$.
	We identify henceforth 
	the complex group algebra $\Complex H_1(M_{\tilde{f}};\Integral)_{\mathtt{free}}$ 
	with a Laurent polynomial ring $\Lambda[t,t^{-1}]$ 
	over another multivariable Laurent polynomial ring
	$\Lambda=\Complex[\Integral^n]$, where $n=b_1(M_{\tilde{f}})-1$.		
		
	\begin{lemma}\label{MAP-formula-covering}
		Suppose $b_1(M_{\tilde{f}})>1$.
		Then the multivariable Lefschetz zeta function 
		$\zeta^\fabcover_{\tilde{f}}$ is a polynomial in $t$ 
		over the subalgebra $\Integral[\Integral^n]$ of $\Lambda$
		with the constant term $1$ and of degree $-\chi(\tilde{S})/b_0(\tilde{S})$.
		In fact, it is a representative of multivariable Alexander polynomial
		$\Delta^\fabcover_{M_{\tilde{f}}}$.
	\end{lemma}
	
	\begin{proof}		
		Choose a connected component $\tilde{S}_0$ of $\tilde{S}$.
		We observe $\chi(\tilde{S})=b_0(\tilde{S})\cdot \chi(\tilde{S}_0)$.
		For $b=b_0(\tilde{S})$,
		denote by $\tilde{f}_0\colon \tilde{S}_0\to \tilde{S}_0$
		the restriction of the iterated automorphism $\tilde{f}^b$ to $\tilde{S}_0$.
		We observe that the mapping torus of $\tilde{f}_0$
		gives rise to the same mapping torus as $\tilde{M}$,
		only with the suspension flow running $b$ times as fast as before.
		This also means $\phi_{\tilde{f}}=b\cdot \phi_{\tilde{f}_0}$ 
		in $H^1(\tilde{M};\Integral)$.
		Identify these mapping tori topologically.
		It follows that $\zeta^\fabcover_{\tilde{f}_0}$	is the same as $\zeta^\fabcover_{\tilde{f}}$.	
		In fact, $\mathrm{Per}_{\tilde{m}}(\tilde{f})$ is nonemepty only if $\tilde{m}$ equals $mb$ for some $m\in\Natural$,
		and in that case, $\sum_{p\in \mathrm{Per}_{\tilde{m}}(\tilde{f})}\mathrm{ind}_{\tilde{m}}(\tilde{f};p)\cdot[\gamma_{\tilde{m}}(\tilde{f};p)]$
		equals $b$ times $\sum_{p\in \mathrm{Per}_{m}(\tilde{f}_0)}\mathrm{ind}_m(\tilde{f}_0;p)\cdot[\gamma_{m}(\tilde{f}_0;p)]$.
		Then the asserted properties follow from Theorem \ref{MAP-formula} and (\ref{deg-mAp}), applied to $\tilde{f}_0$.
		To be precise,
		we see that the $\phi_{\tilde{f}_0}$--degree of $\zeta^\fabcover_{\tilde{f}_0}$ equals 
		$-\chi(\tilde{S}_0)$.
		On the other hand, $\zeta^\fabcover_{\tilde{f}_0}$ is a polynomial in $t$ over $\Integral[\Integral^n]$
		where the $\phi_{\tilde{f}_0}$--degree of $t$ equals $1$.
		So $\zeta^\fabcover_{\tilde{f}_0}$, or equally $\zeta^\fabcover_{\tilde{f}}$,	
		is a polynomial in $t$ over $\Integral[\Integral^n]$ of degree $-\chi(\tilde{S}_0)$,
		which equals $-\chi(\tilde{S})/b_0(\tilde{S})$.		
	\end{proof}
		
	By comparing the expression (\ref{mLzeta}) for $\tilde{f}$ and
	the formal series expansion (\ref{L-m-q}) for 
	$\log\zeta^\fabcover_{\tilde{f}}$,
	we obtain 
	\begin{equation*}
	\sum_{m\in\Natural} \frac{L_m\left(\zeta^\fabcover_{\tilde{f}}\right)}{m}\cdot t^m=
	\sum_{\tilde{m}\in\Natural}
	\left(\sum_{p\in\mathrm{Per}_{\tilde{m}}\left({\tilde{f}}\right)}
	\frac{\mathrm{ind}_{\tilde{m}}({\tilde{f}};p)}{\tilde{m}}
	\cdot[\gamma_{\tilde{m}}({\tilde{f}};p)]\right).
	\end{equation*}
	The homogeneous part of $\phi_{\tilde{f}}$--degree $m\cdot b_0(\tilde{S})$
	yields the following explicit relations for all $m\in\Natural$:
	\begin{equation}\label{L-m-ind-m}
	L_m\left(\zeta^\fabcover_{\tilde{f}}\right)\,t^m=
	\frac{1}{b_0(\tilde{S})}\cdot
	\left(\sum_{p\in\mathrm{Per}_{\tilde{m}}\left({\tilde{f}}\right)}
	\mathrm{ind}_{\tilde{m}}({\tilde{f}};p)
	\cdot[\gamma_{\tilde{m}}({\tilde{f}};p)]\right),
	\end{equation}
	where $\tilde{m}$ equals $m\cdot b_0(\tilde{S})$.
	For any semi-extreme subset $E$ of the homology direction hull 
	$\mathcal{D}(M_{\tilde{f}},\phi_{\tilde{f}})$,
	we introduce the $E$--part $L_m(\zeta^\fabcover_{\tilde{f}};E)$
	for $L_m(\zeta^\fabcover_{\tilde{f}})$ as determined by the relation
	\begin{equation}\label{L-m-E}
	L_m\left(\zeta^\fabcover_{\tilde{f}};E\right)\,t^m=
	\frac{1}{b_0(\tilde{S})}\cdot
	\left(\sum_{p\in\mathrm{Per}_{\tilde{m}}\left({\tilde{f}};E\right)}
	\mathrm{ind}_{\tilde{m}}({\tilde{f}};p)
	\cdot[\gamma_{\tilde{m}}({\tilde{f}};p)]\right).
	\end{equation}
	Here $\tilde{m}$ stands for $m\cdot b_0(\tilde{S})$, and
	$\mathrm{Per}_{\tilde{m}}({\tilde{f}};E)$ stands for the set of the $\tilde{m}$--periodic points
	of $\tilde{f}$ 
	with the property that $[\gamma_{\tilde{m}}({\tilde{f}};p)]$ lie in the linear cone over $E$.

	\begin{lemma}\label{dominant-enfeoffed}
	Suppose $b_1(M_{\tilde{f}})>1$.
	Then for any dominant semi-extreme subset $E$ of $\mathcal{D}(M_{\tilde{f}},\phi_{\tilde{f}})$,
	there exists some positive integer $m_E\in\Natural$ such that the inequality
	$$L_m\left(\zeta^\fabcover_{\tilde{f}};E\right)\neq0$$
	holds for all positive integral multiples $m$ of $m_E$.
	\end{lemma}
	
	\begin{proof}
		We claim that 
		there exists a dominant semi-extreme subset $K$ of
		$\mathcal{D}(M_{\tilde{f}},\phi_{\tilde{f}})$
		which is contained in $E$
		and which is witnessed by some linear function $k$ 
		on $\mathbf{A}(M_{\tilde{f}},\phi_{\tilde{f}})$.
		The latter condition means that $K$ is the set of maximum points
		of $k$ restricted to $\mathcal{D}(M_{\tilde{f}},\phi_{\tilde{f}})$.
		If one assumed the polytope description of $\mathcal{D}(M_{\tilde{f}},\phi_{\tilde{f}})$
		and the characterization of semi-extreme subset in that case
		(Remark \ref{remark_homology_directions}),
		one would know that $E$ is a union of closed faces,
		and must contain at least one dominant closed face $K$,
		so the claim would be clear.
		To prove the claim by definition, 
		take a witnessing convex function for the semi-extreme subset $E$,
		namely, a convex function $h\colon \mathbf{A}(M_{\tilde{f}},\phi_{\tilde{f}})\to \Real$
		that is maximized over $\mathcal{D}(M_{\tilde{f}},\phi_{\tilde{f}})$
		exactly on $E$, (Definition \ref{homology_directions}).
		As $E$ is dominant, 
		there exists some witnessing direction $l_0\in E$,
		namely, a projective point $\Real u\in\mathbf{P}(H_1(M_{\tilde{f}};\Real))$
		given by some $u\in H_1(M_{\tilde{f}};\Integral)_{\mathtt{free}}$
		with the properties $\mathrm{coef}(\zeta^\fabcover_{\tilde{f}};u)\neq0$ and $\phi_{\tilde{f}}(u)>0$,
		(Definition \ref{homology_directions}).
		Take a linear function $k\colon \mathbf{A}(M_{\tilde{f}},\phi_{\tilde{f}})\to \Real$
		with the properties 
		$k(l_0)=h(l_0)$ and $k(l)\leq h(l)$ for all $l\in\mathcal{D}(M_{\tilde{f}},\phi_{\tilde{f}})$.
		It is clear that the maximum-point set for $k$ is a subset $K$ of $E$ which contains $l_0$.
		So $K$ is as claimed.
		
		Take any $K$ as claimed above.
		The claimed witnessing linear function $k$ induces a real-valued new grading for $\Lambda[t]$, namely,
		$\mathrm{deg}_k(xt^m)=k(\Real xt^m)$
		for all $x\in\Integral^n$ and $m\in\Natural$.
		The $K$--part $\zeta^\fabcover_{\tilde{f}}[K]$ 
		is the leading part	for $\zeta^\fabcover_{\tilde{f}}$
		with respect to the $k$--grading, 
		and it is a polynomial in $\Lambda[t]$ other than the constant polynomial $1$.
		It follows that 
		$L_m(\zeta^\fabcover_{\tilde{f}}[K])=L_m(\zeta^\fabcover_{\tilde{f}};K)$
		holds for all $m\in\Natural$, (see (\ref{L-m-q}) and (\ref{L-m-E})).
		By Lemma \ref{algebraic-nontrivial-polynomial},
		there exists some $m_E\in\Natural$, 
		and $L_m(\zeta^\fabcover_{\tilde{f}}[K])\neq0$ holds 
		for all positive integral multiples $m$ of $m_E$.
		For all such $m$, we obtain the inequality
		$L_m(\zeta^\fabcover_{\tilde{f}};K)\neq0$,
		and hence the asserted inequality
		$L_m(\zeta^\fabcover_{\tilde{f}};E)\neq0$.
	\end{proof}	
		
	\begin{remark}\label{o_m_E}
	Under the assumptions of Lemma \ref{dominant-enfeoffed},
	$\zeta^\fabcover_{\tilde{f}}[E]$ is also a polynomial in $\Lambda[t]$ with constant term $1$.
	We point out the following relation:
	$$L_m\left(\zeta^\fabcover_{\tilde{f}}[E]\right)=L_m\left(\zeta^\fabcover_{\tilde{f}};E\right)+o_m(E).$$
	Here $o_m(E)\in\Lambda$ stands for some Laurent polynomial of the form $\sum_x a_xx$,
	where $a_x\neq0$ holds only if $xt^m$ lies in the linear cone over $\mathcal{D}(M_{\tilde{f}},\phi_{\tilde{f}})\setminus E$.
	The correction term $o_m(E)$ 
	is uniquely determined by the sequence 
	$(L_m(\zeta^\fabcover_{\tilde{f}};E))_{m\in\Natural}$,
	and it disappears when $E$ is convex and semi-extreme.
	\end{remark}

\begin{lemma}\label{dominant-Gamma}
		Suppose $b_1(M_{\tilde{f}})>1$.
		Then for any $\Gamma$--invariant dominant semi-extreme subset $E$ of $\mathcal{D}(M_{\tilde{f}},\phi_{\tilde{f}})$,
		there exists some positive integer $m_E\in\Natural$ such that the estimate
		$$\left\|L_m\left(\zeta^\fabcover_{\tilde{f}};E\right)\right\|_1\,\geq\,\frac{1}{b_0(\tilde{S})}\cdot|\Gamma|$$
		holds for all positive integral multiples $m$ of $m_E$.
		Here $\|\cdot\|_1$ stands for the standard $\ell^1$--norm on $\Lambda$,
		and $|\Gamma|$ stands for the cardinality of the deck transformation group $\Gamma$.
	\end{lemma}
	
	\begin{proof}
		As $E$ is semi-extreme and dominant,
		we obtain some $m_E\in\Natural$ as asserted by Lemma \ref{dominant-enfeoffed}.
		In particular, for every ${\tilde{m}}\in\Natural$ divisible by $m_E\cdot b_0(\tilde{S})$,
		there exists some periodic homology direction $l_{\tilde{m}}\in E$
		such that the inequality 
		$\sum_{p\in\mathrm{Per}_{\tilde{m}}({\tilde{f}};\{l_{\tilde{m}}\})}\,\mathrm{ind}_{\tilde{m}}({\tilde{f}};p)\neq0$
		holds.
		Here
		$\mathrm{Per}_{\tilde{m}}({\tilde{f}};\{l_{\tilde{m}}\})$ stands for the set of ${\tilde{m}}$--periodic points
		for $\tilde{f}$ with $[\gamma_{\tilde{m}}(\tilde{f};p)]\in l_{\tilde{m}}$,
		(see (\ref{L-m-E}), (\ref{gamma-m}), and (\ref{ind-m})).
		Note that for all $p\in \mathrm{Per}_{\tilde{m}}({\tilde{f}};\{l_{\tilde{m}}\})$,
		the homology class $[\gamma_{\tilde{m}}(\tilde{f};p)]$ does not vary.
		Let $\tilde{S}_0$ be a component of $\tilde{S}$.
		Since every periodic trajectory intersects the components of $\tilde{S}$ in cyclical order,
		we have the relation
		$$\frac{1}{b_0(\tilde{S})}\cdot\left(\sum_{p\in\mathrm{Per}_{\tilde{m}}\left({\tilde{f}};\,\{l_{\tilde{m}}\}\right)}
		\,\mathrm{ind}_{\tilde{m}}({\tilde{f}};p)\right)
		\,=\,\sum_{p\in\mathrm{Per}_{\tilde{m}}\left({\tilde{f}};\,\{l_{\tilde{m}}\}\right)\cap\tilde{S}_0}\,\mathrm{ind}_{\tilde{m}}({\tilde{f}};p).$$
		Denote by $\Gamma_0$ the subgroup of $\Gamma$ that preserves $\tilde{S}_0$,
		which is normal of index $b_0(\tilde{S})$.
		We observe the following estimate:
		$$
		\left|\sum_{p\in\mathrm{Per}_{\tilde{m}}\left({\tilde{f}};\,\{l_{\tilde{m}}\}\right)\cap\tilde{S}_0}\,\mathrm{ind}_{\tilde{m}}({\tilde{f}};p)\right|
		\geq|\mathrm{Stab}_{\Gamma_0}(\{l_{\tilde{m}}\})|.
		$$
		This follows from the fact that 
		$\sum_{p\in\mathrm{Per}_{\tilde{m}}({\tilde{f}};\{l_{\tilde{m}}\})\cap\tilde{S}_0}
		\,\mathrm{ind}_{\tilde{m}}({\tilde{f}};p)$
		is a nonzero integer divisible by $|\mathrm{Stab}_{\Gamma_0}(\{l_{\tilde{m}}\})|$,
		because $\mathrm{ind}_{\tilde{m}}(\tilde{f};-)$ is constant 
		over any $\mathrm{Stab}_{\Gamma}(\{l_{\tilde{m}}\})$--orbit of ${\tilde{m}}$--periodic points
		in $\mathrm{Per}_{\tilde{m}}({\tilde{f}};\{l_{\tilde{m}}\})$.
		As $E$ is $\Gamma$--invariant,
		for all positive integral multiples $m$ of $m_E$,
		the $\ell^1$--norm of $L_m(\zeta^\fabcover_{M_{\tilde{f}}};E)$
		can be estimated by the following
		\begin{eqnarray*}
		\left\|L_m\left(\zeta^\fabcover_{M_{\tilde{f}}};E\right)\right\|_1
		&\geq&\left|\mathrm{Orb}_{\Gamma}(\{l_{\tilde{m}}\})\right|\cdot
		\left|\sum_{p\in\mathrm{Per}_{\tilde{m}}\left({\tilde{f}};\{l_{\tilde{m}}\}\right)}
		\frac{\mathrm{ind}_{\tilde{m}}({\tilde{f}};p)}{b_0(\tilde{S})}\right|\\
		&\geq&\left|\mathrm{Orb}_{\Gamma_0}(\{l_{\tilde{m}}\})\right|\cdot\left|\mathrm{Stab}_{\Gamma_0}(\{l_{\tilde{m}}\})\right|
		=|\Gamma_0|={|\Gamma|}/{b_0(\tilde{S})},
		\end{eqnarray*}
		where $\tilde{m}$ stands for $m\cdot b_0(\tilde{S})$.
		Here $\mathrm{Orb}_{\Gamma}(\{l_{\tilde{m}}\})$ stands for the $\Gamma$--orbit of $l_{\tilde{m}}$
		in $E$, which coincides with the $\Gamma_0$--orbit.
		This establishes the asserted estimate.
	\end{proof}

	\subsection{Proof of Theorem \ref{criterion-enfeoffed}}\label{Subsec-proof-criterion-enfeoffed}
	We continue to adopt the assumptions at the beginning of Subsection \ref{Subsec-contribution}.
	Suppose that there exist sufficiently many
	mutually disjoint $\Gamma$--invariant dominant semi-extreme subsets,
	say $E_1,E_2,\cdots,E_s\subset\mathcal{D}(M_{\tilde{f}},\phi_{\tilde{f}})$,	
	for some $s>-\chi(S)$, according to the assumption of Theorem \ref{criterion-enfeoffed}.
	In particular, observe $b_1(M_{\tilde{f}})>1$.
	
	For each $E_i$, let $m_i\in\Natural$ be the positive integer as asserted by Lemma \ref{dominant-Gamma}.
	Then for any $m\in\Natural$ divisible by the least common multiple of $m_1,m_2,\cdots,m_s$,
	we estimate
	$$
	\left\|L_m\left(\zeta^\fabcover_{{\tilde{f}}}\right)\right\|_1
	\,\geq\,
	\sum_{i=1}^s\left\|L_m\left(\zeta^\fabcover_{{\tilde{f}}};E_i\right)\right\|_1
	\,\geq\,
	s\cdot\frac{|\Gamma|}{b_0(\tilde{S})}
	\,>\,-\chi(S)\cdot\frac{|\Gamma|}{b_0(\tilde{S})}
	\,=\,-\frac{\chi(\tilde{S})}{b_0(\tilde{S})}.
	$$
	This yields the strict comparison
	$$\left\|L_m\left(\zeta^\fabcover_{{\tilde{f}}}\right)\right\|_1\,>\,
	\mathrm{deg}_t\left(\zeta^\fabcover_{\tilde{f}}\right)$$
	for infinitely many $m\in\Natural$, by Lemma \ref{MAP-formula-covering}.
	Therefore, we obtain the asserted Mahler measure estimate
	for the multivariable Alexander polynomial $\Delta^\fabcover_{M_{\tilde{f}}}$
	of the covering mapping torus $M_{\tilde{f}}$:
	$$\mathbb{M}\left(\Delta^\fabcover_{M_{\tilde{f}}}\right)=\mathbb{M}\left(\zeta^\fabcover_{{\tilde{f}}}\right)>1,$$
	by Lemmas \ref{algebraic-nontrivial-Mahler} and \ref{MAP-formula-covering}.
	This completes the proof of Theorem \ref{criterion-enfeoffed}.

\section{Mapping tori from a perspective of Markov partitions}\label{Sec-perspective_of_Markov_partitions}
	In this section, we study mapping tori for pseudo-Anosov automorphisms of closed orientable surfaces
	using an auxiliary Markov partition.
	In particular, we recover Fried's description of the homology direction hull
	as a polytope (Theorem \ref{D-description} and Remark \ref{D-of-Fried}).
	The approach of this section is 
	extended to the cluster version and the covering version in subsequent sections.
	We remind the reader that the real task of this section is 
	to re-organize relevant materials, and to expose necessary details,
	so as to produce a basic piece of text 
	to be updated in subsequent sections.
	Related works and original ideas in the literature are 
	collected and discussed in Section \ref{Subsec-notes}.
	
	Throughout this section, 
	we assume that $S$ is a connected closed surface and 
	that $f\colon S\to S$	is a pseudo-Anosov automorphism,
	(see Section \ref{Sec-preliminary}).
	
	%
	%
	%
%
	
	\subsection{Markov patitions for pseudo-Anosov automorphisms}
	We recall some terminology from \cite[Expos\'e 10]{FLP}. 
	With respect to the pair of invariant foliations $(\mathscr{F}^{\mathtt{s}},\mathscr{F}^{\mathtt{u}})$
	of the pseudo-Anosov $f$,
	a (good) \emph{birectangle} refers to an embedded compact rectangular region
	whose product structure is inherited from the invariant foliations.
	In other words, a birectangle is the image  $R\subset S$
	of an inclusion $[0,1]\times[0,1]\to S$ such that
	for any $x\in[0,1]$ the vertical segment
	$\{x\}\times [0,1]$ is contained in 
	a finite union of leaves or singular points of $\mathscr{F}^{\mathtt{s}}$, 
	and that for any $y\in[0,1]$ the horizontal segment $[0,1]\times \{y\}$ 
	is compatible with $\mathscr{F}^{\mathtt{u}}$ in the same fashion.
	In this paper, we keep the above arrangement convention 
	so that birectangles are 
	vertically stretching and horizontally shrinking under the pseudo-Anosov automorphism $f$.
	Note that any vertical or horizontal segment must be contained in 
	a unique leaf, unless it is a side of the birectangle.
	Moreover, any side of the birectangle contains
	at most one singular point of the invariant foliations.	
	This is because for invariant foliations of pseudo-Anosov automorphisms,
	there are no leaves connecting pairs of singular points \cite[Expos\'e 9, Lemma 9.17]{FLP}.
		
	A \emph{Markov partition} of $S$ with respect to $f$
	is a finite partition of $S$ into birectangles with respect to $f$.
	Moreover,
	every birectangle of the partition is the vertical juxtaposition
	of finitely many horizontal blocks,
	which are mapped under $f$ onto vertical blocks
	in mutually distinct birectangles of the partition.
	To be more precise,
	a Markov partition is a finite collection of birectangles
	\begin{equation}\label{MarkovPartition}
	\mathcal{R}=\left\{R_1,\cdots,R_k\right\},
	\end{equation}
	with all the following properties:
	\begin{itemize}
	\item For any $p\in S$, there is some $R_i$ that contains $p$,
	which is unique if $p$ lies in $\interior(R_i)$.
	\item For any $p\in \interior(R_i)\cap f^{-1}(\interior(R_j))$, we have
	\begin{equation*}
	f\left(\mathscr{F}^{\mathtt{u}}(p,R_i)\right)\subset \mathscr{F}^{\mathtt{u}}(f(p),R_j)
	\textrm{ and }
	f^{-1}\left(\mathscr{F}^{\mathtt{s}}(f(p),R_j)\right)\subset \mathscr{F}^{\mathtt{s}}(p,R_i).
	\end{equation*}
	\item For any $p\in \interior(R_i)\cap f^{-1}(\interior(R_j))$, we have
	\begin{equation*}
	f\left(\mathscr{F}^{\mathtt{s}}(p,R_i)\right)\cap R_j= \mathscr{F}^{\mathtt{s}}(f(p),R_j)
	\textrm{ and }
	f^{-1}\left(\mathscr{F}^{\mathtt{u}}(f(p),R_j)\right)\cap R_i= \mathscr{F}^{\mathtt{u}}(p,R_i).
	\end{equation*}
	\end{itemize}
	Here the notation $\mathscr{F}^{\mathtt{s}}(p,R_i)$ 
	stands for the component containing $p$
	of the intersection between $R_i$ and the leaf of $\mathscr{F}^{\mathtt{s}}$ through $p$,
	and other similar notations are understood likewise.
	(The assumptions agree with \cite[Expos\'e 10]{FLP}.)
	Note that any $\interior(R_i)\cap f^{-1}(\interior(R_j))$ is necessarily connected.
	When it is nonempty, we call $R_i\cap f^{-1}(R_j)$ a \emph{horizontal block} in $R_i$,
	and $f(R_i)\cap R_j$ a \emph{vertical block} in $R_j$.
	
	It is known that 	
	every pseudo-Anosov automorphism admits a Markov partition \cite[Expos\'e 10, Proposition 10.17]{FLP}.

	\subsection{Objects derived from a Markov partition}
	Fix a Markov partition $\mathcal{R}$ of any given connected closed orientable surface $S$
	with respect to a given pseudo-Anosov automorphism $f$.
	We derive combinatorial objects at three different levels:
	The flow-box complex $X_{f,\mathcal{R}}$ is a polyhedral complex
	obtained by gluing up abstract flow boxes, which remembers all the information
	of the triple $(S,f,\mathcal{R})$; 
	the transition graph $\tdigraph_{f,\mathcal{R}}$
	is the directed graph presentation of the induced symbolic dynamical system, as usual, 
	which forgets about the topology of $S$;
	the space of projective currents $\mathcal{P}_{f,\mathcal{R}}$
	is an affine linear polytope whose rational points are 
	all the projectivized abelianized directed multi-cycles of the transition graph.
	Below we construct these objects concretely, 
	and obtain descriptions about their key features.
		
	\subsubsection{The flow-box complex}\label{Subsubsec-flow-box_complex}
	%
	\begin{definition}\label{abstract_flow-box_complex}
	Suppose that $\mathcal{R}$ is a Markov partition of a closed orientable surface $S$
	with respect to a pseudo-Anosov automorphism $f$.
	The \emph{flow-box complex} $X_{f,\mathcal{R}}$
	associated to $(f,\mathcal{R})$ 
	is defined to be the compact polyhedral complex
	constructed by the following procedure:
	\begin{itemize}
	\item	Take an abstract disjoint union of the birectangles
	$R_1\sqcup\cdots\sqcup R_k$ of $\mathcal{R}$, and
	for every ordered pair $(R_i,R_j)$,
	if $\interior(R_i)\cap f^{-1}(\interior(R_j))$ is nonempty in $S$,
	take a product $R_{ij}\times [0,1]$
	with $R_{ij}=R_i\cap f^{-1}(R_j)$.
	Identify $R_{ij}\times\{0\}$
	with the subset $R_{ij}$ of $R_i$ and
	$R_{ij}\times\{1\}$ with the subset $f(R_{ij})$ of $R_j$,
	in the obvious canonical fashion.
	The resulting compact polyhedral complex 
	is recorded as the declared $X_{f,\mathcal{R}}$.
	\end{itemize}
	\end{definition}
	
	The distinguished subsets	of the flow-box complex $X_{f,\mathcal{R}}$
	parametrized by $R_i$ and $R_{ij}\times[0,1]$
	are called the \emph{abstract birectangles}
	and the \emph{abstract flow boxes}, respectively.
	By an \emph{abstract flow segment} 
	we mean a forward directed segment in $X_{f,\mathcal{R}}$
	of the form $\{p\}\times [0,1]$
	as contained in some abstract flow box	$R_{ij}\times[0,1]$.
	An \emph{abstract trajectory} therefore refers to 
	an immersed oriented $1$--submanifold of $X_{f,\mathcal{R}}$
	obtained by concatenating	a sequence of directly consecutive abstract flow segments, 
	(possibly infinitely many and with repetition).
	An \emph{abstract periodic trajectory} is moreover an immersed oriented loop, 
	(unparametrized and with no basepoint).
		
	Note that there are five mutually exclusive intrinsic \emph{types} of points in any abstract birectangle:
	interior (I), side horizontal (SH), side vertical (SV), corner own-left (KL), and corner own-right (KR).
	\footnote{Imagine two Go players sitting at the bottom and the top facing each other.}
	Using a parametrization $[0,1]\times[0,1]$ of the abstract birectangle, 
	these types are given by the subsets
	$(0,1)\times(0,1)$, $(0,1)\times\{0,1\}$, $\{0,1\}\times(0,1)$, $\{(0,0),(1,1)\}$, and $\{(0,1),(1,0)\}$,
	respectively. The type of an abstract flow segment is given by its endpoint types,
	the possibilities being SV to any, I to I, KL to KL, KR to KR, or any to SH,
	and hence twelve in all.	
	It follows that any abstract periodic trajectory must visit 
	the abstract birectangles	of the flow-box complex	in only one type of points.
	Therefore, we have five types of abstract periodic trajectories in total, 
	which are again I, SH, SV, KL, and KR.
	We call the interior type \emph{ordinary}, and the rest four types \emph{exceptional}.
	
	There is a canonical quotient map
	\begin{equation}\label{zipping-map}
		q_{f,\mathcal{R}}\colon X_{f,\mathcal{R}}\to M_f,
	\end{equation}
	which we call the \emph{zipping map}.
	The zipping map includes every abstract birectangle $R_i$ of $X_{f,\mathcal{R}}$
	into the distinguished fiber $S$ of $M_f$, 
	and projects every abstract flow box $R_{ij}\times[0,1]$
	onto the flow box $\theta_{[0,1]}(R_{ij})$.
	Here $\theta_t\colon M_f\to M_f$ stands for the distinguished
	suspension flow parametrized by $t\in\Real$.
	
	\begin{lemma}\label{X-description}
		The flow-box complex $X_{f,\mathcal{R}}$ is connected.
		All primitive abstract periodic trajectories of $X_{f,\mathcal{R}}$ 
		are embedded and mutually disjoint,
		and only finitely many of them are exceptional.
	\end{lemma}
	
	\begin{proof}
		The connectedness of $X_{f,\mathcal{R}}$ is a consequence of 
		the ergodicity of pseudo-Anosov automorphisms. 
		In fact, for any pair of birectangles $R_i,R_j\in\mathcal{R}$,
		there exists some $n\in\Integral$,
		such that $f^n(\interior(R_i))\cap \interior(R_j)$ is nonempty.
		This is because the open subset $\bigcup_{m\in\Integral}f^m(\interior(R_i))\subset S$
		is invariant under $f$ and hence must be dense in $S$, 
		see \cite[Expos\'e 9, Proposition 9.18]{FLP}. 
		In particular, any pair of abstract birectangles can be connected through an abstract trajectory path,
		(that is, a path in $X_{f,\mathcal{R}}$	obtained 
		by concatenating abstract flow segments directed-consecutively).
		So $X_{f,\mathcal{R}}$ is connected.
		
		For any point of an abstract birectangle,
		if the point occurs as the initial point for two distinct abstract flow segments,
		the only possible types of the abstract flow segments are SV to SH/KL/KR, or I to SH.
		Such points never occur on abstract periodic trajectories,
		as the above abstract flow segments all connect distinct types of points.
		Similarly, no point on an abstract periodic trajectory occurs as the terminal point
		of two distinct abstract flow segments.
		Intuitively this means that primitive abstract periodic trajectories have no chance
		to bifurcate in either the forward or the backward direction.
		Therefore, primitive abstract periodic trajectories must be embedded and disjoint 
		from any other abstract periodic trajectories.
		Moreover, the stretching property of $f$ and $f^{-1}$ 
		forbids	any side of an abstract birectangle	from containing more than one point 
		of a primitive abstract periodic trajectory. 
		So there are at most finitely many exceptional primitive abstract periodic trajectories.
	\end{proof}
	
	\begin{lemma}\label{zipping-correspondence}
		The zipping map (\ref{zipping-map}) is $\pi_1$--surjective.
		Every abstract periodic trajectory of $X_{f,\mathcal{R}}$
		is immersed under the zipping map	as a periodic trajectory of the mapping torus $M_f$.
		For any primitive periodic trajectory $\gamma$ of $M_f$,
		the preimage of $\gamma$ is 
		a union of mutually disjoint primitive abstract periodic trajectories,
		each covering $\gamma$ of degree $\mathrm{po}(\gamma)$.
		All the (not necessarily occuring) possible type combinations of components and covering degrees 
		are classified by Table \ref{degree-table} below,
		(see (\ref{pn-po})).
	\end{lemma}
	
	\begin{table}[h]
		{\footnotesize
		\begin{tabular}{l c c c}
		\hline
		type combination of components & total degree &  $\mathrm{pn}(\gamma)$ &$\mathrm{po}(\gamma)$  \\
		\hline
		$\mathrm{I}^{\times1}$ & $1$ & $2$ & $1$ \\
		$\mathrm{SH}^{\times m}$ &  $md$ & $md$ & $d$\\
		$\mathrm{SV}^{\times n}$ &  $nd$ & $nd$ & $d$\\
		$\mathrm{KL}^{\times l}+\mathrm{KR}^{\times l}$
			 &  $2ld$ & $ld$ & $d$\\
		$\mathrm{KL}^{\times l}+\mathrm{KR}^{\times l}+\mathrm{SH}^{\times m}$
			 &  $(2l+m)d$ & $(l+m)d$ & $d$\\
		$\mathrm{KL}^{\times l}+\mathrm{KR}^{\times l}+\mathrm{SV}^{\times n}$
			 &  $(2l+n)d$ & $(l+n)d$ & $d$\\
		$\mathrm{KL}^{\times l}+\mathrm{KR}^{\times l}+\mathrm{SH}^{\times m}+\mathrm{SV}^{\times n}$
			 &  $(2l+m+n)d$ & $(l+m+n)d$ & $d$\\
		\hline
		\end{tabular}
		}
		\bigskip
		\caption{The numbers $m,n,l,d$ are all positive integers.
		The notation $\mathrm{pn}(\gamma)$ stands for 
		the number of prongs on $S$ at any point in $\gamma\cap S$;
		the notation $\mathrm{po}(\gamma)$ stands for the number of prongs 
		in any prong orbit at that point,
		(see (\ref{pn-po})).			
		For example, the KL+KR row of the table means that the preimage of $\gamma$ in $X_{f,\mathcal{R}}$
		consists of $2l$ primitive abstract periodic trajectories,
		half of type KL and half KR. 
		The total covering degree is $2ld$, and each component covers of degree $\mathrm{po}(\gamma)=d$.
		When $\gamma$ is a regular primitive periodic trajactory,	
		which means $\mathrm{pn}(\gamma)=2$, 
		the possible type combinations can be listed as follows:
		$\mathrm{I}$, $\mathrm{SH}^{\times 1}$, $\mathrm{SH}^{\times 2}$,
		$\mathrm{SV}^{\times 1}$, $\mathrm{SV}^{\times 2}$, 
		$\mathrm{KL}^{\times 1}+\mathrm{KR}^{\times 1}$,
		$\mathrm{KL}^{\times 2}+\mathrm{KR}^{\times 2}$,
		$\mathrm{KL}^{\times 1}+\mathrm{KR}^{\times 1}+\mathrm{SH}^{\times 1}$,
		$\mathrm{KL}^{\times 1}+\mathrm{KR}^{\times 1}+\mathrm{SV}^{\times 1}$.		
		}\label{degree-table}
	\end{table}	
	
	For our application, it is more important to remember	the qualitative conclusion 
	that there are only finitely many exceptional primitive abstract periodic trajectories
	in the flow-box complex (Lemma \ref{X-description}).
	The details of the classification table is only necessary for 
	some explicit calculation of the multivariable Lefschetz zeta function,
	(see Example \ref{kappa-zeta}). 
	
	\begin{proof}
		To see the $\pi_1$--surjectivity of the zipping map	(\ref{zipping-map}),
		we claim that 
		any loop $\alpha\colon S^1\to M_f$ 
		can be modified by homotopy to avoid the exceptional locus,
		(that is, the image of all $\partial R_i$ and $\partial R_{ij}\times(0,1)$
		under the zipping map). 
		The $\pi_1$--surjectivity follows immediately from the claim because
		the complement of the exceptional locus lifts	homeomorphically
		into $X_{f,\mathcal{R}}$.
		To prove the claim, clearly we can first homotope $\alpha$ in $M_f$ 
		so that	it is composed of 
		leaf segments of the invariant foliations in the distinguished
		fiber $S\subset M_f$ 
		and trajectory segments of the suspension flow (in either directions).
		We may also require the trajectory segments 
		to avoid the image of any $\partial R_{ij}\times (0,1)$,
		and the leaf segments to avoid or transversely intersect the image of any $\partial R_i$.
		Therefore, it suffices to argue 
		that for any open interval $b\subset \partial R_i\cap \partial R_j$, 
		and any short leaf-segment subpath $a$ of $\alpha$ 
		contained in $\interior(R_i)\cup\interior(R_j)\cup b$	
		and transverse to $b$ at a single point of intersection $p$,
		it is possible to push $a$ off the exceptional locus,
		by homotopy in $M_f$ relative to its endpoints.
		To this end, 
		we observe that for some $n\in\Integral$ and some $R_h\in\mathcal{R}$,
		the intersection $f^n(b)\cap\interior(R_h)$ is nonempty.
		This is because for any sufficiently large $m\in\Natural$,	
		one of the leaf segments $f^m(b)$ and $f^{-m}(b)$ must be stretched 
		so much that its length exceeds the total length of all $\partial R_i$,
		measured in the (defining) transverse measures of the invariant foliations.
		Take some compact subsegment $[p,p']\subset b$,
		which has $p$ as one endpoint, and some $p'$ in $b\cap f^{-n}(\interior(R_h))$
		as the other endpoint.
		Denote by $[p',q']$ the trajectory segment in $M_f$ that connects $p'$ and $q'=f^n(p')$.
		Intuitively, we modify $a$ by a finger move, 
		pushing it first along $[p,p']$, and then along $[p',q']$.
		To be precise, 
		first, we replace some small subsegment $a'\subset a$ that contains $p$
		with the other three sides of the thin birectangle $a'\times [p,p']$ on $S$,
		(namely, the union of $(\partial a')\times[p,p']$ and $a'\times\{p'\}$).
		Then, we replace some even smaller subsegment $a''$ of $a'\times\{p'\}$,
		with the other three sides of the trajectory band $a''\times[p',q']$ in $M_f$,
		(namely, the union of $(\partial a'')\times[p',q']$ and $a''\times\{q'\}$).
		The above construction gives rise to a new path $a^*$, 
		which is homotopic to the original $a$ relative to the endpoints.
		Moreover, $a^*$ avoids the exceptional locus.
		By performing the above modification for every bad point $p$ of $\alpha$ as above, 
		we obtain a new loop homotopic to $\alpha$ avoiding the exceptional locus.
		This proves the claim and therefore the $\pi_1$--surjectivity of the zipping map.
		
		It is clear that under the zipping map,
		every abstract periodic trajectory of $X_{f,\mathcal{R}}$
		becomes a periodic trajectory of $M_f$.
		For any primitive periodic trajectory $\gamma$ of $M_f$,
		the preimage of $\gamma$ is a finite cell $1$--complex form by abstract flow segments.
		It is a union of primitive abstract periodic trajectories,
		so every connected component 
		must be a primitive abstract periodic trajectory by Lemma \ref{X-description}.
		
		We figure out Table \ref{degree-table} as follows.
		If $\gamma$ avoids the exceptional locus, there is a unique ordinary abstract periodic trajectory
		which projects $\gamma$ homeomorphically,	so we obtain the first row of the classification table.
		Otherwise, the primitive abstract periodic trajectories are all exceptional by Lemma \ref{X-description}.
		At any periodic point $p\in\gamma\cap S$, 
		we read counterclockwise the birectangles around $p$,
		so we obtain a cyclic word of types $\tau=[\tau_0,\tau_1,\cdots,\tau_{s-1}]$.
		Each $\tau_i\in\{\mathrm{SH},\mathrm{SV},\mathrm{KL},\mathrm{KR}\}$
		indicates the type of $p$ in the $i$--th birectangle in the cyclical ordering.
		By definition, the smallest power of $f$ that fixes $p$ 
		permutes cyclically	the prongs of $\mathscr{F}^{\mathtt{s}}$ at $p$,
		and also the prongs of $\mathscr{F}^{\mathtt{u}}$ at $p$,
		both of order $\mathrm{po}(\gamma)=d$.
		This means that $\tau$ is $d$--periodic, and 
		the first $s/d$ terms of $\tau$	correspond bijectively to the preimage components of $\gamma$ in $X_{f,\mathcal{R}}$.
		Moreover, 
		the allowable patterns for any neighboring pair $(\tau_i,\tau_{i+1})$ in $\tau$
		can be listed as follows, (setting $\tau_{s}=\tau_0$): 
		$(\mathrm{SH},\mathrm{SH})$, $(\mathrm{SH},\mathrm{KL})$,
		$(\mathrm{SV},\mathrm{SV})$, $(\mathrm{SV},\mathrm{KR})$,
		$(\mathrm{KL},\mathrm{KR})$, $(\mathrm{KL},\mathrm{SV})$,
		and $(\mathrm{KR},\mathrm{KL})$, $(\mathrm{KR},\mathrm{SH})$.
		This follows directly from the definition of Markov partition.
		If we construct a directed graph $\Gamma$ 
		whose vertices are $\{\mathrm{SH},\mathrm{SV},\mathrm{KL},\mathrm{KR}\}$
		and whose directed edges are the allowable neighboring patterns,
		then $\tau$ can be treated as a directly immersed loop in $\Gamma$.
		By simple observation, 
		$\tau$ has to pass through the vertices $\mathrm{KL}$ and $\mathrm{KR}$ 
		the same number of times,
		unless the terms of $\tau$ are constantly $\mathrm{SH}$ or $\mathrm{SV}$.
		Therefore, if we count the terms in $\tau$ according to their types,
		there should be 
		$md$ $\mathrm{SH}$, $nd$ $\mathrm{SV}$, $ld$ $\mathrm{KL}$, and $ld$ $\mathrm{KR}$,
		where $m,n,l$ are nonnegative integers.
		Except the triples $(m>0,n>0,l=0)$ and $(m=0,n=0,l=0)$, 
		any triple $(m\geq0,n\geq0,l\geq0)$ can be realized by a directly immersed loop in $\Gamma$.
		Therefore, they can also be realized 
		by a local configuration of some Markov partition near a point.
		(In this sense, we say such triples are possible.)
		The number of pre-zipping components of $\gamma$ is $s/d=2l+m+n$.
		Each component covers $\gamma$ of degree $d$, as explained above.
		This gives the total covering degree $s=(2l+m+n)d$.
		The prongs of $\mathscr{F}^{\mathtt{s}}$ and $\mathscr{F}^{\mathtt{u}}$ together
		cut out $2\mathrm{pn}(\gamma)$ sectors near $p$.
		This number equals $(2m+2n+2l)d$ by the definition of the Markov partition and the types.
		So we obtain $\mathrm{pn}(\gamma)=(l+m+n)d$.
		Therefore, Table \ref{degree-table} follows.
	\end{proof}
	
	\subsubsection{The transition graph}\label{Subsubsec-transition_graph}	
	\begin{definition}\label{transition_digraph}
		Let $X_{f,\mathcal{R}}$ be the flow-box complex
		associated to a pseudo-Anosov automorphism $f$ of
		a closed orientable surface and a Markov partition $\mathcal{R}$.
		The \emph{transition graph} $\tdigraph_{f,\mathcal{R}}$
		associated to $(f,\mathcal{R})$ is defined to be the finite cell $1$--complex
		constructed from $X_{f,\mathcal{R}}$
		by projecting every abstract birectangle $R_i$ to a distinct $0$--cell $v_i$
		and every abstract flow box $R_{ij}\times[0,1]$ to a distinct $1$--cell $e_{ij}$
		parametrized by $[0,1]$, via the projection to the second factor.
	\end{definition}
		
	The $0$--cells of the transition graph $\tdigraph_{f,\mathcal{R}}$ 
	are called the \emph{vertices}.
	The canonically oriented $1$--cells are called the \emph{directed edges}.
	By a \emph{dynamical cycle} of $\tdigraph_{f,\mathcal{R}}$, we mean
	an immersed oriented loop, unparametrized and without a base point,
	obtained by concatenating	finitely many cyclically directly consecutive directed edges.
	It is a \emph{simple dynamical cycle}, 
	if the loop is embedded without self-intersection.
	
	By a \emph{subgraph} of $\tdigraph_{f,\mathcal{R}}$, 
	we mean a cell subcomplex with inherited edge directions.
	A subgraph is said to be \emph{irreducible},
	if every pair of vertices occurs in some dynamical cycle.
	It is said to be \emph{nonwandering},
	if every vertex occurs in some dynamical cycle,
	and if every topologically connected component is irreducible.
	Note that every directed graph contains a collection of maximal irreducible subgraphs
	(also known as the \emph{strongly connected components} in directed graph theory),
	which are mutually disjoint. 
	Any directed edge that transits between different maximal irreducible subgraphs 
	never occurs in any dynamical cycle.
	Therefore, a subgraph of $\tdigraph_{f,\mathcal{R}}$  is nonwandering
	if and only if every vertex or directed edge occurs in some dynamical cycle.
		
	\begin{remark}\label{remark_transition_graph}
	In terms of symbolic dynamics, 
	the transition graph is the directed graph model of 
	the one-sided subshift of finite type	that is determined by the Markov partition.
	Every subgraph $W$ of $T_{f,\mathcal{R}}$ can be represented as a square matrix,
	whose columns and rows are labeled by the vertices of $W$,
	and whose entries are $1$ or $0$ indicating the presence or absence
	of a directed edge between the vertices.
	This matrix is called the \emph{transition matrix} associated to $W$,
	which we denote as $A_W$.
	The \emph{one-sided subshift} determined by $W$ is 
	the subset $\mathfrak{X}_W$ of $\mathrm{Vertex}(W)^\Natural$ 
	together with a canonical shift map $\sigma\colon \mathfrak{X}_W\to \mathfrak{X}_W$:
	The points $x=(x_i)_{i\in\Natural}$ in $\mathfrak{X}_W$ 
	are determined by the relation $(A_W)_{x_ix_{i+1}}=1$ for all $i\in\Natural$.
	The map $\sigma$ is defined as $\sigma(x)_i=x_{i+1}$ for all $i\in\Natural$.
	According to \cite[Section 1.4]{Kitchens-book},
	$(\mathfrak{X}_W,\sigma)$ is said to be \emph{irreducible} if $A_W$ is irreducible,
	or \emph{nonwandering} if there are no wandering points.
	Our terminology comes from characterizations of those properties in terms of the subgraphs.
	For details, see \cite[Sections 1.1 and 5.1]{Kitchens-book}, 
	and in particular, Corollary 5.13 thereof.
	We adopt the directed graph model rather than the matrix model in this paper,
	because it is more convenient in the covering setting.
	\end{remark}
	
	There is by definition a canonical quotient map
	\begin{equation}\label{collapse-map}
	X_{f,\mathcal{R}}\to \tdigraph_{f,\mathcal{R}}
	\end{equation}
	which we call the \emph{collapse map}.
	For any subgraph $W$ of $\tdigraph_{f,\mathcal{R}}$,
	we denote by $X_{f,\mathcal{R}}(W)$ the preimage of $W$ in $X_{f,\mathcal{R}}$.
	As it is the union of the abstract birectangles and the abstract flow-boxes
	which collapse to the vertices and the directed edges $W$,
	we call $X_{f,\mathcal{R}}(W)$ the \emph{flow-box subcomplex} over $W$.
	
	\begin{lemma}\label{T-description}
		The transition graph $\tdigraph_{f,\mathcal{R}}$ is irreducible.
	\end{lemma}
	
	\begin{proof}
		This is immediately implied by the ergodicity of pseudo-Anosov automorphisms,
		\cite[Expos\'e 9, Proposition 9.18]{FLP}. 
		(See also the argument of Lemma \ref{X-description}.)
	\end{proof}
	
	\begin{lemma}\label{collapse-correspondence}
		For any subgraph $W$ of the transition graph $\tdigraph_{f,\mathcal{R}}$,
		restriction of the collapse map  (\ref{collapse-map}) 
		yields a homotopy equivalence between the flow-box subcomplex
		$X_{f,\mathcal{R}}(W)$ and $W$.
		Composition with the collapse map 
		yields a bijective correspondence	between the abstract periodic trajectories of $X_{f,\mathcal{R}}(W)$
		and the dynamical cycles of $W$.
	\end{lemma}
	
	\begin{proof}
		Note that $X_{f,\mathcal{R}}$ is topologically a polyhedral complex,
		obtained by gluing contractible pieces (abstract birectangles and flow-boxes)
		nicely along contractible sub-pieces (horizontal or vertical blocks).
		It is glued up according to the cell $1$--complex structure of $\tdigraph_{f,\mathcal{R}}$.
		(Definitions \ref{abstract_flow-box_complex} and \ref{transition_digraph}.)
		Therefore, the homotopy equivalence property of the collapse map (\ref{collapse-map})
		is readily observed from the defining construction.
		
		It is also clear that every abstract periodic trajectory of $X_{f,\mathcal{R}}$ gives rise to 
		a dynamical cycle
		of the transition graph $\tdigraph_{f,\mathcal{R}}$ via the collapse map.
		Conversely, given a dynamical cycle of the transition graph $\tdigraph_{f,\mathcal{R}}$,
		by choosing an auxiliary base vertex $R_{i_0}$ of that dynamical cycle,
		we obtain a sequence of partial maps
		$R_{i_0}\dasharrow R_{i_1}\dasharrow\cdots\dasharrow R_{i_{m-1}}\dasharrow R_{i_0}$
		where $R_{i_j}\dasharrow R_{i_{j+1}}$ 
		stands for the restricted map $f|\colon R_{i_ji_{j+1}}\to f(R_{i_ji_{j+1}})$
		for $j=0,\cdots,m-1$ and $i_m=i_0$.
		The Markov partition property implies that the set of points	$E\subset R_{i_0}$ 
		with a prescribed $m$--itinerary $R_{i_1},\cdots,R_{i_{m-1}},R_{i_0}$
		form a horizontal block of the birectangle $R_{i_0}$.
		Namely, $E$ consists of those $p\in R_{i_0}$ 
		with $f^j(p)\in R_{i_j}$ for $j=1,2,\cdots,m$,
		and $E$ is a birectangle formed by horizontal fibers of $R_{i_0}$.
		Note that $f^m(E)$ is a vertical block of $R_{i_0}$.
		The pseudo-Anosov local picture therefore 
		forces the existence of a unique point $p\in E$ with $f^m(p)=p$.
		We obtain an abstract periodic trajectory of $X_{f,\mathcal{R}}$
		by cyclically concatenating the abstract flow segments $f^j(p)\times[0,1]$
		of the abstract flow boxes $R_{i_ji_{j+1}}\times[0,1]$. 
		The uniqueness argument above implies that the constructed
		abstract periodic trajectory does not depend on 
		the auxiliary choice of the base vertex $R_{i_0}$.
		So any dynamical cycle of the transition graph $\tdigraph_{f,\mathcal{R}}$ 
		determines a unique abstract periodic trajectory of the flow-box complex $X_{f,\mathcal{R}}$.
		It is straightforward to see that the above constructions establish the claimed
		bijective correspondence.
	\end{proof}
		
	\begin{notation}\label{dc-apt-pt}
		For any dynamical cycle $z$ of the transition graph $\tdigraph_{f,\mathcal{R}}$,
		we denote by $\hat{\gamma}_z\colon S^1\to X_{f,\mathcal{R}}$ 
		the abstract periodic trajectory of the abstract flow box complex which collapses to $z$,
		and by $\gamma_z\colon S^1\to M_f$
		the zipped periodic trajectory $q_{f,\mathcal{R}}\circ\hat{\gamma}_z$.
		We treat $\hat{\gamma}_z$ and $\gamma_z$ 
		as immersed oriented loops without any preferred parametrization or any basepoint.
		In this way they are uniquely determined by $z$.
		(See Lemmas \ref{zipping-correspondence} and \ref{collapse-correspondence}.)
	\end{notation}
	
	\subsubsection{The space of projective currents}
			
	\begin{definition}\label{space_of_projective_currents}
		Let $\tdigraph_{f,\mathcal{R}}$ be the transition graph
		associated to a pseudo-Anosov automorphism $f$ of
		a closed orientable surface and a Markov partition $\mathcal{R}$.
		A \emph{projective current} on $\tdigraph_{f,\mathcal{R}}$ 
		is defined to be a probability measure on
		the finite discrete set of the directed edges of $\tdigraph_{f,\mathcal{R}}$,
		and moreover,
		it is required to be balanced at every vertex of $\tdigraph_{f,\mathcal{R}}$,
		in the sense that	the total measure of the incoming edges should be equal to
		the total measure of the outgoing edges.
		Any projective current is denoted as a formal convex combination	
		$\mu=\sum_{e}\mu_e\,e$.
		The summation here is taken over all the directed edges of $\tdigraph_{f,\mathcal{R}}$,
		and the coefficients $\mu_e\in\Real$ 
		satisfy $\sum_e\mu_e=1$, and $\mu_e\geq0$, and also the required balance condition.
		The \emph{space of projective currents} $\mathcal{P}_{f,\mathcal{R}}$ 
		is defined to be the set of all the projective currents on $\tdigraph_{f,\mathcal{R}}$.
		We treat $\mathcal{P}_{f,\mathcal{R}}$ 
		as an affine linear convex subset 
		of the abstract simplex $\mathbf{\Delta}(\tdigraph_{f,\mathcal{R}})$
		spanned by the set of directed edges of $\tdigraph_{f,\mathcal{R}}$.
	\end{definition}
	
	The \emph{support subgraph} of a projective current $\mu\in\mathcal{P}_{f,\mathcal{R}}$
	refers to the unique nonwandering subgraph of $\tdigraph_{f,\mathcal{R}}$
	whose set of directed edges is the support of the probability measure $\mu$.
	A projective current is said to be \emph{irreducible} 
	if its support subgraph is irreducible,
	and \emph{elementary} if the support subgraph is a simple dynamical cycle of
	$\tdigraph_{f,\mathcal{R}}$.
	
	There is a canonical affine linear map
	\begin{equation}\label{hd-map}
		\mathrm{hd}_{f,\mathcal{R}}\colon \mathcal{P}_{f,\mathcal{R}}\to \mathbf{A}(M_f,\phi_f),
	\end{equation}
	which we call the \emph{homology direction map}, (see Definition \ref{homology_directions}).
	For any projective current $\mu\in \mathcal{P}_{f,\mathcal{R}}$, 
	identify $\mu$ naturally as a real cellular $1$--cycle of $\tdigraph_{f,\mathcal{R}}$
	in the usual chain complex context.
	Then $\mathrm{hd}_{f,\mathcal{R}}(\mu)$ in $\mathbf{A}(M_f,\phi_f)$
	is defined to be the image of 
	the projective point $\Real\cdot[\mu]\in \mathbf{P}(H_1(\tdigraph_{f,\mathcal{R}};\Real))$
	under the projective linear composite map
	$\mathbf{P}(H_1(T_{f,\mathcal{R}};\Real))\to\mathbf{P}(H_1(X_{f,\mathcal{R}};\Real))
	\to\mathbf{P}(H_1(M_f;\Real))$
	induced by a homotopy inverse of the collapse map (\ref{collapse-map})
	and the zipping map (\ref{zipping-map}).

	\begin{lemma}\label{P-description}
		The space of projective currents $\mathcal{P}_{f,\mathcal{R}}$
		is a polytope.
		For every nonwandering subgraph $W$ of $\tdigraph_{f,\mathcal{R}}$,
		there exists a unique open face $F_W$ of $\mathcal{P}_{f,\mathcal{R}}$
		with the following characterization:
		A projective current lies in $F_W$ if and only if its support subgraph is $W$.
		This correspondence sets up an isomorphism between the lattice of nonwandering subgraphs
		of $\tdigraph_{f,\mathcal{R}}$
		and the face lattice of $\mathcal{P}_{f,\mathcal{R}}$.
	\end{lemma}
	
	Recall that a \emph{polytope} refers to a compact convex polyhedron 
	in the sense of real affine geometry,
	(see Remark \ref{remark_homology_directions} for our terminology).
	For any polytope, an open face of dimension $0$ is a vertex,
	and an open face of codimension $0$ is the point-set interior.	
	
	\begin{proof}
		Denote by $C_1(\tdigraph_{f,\mathcal{R}};\Real)$
		the $\Real$--module of cellular $1$--chains,
		with a standard unordered basis given by the directed edges.
		We identify $\mathbf{\Delta}(\tdigraph_{f,\mathcal{R}})$ 
		as the standard simplex in $C_1(\tdigraph_{f,\mathcal{R}};\Real)$,
		which consists of the convex combinations of the basis vectors,
		(namely, $\sum_e \mu_e e$ with $\sum_e\mu_e=1$ and $\mu_e\geq0$).
		The balance conditions at the vertices give rise 
		a linear system of equations on $C_1(\tdigraph_{f,\mathcal{R}};\Real)$,
		whose solution space is exactly	the $\Real$--submodule of cellular $1$--cycles, 
		$Z_1(\tdigraph_{f,\mathcal{R}};\Real)$.
		The space of projective currents $\mathcal{P}_{f,\mathcal{R}}$ 
		is the intersection of $\mathbf{\Delta}(\tdigraph_{f,\mathcal{R}})$
		with $Z_1(\tdigraph_{f,\mathcal{R}};\Real)$,
		(Definition \ref{space_of_projective_currents}).
		Therefore, $\mathcal{P}_{f,\mathcal{R}}$ is a polytope.
		It remains to establish the asserted correspondence.
		
		For any $\mu\in\mathcal{P}_{f,\mathcal{R}}$, 
		denote by $\tdigraph_{f,\mathcal{R}}(\mu)$ the support subgraph of $\mu$.
		Observe that the assignment $\mu\mapsto\tdigraph_{f,\mathcal{R}}(\mu)$
		is constant on any open face of $\mathcal{P}_{f,\mathcal{R}}$.
		(This follows from the fact that 
		faces of $\mathcal{P}_{f,\mathcal{R}}$ all arise from
		the intersection of $Z_1(\tdigraph_{f,\mathcal{R}};\Real)$
		with faces of $\mathbf{\Delta}(\tdigraph_{f,\mathcal{R}})$.)
		Therefore, we obtain a face-to-subgraph assignment
		\begin{equation}\label{face-to-subgraph}
		F\mapsto \tdigraph_{f,\mathcal{R}}(F).
		\end{equation}
		For any open face $F$ of $\mathcal{P}_{f,\mathcal{R}}$,
		$\tdigraph_{f,\mathcal{R}}(F)$ is the subgraph $\tdigraph_{f,\mathcal{R}}(\mu)$ 
		for any $\mu\in F$.
		Below we show that the assignment (\ref{face-to-subgraph}) realizes
		the asserted correspondence.
		
		The assignment (\ref{face-to-subgraph}) 
		sends any open face of $\mathcal{P}_{f,\mathcal{R}}$
		to a nonwandering subgraph of $\tdigraph_{f,\mathcal{R}}$.
		To see this, we observe that the vertices of $\mathcal{P}_{f,\mathcal{R}}$
		all have nonnegative rational coefficients at the directed edges of $\tdigraph_{f,\mathcal{R}}$.
		For any open face $F$, we take a rational convex combination of the vertices of $F$
		to obtain a projective current $\mu\in F$ with rational coefficients.
		By taking some multiple of $\mu$, we obtain a cellular $1$--cycle $\nu$
		of $\tdigraph_{f,\mathcal{R}}(F)$.
		The coefficients of $\nu$ are strictly positive 
		for the directed edges of $\tdigraph_{f,\mathcal{R}}(F)$.
		Treat $\nu$ as a collection of directed edges, with multiplicity indicated by the coefficients.
		Then the balance condition allows us to pair them up at the vertices of $\tdigraph_{f,\mathcal{R}}(F)$.
		This construction results in a collection of dynamical cycles $\tdigraph_{f,\mathcal{R}}(F)$.
		Every directed edge or vertex of $\tdigraph_{f,\mathcal{R}}(F)$ lies in
		at least one of these dynamical cycles. 
		Therefore, $\tdigraph_{f,\mathcal{R}}(F)$ must be nonwandering.	
		
		The assignment (\ref{face-to-subgraph}) 
		sends different open faces of $\mathcal{P}_{f,\mathcal{R}}$
		to different subgraphs of $\tdigraph_{f,\mathcal{R}}$.
		In fact, if $F$ and $F'$ are different open faces,
		then there must be some directed edge $e$ of $\tdigraph_{f,\mathcal{R}}(F)$,
		such that	any $\mu\in F$ and $\mu'\in F'$ have different triviality at $e$.
		In other words, 
		$\mu_e$ and $\mu'_e$ are neither simutaneously zero nor simutanenous strictly positive.
		In particular, $\mu$ and $\mu'$ have different support subgraphs.
		
		The assignment (\ref{face-to-subgraph}) 
		gives rise to every nonwandering subgraph of $\tdigraph_{f,\mathcal{R}}$.
		To see this, suppose $W$ is any nonwandering subgraph of $\tdigraph_{f,\mathcal{R}}$.
		We take the sum of all the simple dynamical cycles of $W$
		(regarded naturally as a cellular $1$--cycles),
		and then normalize to make the coefficient sum $1$.
		This construction yields to a unique projective current
		$\mu_W\in \mathcal{P}_{f,\mathcal{R}}$.
		As $W$ is nonwandering,
		any directed edge or vertex of $W$ lies in some dynamical cycle of $W$,
		which must be simple if the length is minimized among all such dynamical cycles.
		We see that $\tdigraph_{f,\mathcal{R}}(\mu_W)$ equals the given nonwandering subgraph $W$.
		So $W$ is assigned to the open face that contains $\mu_W$.

		From the above discussions, we see that the assignment (\ref{face-to-subgraph}) 
		is a bijective correspondence between the open faces of $\mathcal{P}_{f,\mathcal{R}}$
		and the nonwandering subgraphs of $\tdigraph_{f,\mathcal{R}}$.
		It preserves the (order-theoretic) lattice structure of both sides obviously.
		The inverse of the assignment (\ref{face-to-subgraph}) 
		determines a unique open face $F_W$ of $\mathcal{P}_{f,\mathcal{R}}$
		for any nonwandering subgraph $W$ of $\tdigraph_{f,\mathcal{R}}$, as asserted.
		The characterization also follows from our construction of (\ref{face-to-subgraph}).
	\end{proof}

	\begin{corollary}\label{P_face_dimension}
		For any nonwandering subgraph $W$ of $\tdigraph_{f,\mathcal{R}}$,
		the corresponding open face $F_W$ of $\mathcal{P}_{f,\mathcal{R}}$ 
		has dimension $b_1(W)-1$.
	\end{corollary}
	
	\begin{proof}
		Given any nonwandering subgraph $W$ of $\tdigraph_{f,\mathcal{R}}$,
		let $W_1\subsetneq W_2\subsetneq\cdots\subsetneq W_l=W$ 
		be any ascending chain of nonwandering subgraphs of $W$.
		Denote by $W_i/W_{i-1}$ the quotient direct graph obtained from $W_i$ 
		by collapsing	every component of $W_{i-1}$ to a distinct point.				
		We claim that if the length $l$ is maximized, 
		then $W_i/W_{i-1}$ is a simple dynamical cycle for $i=1,\cdots,l$,	
		setting $W_0$ to be a vertex in $W_1$.
		Observe that $W_i/W_{i-1}$ is nonwandering, as $W_i$ is nonwandering.
		Observe also that	the preimage in $W_i$
		of any nonwandering subgraph of $W_i/W_{i-1}$ is nonwandering, 
		as $W_{i-1}$ is either nonwandering or a single vertex.
		If $W_i/W_i$ was not a simple dynamical cycle, 
		the preimage of a simple dynamical cycle of $W_i/W_{i-1}$
		would be a nonwandering subgraph $W'$ such that $W_{i-1}\subsetneq W'\subsetneq W_i$,
		contrary to the maximality of $l$.
		Therefore, we obtain $b_1(W_i/W_{i-1})=1$ and $l=b_1(W)-1$.
		It follows from Lemma \ref{P-description}
		that the open face $F_W$ has dimension $b_1(W)-1$. 
	\end{proof}
		
	\begin{remark}\label{nonwandering_subgraph_homology}\
		\begin{enumerate}
		\item
		It follows from Corollary \ref{P_face_dimension} that
		the naturally induced projective linear map $F_W\to \mathbf{P}(H_1(W;\Real))$
		is an embedding of codimension $0$.
		Here we point out another easy fact,
		although it is not quite needed for our whole argument:
		The simple dynamical cycles of any nonwandering subgraph $W$ generates $H_1(W;\Integral)$.
		This fact also implies the above codimension--$0$ embedding property immediately.
		\item
		One may introduce the \emph{abstract homology direction hull} 
		$\mathcal{D}(X_{f,\mathcal{R}},\phi_{f,\mathcal{R}})$
		mimicking the definition of $\mathcal{D}(M_f,\phi_f)$.
		It lives in the affine subset $\mathbf{A}(X_{f,\mathcal{R}},\phi_{f,\mathcal{R}})$ 
		of the projective space $\mathbf{P}(H_1(X_{f,\mathcal{R}};\Real))$,
		complementary to the projective hyperplane $\mathbf{P}(\mathrm{Ker}(\phi_{f,\mathcal{R}}))$,
		where $\phi_{f,\mathcal{R}}$ stands for 
		$q_{f,\mathcal{R}}^*(\phi_f)\in H^1(X_{f,\mathcal{R}};\Integral)$.
		The homology direction map (\ref{hd-map}) is the zipping of the abstract homology direction map,
		in the apparent sense as indicated by the composition
		$\mathcal{P}_{f,\mathcal{R}}\to\mathbf{A}(X_{f,\mathcal{R}},\phi_{f,\mathcal{R}})\to
		\mathbf{A}(M_f,\phi)$.
		We point out that under the abstract homology direction map,
		$\mathcal{P}_{f,\mathcal{R}}$ projects isomorphically onto
		$\mathcal{D}(X_{f,\mathcal{R}},\phi_{f,\mathcal{R}})$,
		and that the latter has codimension $0$ 
		in $\mathbf{A}(X_{f,\mathcal{R}},\phi_{f,\mathcal{R}})$.
	\end{enumerate}
	\end{remark}
		
	\begin{lemma}\label{hd-description}
		Under the homology direction map (\ref{hd-map})
		the space of projective currents $\mathcal{P}_{f,\mathcal{R}}$
		projects onto the homology direction hull $\mathcal{D}(M_f,\phi_f)$.
		(See Definition \ref{homology_directions}.)
	\end{lemma}
	
	\begin{proof}
		For any dynamical cycle $C$ of $\tdigraph_{f,\mathcal{R}}$,
		the corresponding cellular $1$--cycle can be regarded as a counting measure
		on the set of the directed edges,
		so its normalization is a projective current $\mu_C$.
		We observe that
		every periodic homology direction in $\mathcal{D}(M_f,\phi_f)$
		can be realized as $\mathrm{hd}_{f,\mathcal{R}}(\mu_C)$ 
		for some dynamical cycle $C$ of $\tdigraph_{f,\mathcal{R}}$.
		This follows from Definition \ref{homology_directions} and
		Lemmas \ref{zipping-correspondence} and \ref{collapse-correspondence}.
		Then the convexity of $\mathcal{P}_{f,\mathcal{R}}$
		implies that $\mathrm{hd}_{f,\mathcal{R}}(\mathcal{P}_{f,\mathcal{R}})$ contains $\mathcal{D}(M_f,\phi_f)$.
		Note that the vertices of $\mathcal{P}_{f,\mathcal{R}}$ 
		are precisely the elementary projective currents (Lemma \ref{P-description}),
		which all occur in $\mathcal{D}(M_f,\phi_f)$ under $\mathrm{hd}_{f,\mathcal{R}}$.
		Then the convexity of $\mathcal{D}(M_f,\phi_f)$
		implies that $\mathcal{D}(M_f,\phi_f)$ contains $\mathrm{hd}_{f,\mathcal{R}}(\mathcal{P}_{f,\mathcal{R}})$.
		Therefore, $\mathrm{hd}_{f,\mathcal{R}}(\mathcal{P}_{f,\mathcal{R}})$ equals $\mathcal{D}(M_f,\phi_f)$,
		as asserted.
	\end{proof}

	\subsection{Combinatorial description of the homology direction hull}
	As an application of our discussion so far, 
	we obtain a basic description 
	of the homology direction hull:
	
	\begin{theorem}\label{D-description}
		Let $f$ be a pseudo-Anosov automorphism of a connected orientable closed surface $S$.
		Then the homology direction hull $\mathcal{D}(M_f,\phi_f)$ is 
		a polytope
		of codimension $0$ in $\mathbf{A}(M_f,\phi_f)$.
		It coincides with the point-set closure
		of all the periodic homology directions 
		in $\mathbf{A}(M_f,\phi_f)$.
		(See Definition \ref{homology_directions} and Remark \ref{remark_homology_directions}.)
	\end{theorem}
	
	\begin{proof}
		By Lemma \ref{P-description} and Corollary \ref{P_face_dimension},
		the naturally induced projective linear map 
		$\mathcal{P}_{f,\mathcal{R}}\to \mathbf{P}(H_1(T_{f,\mathcal{R}};\Real))$
		is a codimension--$0$ embedding of a polytope.
		The induced projective linear map 
		$\mathbf{P}(H_1(T_{f,\mathcal{R}};\Real))\to\mathbf{P}(H_1(X_{f,\mathcal{R}};\Real))\to\mathbf{P}(H_1(M_f;\Real))$
		as appeared in the definition of the homology direction map (\ref{hd-map})
		is surjective because of Lemmas \ref{zipping-correspondence} and \ref{collapse-correspondence}.
		By Definition \ref{homology_directions} and Lemma \ref{hd-description}, 
		$\mathcal{D}(M_f,\phi_f)$ is the image of $\mathcal{P}_{f,\mathcal{R}}$ of the composition of the above maps,
		and is contained in $\mathbf{A}(M_f,\phi_f)$.
		Therefore, 
		$\mathcal{D}(M_f,\phi_f)$ is a polytope of codimension $0$ in $\mathbf{A}(M_f,\phi_f)$.
				
		As $\tdigraph_{f,\mathcal{R}}$ is irreducible (Lemma \ref{T-description}),
		the dynamical cycles of $\tdigraph_{f,\mathcal{R}}$ give rise to all the
		rational points in the interior of $\mathcal{P}_{f,\mathcal{R}}$,
		by taking the normalization of their edge-count measures.
		Therefore, the periodic homology directions include all the rational points in the interior of $\mathcal{D}(M_f,\phi_f)$.
		So their point-set closure in $\mathbf{A}(M_f,\phi_f)$ 
		is the same as $\mathcal{D}(M_f,\phi_f)$,
		by the convexity, and the codimension $0$, and the compactness of $\mathcal{D}(M_f,\phi_f)$.
	\end{proof}
	
	\begin{remark}\label{D-of-Fried}
	Using Theorem \ref{D-description} one can identify
	the homology direction hull $\mathcal{D}(M_f,\phi_f)$ 
	with the set of homological directions
	as introduced by Fried \cite[Section 2]{Fried-sections}.
	In fact, for the suspension flow of a pseudo-Anosov automorphism,
	any homology direction occurs as an accumulation point
	of the periodic homology directions.
	It follows that $\mathcal{D}(M_f,\phi_f)$ is the same as 
	the set of homology directions in Fried's sense.
	(See also Section \ref{Sec-picture}.)
	\end{remark}
	

	\subsection{Notes}\label{Subsec-notes}
	\begin{enumerate}
	\item 
	Theorem \ref{D-description} is well-known.
	For example, it follows from Fried's characterization of homology directions using flow cross-sections \cite{Fried-sections},
	and Thurston's combinatorial description of the fibered cones 
	for $3$--manifolds \cite{Thurston-norm}.
	See \cite[Expos\'e 14]{FLP} for an exposition in that approach.
	Apart from technical details, the method we use in this section also appears in Fried \cite{Fried-flowEqv}.
	In particular, Theorem \ref{D-description} is essentially covered by \cite[Theorem H]{Fried-flowEqv} and its proof.
	\item 
	Section \ref{Subsubsec-transition_graph} is basically symbolic coding in term of the directed graph model.
	Points of the subshift associated to the Markov partition 
	encodes dynamical paths in the transition graph as their itineraries.
	The idea of using subshifts of finite types to study suspension flow dynamics,
	(such as periodic trajectories, zeta functions, entropy, etc.,)
	has been well developed, see \cite{Bowen-symbolic,Fried-flowEqv} and \cite[Expos\'e 10]{FLP}, for example.
	Our Lemmas \ref{zipping-correspondence} and \ref{collapse-correspondence} follow the idea,
	and clarify some details of the coding in the case of pseudo-Anosov suspension flows.
	The reader is referred to Kitchens' textbook \cite{Kitchens-book} for a systematic introduction to symbolic dynamics.
	\item Flow boxes are familar objects in the study of Axiom A flows, see \cite{Fried-flowEqv,Fried-zetaRS}, for example.
	The flow-box complex that appears in Section \ref{Subsubsec-flow-box_complex} 
	seems to be a new construction for organizing the Markov partition data.
	\end{enumerate}

\section{Clusters via the approach of Markov patitions}\label{Sec-cluster}
	In this section, we introduce clusters 
	for a given pseudo-Anosov automorphism of a closed orientable surface 
	with respect to a chosen Markov partition (Definition \ref{cluster}).
	For any individual cluster,	we obtain derived objects and the cluster homology direction hull,
	which are analogous to the mapping torus case.
	However, 
	there are two major reasons which make clusters more useful:
	First, the clusters form a finite partial-order system 
	with respect to the cluster subordination maps
	(Definition \ref{cluster}).
	Secondly, dimension of cluster homology direction hulls reflects
	the elementary-versus-nonelementary dichotomy in $3$--manifold topology,
	(see Theorem \ref{cluster_D-description}).
	
	It is possible to work with the covering setting directly. However, for this section,
	we keep the same setting as with Section \ref{Sec-perspective_of_Markov_partitions}
	to avoid distraction.
		
	
	\subsection{Clusters and their homology direction hulls}
	\begin{definition}\label{cluster}
		Let $f$ be a pseudo-Anosov automorphism of a connected closed orientable surface $S$,
		and $\mathcal{R}$ be a Markov partition of $S$ with respect to $f$.
		\begin{enumerate}
		\item 
		For any irreducible subgraph $V$ of the transition graph $\tdigraph_{f,\mathcal{R}}$,
		the \emph{cluster} subordinate to the mapping torus $M_f$ and associated to $V$
		is defined to be a connected Eilenberg--MacLane space $Q_V$ together with
		 a $\pi_1$--surjective map $q_V$ and a $\pi_1$--injective map $i_V$
		which make the following diagram commutative up to homotopy:
		$$\xymatrix{
		X_{f,\mathcal{R}}(V) \ar[r]^-{\mathrm{incl.}} \ar[d]_-{q_V} & X_{f,\mathcal{R}} \ar[d]^-{q_{f,\mathcal{R}}}\\
		Q_V \ar[r]^-{i_V} & M_f\\
		}$$
		Note that the triple $(Q_V,q_V,i_V)$ 
		is unique up to homotopy equivalence and homotopy.
		The irreducible subgraph $V$ is called the \emph{cluster transition graph},
		and $q_V$ the \emph{cluster zipping map}, and $i_V$ the \emph{subordination map}
		to the mapping torus.
		For brevity we often denote a cluster only by $Q_V$,
		assuming other data implicitly prescribed.
		(See Definitions \ref{abstract_flow-box_complex} and \ref{transition_digraph}.)
		\item For any pair of irreducible subgraphs $U,V$ of the transition graph $\tdigraph_{f,\mathcal{R}}$,
		the cluster $Q_U$ is said to be \emph{subordinate to} the cluster $Q_V$
		if $U$ is contained in $V$.
		In this case, the \emph{cluster subordination map} 
		is defined to be a $\pi_1$--injective map
		$$j_{U,V}\colon Q_U\to Q_V$$
		with the property $i_U\simeq i_V\circ j_{U,V}$.
		Note that $j_{U,V}$ is unique up to homotopy.
		\end{enumerate}
	\end{definition}
	
	The distinguished cohomology class for $Q_V$ is denoted as
	\begin{equation}\label{cluster_phi}
	\phi_V\in H^1(Q_V;\Integral),
	\end{equation}
	and it is defined by $\phi_V=i_V^*\phi_f$
	using the distinguished cohomology class $\phi_f$ of the mapping torus $M_f$.

	\begin{lemma}\label{Q-description}
		For any irreducible subgraph $V$ of the transition graph $\tdigraph_{f,\mathcal{R}}$,
		the cluster $Q_V$ associated to $V$ exists, and in fact,
		a model of $Q_V$ can be taken as an orientable connected compact $3$--manifold which is aspherical and atoroidal.
		The distinguished cohomology class $\phi_V$ is nontrivial.
	\end{lemma}
	
	For an orientable connected compact $3$--manifold,
	being aspherical is equivalent to the property
	that every embedded sphere bounds an embedded $3$--ball,
	and meanwhile, that the universal cover is not a $3$--sphere.
	Being atoroidal requires moreover that 
	any embedded incompressible torus should be 
	parallel to a boundary component.
	We refer the reader to Hempel \cite{Hempel-book} for terminology and standard facts in $3$--manifold topology.
	
	\begin{proof}
		Choose an auxiliary basepoint $x\in X_{f,\mathcal{R}}(V)$ and denote by $p\in M_f$ its image under zipping.
		Denote by $G_V$ the image of the induced homomorphism
		$\pi_1(X_{f,\mathcal{R}}(V),x)\to \pi_1(M_f,p)$ 
		for the restricted zipping map.
		Denote by $(\tilde{M}_V,\tilde{p})$ the covering space of $M_f$
		together with a lifted point $\tilde{p}$ of $p$
		so that $\pi_1(\tilde{M}_V,\tilde{p})$ is isomorphic to
		the subgroup $G_V$ of $\pi_1(M_f,p)$ via the induced homomorphism of the fundamental group.
		
		Note that $X_{f,\mathcal{R}}(V)\simeq V$ has a finitely generated free fundamental group, 
		so $G_V$ is finitely generated. 
		In general, 
		for any connected $3$--manifold with a finitely generated fundamental group,
		there exists an embedded compact submanifold 
		with the property that the inclusion map is a homotopy equivalence.
		This is a theorem due to P.~Scott \cite{Scott-core}, 
		and a compact submanifold  with such property
		is usually called a \emph{Scott core} for the considered $3$--manifold.
		We take any Scott core $Q_V$ for $\tilde{M}_V$
		as a model of the cluster associated to $V$.
		Since $M_f$ is the mapping torus of a pseudo-Anosov automorphism,
		it is well known that any covering space of $M_f$ is orientable, aspherical, and atoroidal,
		(see \cite[Chapter 1]{AFW-group}).
		These properties are all preserved under homotopy equivalence between $3$--manifolds,
		so the connected compact model $Q_V$ possesses the same properties as asserted.
		
		The distinguished cohomology class $\phi_V$ is nontrivial 
		because $V$ contains at least one dynamical cycle $z$.
		Indeed, the corresponding abstract periodic trajectory $\hat{\gamma}_z$ of $X_{f,\mathcal{R}}(V)$
		gives rise to a nontrivial value $\phi_V(q_{V*}[\hat{\gamma}_z])>0$, (see Notation \ref{dc-apt-pt}).
	\end{proof}
	
	
	\begin{definition}\label{cluster_homology_direction_hull}
		Let $f$ be a pseudo-Anosov automorphism of a connected closed orientable surface $S$,
		and $\mathcal{R}$ be a Markov partition of $S$ with respect to $f$.
		\begin{enumerate}
		\item
		For any cluster $Q_V$ associated to 
		an irreducible subgraph $V$ of a transition graph $\tdigraph_{f,\mathcal{R}}$,
		denote by $\mathbf{A}(Q_V,\phi_V)$
		the complement of the (projective) hyperplane $\mathbf{P}(\mathrm{Ker}(\phi_V))$	
		in the projectivization $\mathbf{P}(H_1(Q_V;\Real))$,
		where $\phi_V$ stands for the distinguished cohomology class (\ref{cluster_phi}).
		We furnish $\mathbf{A}(Q_V,\phi_V)$ 
		with the naturally induced affine linear space structure,
		(see Remark \ref{remark_homology_directions}).
		The \emph{cluster homology direction hull} 
		$$\mathcal{D}(Q_V,\phi_V)\subset\mathbf{A}(Q_V,\phi_V)$$
		is defined to be the affine linear convex hull of all the projective points
		$\Real\cdot q_{V*}[\hat{C}]\in \mathbf{A}(Q_V,\phi_V)$,
		where $\hat{C}$ ranges over all the abstract periodic trajectories of $X_{f,\mathcal{R}}(V)$.
		\item
		For any cluster $Q_U$ subordinate to $Q_V$, 
		we define the \emph{subordination map} 
		between the cluster homology direction hulls
		$$\mathcal{D}(j_{U,V})\colon \mathcal{D}(Q_U,\phi_U)\to \mathcal{D}(Q_V,\phi_V)$$
		to be the affine linear map 
		which is naturally induced by the cluster subordination map 
		$j_{U,V}\colon Q_U\to Q_V$.
		\end{enumerate}
	\end{definition}

	\subsection{Derived objects for clusters}
	For any cluster $Q_V$, 
	we call the subcomplex $X_{f,\mathcal{R}}(V)$ over $V$
	the \emph{cluster flow-box complex} for $Q_V$.
	The \emph{space of cluster projective currents} for $Q_V$
	refers to the affine linearly convex set of all the projective currents on $V$,
	which can be treated as 
	a closed face of the polytope $\mathcal{P}_{f,\mathcal{R}}$
	via the obvious inclusion (Lemma \ref{P-description}),
	so we denote it as
	$\mathcal{P}_{f,\mathcal{R}}(V)\subset \mathcal{P}_{f,\mathcal{R}}$.
	Note that the codimension--$0$ open face of $\mathcal{P}_{f,\mathcal{R}}(V)$
	is precisely the open face $F_V$ as appears in Lemma \ref{P-description}.
		
	For any cluster $Q_U$ subordinate to $Q_V$,
	there are naturally determined inclusions between the cluster transition graphs
	$U\subset V$, and between the cluster flow-box complexes
	$X_{f,\mathcal{R}}(U)\subset X_{f,\mathcal{R}}(V)$,
	and between the spaces of cluster projective currents
	$\mathcal{P}_{f,\mathcal{R}}(U)\subset\mathcal{P}_{f,\mathcal{R}}(V)$.
	The second inclusion completes a commutative diagram of maps up to homotopy
	together with the maps $q_U$, $q_V$, and $j_{U,V}$.
	The last inclusion agrees with the inclusion 
	between closed faces of $\mathcal{P}_{f,\mathcal{R}}$
	(Lemma \ref{P-description}).

	The \emph{cluster homology direction map}
	is the affine linear map
	\begin{equation}\label{cluster_hd-map}
		\mathrm{hd}_V\colon \mathcal{P}_{f,\mathcal{R}}(V)\to \mathbf{A}(Q_V,\phi_V)
	\end{equation}
	defined in the similar way as the homology direction map (\ref{hd-map}).
	To be precise, any $\mu\in\mathcal{P}_{f,\mathcal{R}}(V)$
	is a projective current fully supported on $V$, 
	and	$\mathrm{hd}_V(\mu)$ is the image of 
	the projective point $\Real\cdot[\mu]\in \mathbf{P}(H_1(V;\Real))$
	under the projective linear composite map
	$\mathbf{P}(H_1(V;\Real))\to\mathbf{P}(H_1(X_{f,\mathcal{R}}(V);\Real))
	\to\mathbf{P}(H_1(Q_V;\Real))$
	induced by a homotopy inverse of the restricted collapse map
	and the map $q_V$.
	
	\begin{lemma}\label{cluster_hd-description}
		Under the cluster homology direction map (\ref{hd-map}),
		the space of cluster projective currents $\mathcal{P}_{f,\mathcal{R}}(V)$
		projects onto the cluster homology direction hull $\mathcal{D}(Q_V,\phi_V)$.
	\end{lemma}
	
	\begin{proof}
		The argument is similar to Lemma \ref{hd-description} but even more straightforward.
		The simplification comes from Definition \ref{cluster_homology_direction_hull},
		in which we have required $\mathcal{D}(Q_V,\phi_V)$ to be spanned only by 
		those cluster periodic homology directions encoded by $V$.
		Therefore,
		analogue results as Lemmas \ref{zipping-correspondence} and \ref{collapse-correspondence}
		are no longer needed.
		Note that $\mathcal{P}_{f,\mathcal{R}}(V)$ 
		is a closed face of the polytope $\mathcal{P}_{f,\mathcal{R}}$,
		so it is spanned by the elementary projective currents on $V$, (Lemma \ref{P-description}).
		Then by inheriting the previous argument
		we conclude that $\mathcal{D}(Q_V,\phi_V)$ equals the image $\mathrm{hd}_V(\mathcal{P}_{f,\mathcal{R}}(V))$. 
	\end{proof}

	\subsection{Combinatorial description of cluster homology direction hulls}

	\begin{theorem}\label{cluster_D-description}
		Let $f$ be a pseudo-Anosov automorphism of a connected closed orientable surface $S$,
		and $\mathcal{R}$ be a Markov partition of $S$ with respect to $f$.
		Then for any irreducible subgraph $V$ of the transition graph $\tdigraph_{f,\mathcal{R}}$,
		the cluster homology direction hull 
		$\mathcal{D}(Q_V,\phi_V)$ 
		is a polytope of codimension $0$ in $\mathbf{A}(Q_V,\phi_V)$.
		
		Furthermore, suppose that $\mathcal{D}(Q_V,\phi_V)$ is not a single vertex.
		Then for any irreducible subgraph $U$ contained in $V$,
		the polytope $\mathcal{D}(Q_U,\phi_U)$ is not a single vertex
		unless $U$ is a simple dynamical cycle.
		
		(See Definitions \ref{cluster} and \ref{cluster_homology_direction_hull}.)
	\end{theorem}
	
	\begin{proof}
		The cluster homology direction hull $\mathcal{D}(Q_V,\phi_V)$ is 
		a polytope because it is the image of the cluster homology direction map,
		namely, $\mathrm{hd}_V(\mathcal{P}_{f,\mathcal{R}}(V))$,
		and because the space of cluster projective currents
		$\mathcal{P}_{f,\mathcal{R}}(V)$ is a closed face
		of the polytope $\mathcal{P}_{f,\mathcal{R}}$,
		(Lemmas \ref{cluster_hd-description} and \ref{P-description}).
		Since the cluster defining map $q_V\colon X_{f,\mathcal{R}}(V)\to Q_V$ 
		is $\pi_1$--surjective (Definition \ref{cluster}),
		it induces a surjective homomorphism 
		$q_{V*}\colon H_1(X_{f,\mathcal{R}}(V);\Real)\to H_1(Q_V;\Real)$.
		We also observe that the naturally induced projective linear map
		$\mathcal{P}_{f,\mathcal{R}}(V)\to \mathbf{P}(H_1(X_{f,\mathcal{R}}(V);\Real)$
		is an embedding of codimension $0$, 
		(Corollary \ref{P_face_dimension} and Remark \ref{nonwandering_subgraph_homology} (1)).
		As $\mathcal{D}(Q_V,\phi_V)$ is the image of the composite map
		$\mathcal{P}_{f,\mathcal{R}}(V)\to \mathbf{P}(H_1(X_{f,\mathcal{R}}(V);\Real))\to \mathbf{P}(H_1(Q_V;\Real))$,
		we see that $\mathcal{D}(Q_V,\phi_V)$ has codimension $0$ in $\mathbf{A}(Q_V,\phi_V)$,
		(Definition \ref{cluster_homology_direction_hull}).
		
		Suppose for the rest of the proof that $\mathcal{D}(Q_V,\phi_V)$ 
		is not a single vertex. 
		This is equivalent to $b_1(Q_V)>1$,
		as we have the relation
		$$\mathrm{dim}\left(\mathcal{D}(Q_V,\phi_V)\right)=\mathrm{dim}\left(\mathbf{A}(Q_V,\phi_V)\right)=b_1(Q_V)-1.$$
		For any irreducible subgraph $U$ contained in $V$,
		it suffices to argue that $b_1(Q_U)=1$ implies that $U$ is a simple dynamical cycle.
		
		To this end, observe that under the assumption $b_1(Q_U)=1$,
		the homomorphism on homology
		$$j_{U,V*}\colon H_1(Q_U;\Real)\to H_1(Q_V;\Real)$$
		induced by 
		the cluster subordination map $j_{U,V}$	must have positive-dimensional cokernel,
		(see Definition \ref{cluster}). 
		Therefore, with respect to any compatibly chosen auxiliary basepoints,
		the image of the induced homomorphism 
		$\pi_1(j_{U,V})\colon\pi_1(Q_U)\to\pi_1(Q_V)$
		has infinite index in $\pi_1(Q_V)$.
		
		We choose models for $Q_U$ and $Q_V$ 
		which are aspherical and atoroidal orientable compact $3$--manifolds (Lemma \ref{Q-description}).
		As a cluster $Q_U$ is also homotopy equivalent to an infinite cover of $Q_V$,
		which corresponds to the image of $\pi_1(j_{U,V})$,
		(see Definition \ref{cluster}).
		It follows that $b_3(Q_U)$ must vanish and therefore
		$Q_U$ has nonempty boundary.
		The boundary $\partial Q_U$ contains no sphere components,
		for otherwise the aspherical $3$--manifold $Q_U$ would have to be a $3$--ball,
		contrary to the nontriviality of $\phi_U$	(Lemma \ref{Q-description}).
		Using the well-known inequality in $3$--manifold topology:
		$$b_1(Q_U)\geq b_1(\partial Q_U)/2,$$ 
		(which is an exercise of the Poincar\'e--Lefschetz duality,)
		we see that $b_1(Q_U)=1$ occurs only if
		$\partial Q_U$ is connected and is homeomorphic to a torus.
		As $M_f$ is closed, atoroidal and aspherical, we infer that $Q_U$ must be a solid torus
		by $3$--manifold topology.
		This means that
		the restriction of the zipping map (\ref{zipping-map}),
		$X_{f,\mathcal{R}}(U)\to M_f$,
		has infinite cyclic $\pi_1$--image.
		
		To finish the proof, we invoke the fact that every primitive periodic trajectory of $M_f$
		represents a distinct conjugacy class of maximal infinite cyclic subgroups in $\pi_1(M_f)$.
		This is implied, for example, 
		by the fact that distinct fixed points of $f^m$ represent distinct fixed point classes,
		for any given $m\in\Natural$,	(see Remark \ref{fixed-point-index}).
		According to the mapping torus approach \cite[Definition 4.3]{Jiang-survey},
		any pair of fixed points $p,q$ of $f$
		represent the same fixed point class if and only if
		the loops $\gamma_1(f;p)$ and $\gamma_1(f;q)$	are freely homotopic to each other in $M_f$.
		The above fact about primitive periodic trajectories follows immediately
		from this characterization, by considering $f^m$ for all $m\in\Natural$.
		As $\pi_1(X_{f,\mathcal{R}}(U))$ projects onto an infinite cyclic subgroup of $\pi_1(M_f)$
		under zipping, we see that 
		$U$ has to be a simple dynamical cycle,
		because of the cycle--trajectory correspondences,
		(Lemmas \ref{zipping-correspondence} and \ref{collapse-correspondence}).
		This completes the proof.
	\end{proof}
	
	We close this section with the following useful notion.
	
	\begin{definition}\label{support_subgraph}
		Let $f$ be a pseudo-Anosov automorphism of a connected closed orientable surface and
		$\mathcal{R}$ be a Markov partition for $f$. 
		\begin{enumerate}
		\item
		For any cluster $Q_V$ subordinate to the mapping torus $M_f$
		and any closed face $E$ of the cluster homology direction hull $\mathcal{D}(Q_V,\phi_V)$,
		the \emph{support subgraph} for $E$,
		denoted as
		$$V[E]\subset V,$$
		is defined to be the nonwandering subgraph of the cluster transition graph $V$
		which corresponds to the closed face $\mathrm{hd}_{V}^{-1}(E)$ of 
		the space of cluster projective currents $\mathcal{P}_{f,\mathcal{R}}(V)$.
		\item
		We say that a closed face $E$ of $\mathcal{D}(Q_V,\phi_V)$ is 
		\emph{purely ordinary} if the flow-box subcomplex
		$X_{f,\mathcal{R}}(V[E])$ over the support subgraph $V[E]$
		contains no exceptional abstract periodic trajectories,
		or \emph{purely exceptional} if $X_{f,\mathcal{R}}(V[E])$
		contains no ordinary abstract periodic trajectories.
		(See Lemmas \ref{P-description} and \ref{cluster_hd-description}.)
		\end{enumerate}
	\end{definition}
	
	Note that for any purely exceptional closed face $E$ 
	of a cluster homology direction hull $\mathcal{D}(Q_V,\phi_V)$,	
	a primitive dynamical cycle $z$ of the cluster subgraph $V$ 
	is carried by the support subgraph $V[E]$
	if and only if the homology class $q_{V*}[\hat{\gamma}_z]\in H_1(Q_V;\Integral)_{\mathtt{free}}$
	lies in the linear cone over $E$, (see Definition \ref{cluster}).
	If $E$ is purely exceptional, the support subgraph $V[E]$ is necessarily 
	a union of mutually disjoint simple dynamical cycles
	due to the finiteness of exceptional abstract periodic trajectories
	(Lemma \ref{X-description}).
	Trivial examples of purely ordinary and purely exceptional closed faces
	occur for clusters associated to simple dynamical cycles,
	as the $0$--dimensional closed face of cluster homology direction hull.
	Closed faces may also be neither purely exceptional nor purely ordinary.
	The homology direction hull $\mathcal{D}(M_f,\phi_f)$ itself,
	as the codimension--$0$ closed face	for the cluster $M_f$ associated to $\tdigraph_{f,\mathcal{R}}$,
	provides a simple example.

\section{Reciprocal characteristic polynomials for clusters}\label{Sec-rcp}
	In this section, we continue the discussion of Section \ref{Sec-cluster},
	and introduce a calculation gadget for clusters
	called the reciprocal characteristic polynomial.
	At the end of this section, we exhibit three motivating examples
	for our main goal, 
	(compare Subsection \ref{Subsec-vertex-dominant}).
	
	Let $f$ be a pseudo-Anosov automorphism of a connected closed orientable surface and
	$\mathcal{R}$ be a Markov partition for $f$.
	To any primitive dynamical cycle $z$ of the transition graph $\tdigraph_{f,\mathcal{R}}$,
	we associate an integer
	\begin{equation}\label{sgn-z}
	\mathrm{sgn}(z)\in\{\pm1\}
	\end{equation}
	by the following rule: 
	Adopting Notation \ref{dc-apt-pt},
	we assign $\mathrm{sgn}(z)$ to be $1$
	if the abstract periodic trajectory $\hat{\gamma}_z$ is exceptional.
	Otherwise, as $\hat{\gamma}_z\colon S^1\to X_{f,\mathcal{R}}$ is embedded in the ordinary part,
	we observe that the tangent bundle of $X_{f,\mathcal{R}}$ is well-defined over $\hat{\gamma}_z(S^1)$.
	The tangent bundle over $\hat{\gamma}_z(S^1)$
	splits as the direct sum of three line subbundles,
	parallel to the ordinary abstract flow direction, the horizontal direction, and the vertical direction,
	respectively. The first one is always canonically oriented,
	whereas the latter two are either both orientable, or both nonorientable.
	We assign $\mathrm{sgn}(z)$ to be $1$ for the orientable case, or $-1$ for the nonorientable case.
	
	\begin{definition}\label{rcp}
	Let $f$ be a pseudo-Anosov automorphism of a connected closed orientable surface and
	$\mathcal{R}$ be a Markov partition for $f$.
	For any cluster $Q_V$ subordinate to $M_f$, and
	for any subgraph $W$ of the cluster transition graph $V$,
	we introduce the \emph{reciprocal characteristic polynomial} of $W$ 
	with respect to $Q_V$
	by the following expression, 
	in the positive half completion of $\Complex H_1(Q_V;\Integral)_{\mathtt{free}}$ with respect to the $\phi_V$--grading:
	$$\kappa^\fabcover_{f,\mathcal{R}}\left(W;Q_V\right)=
		\prod_{z\textrm{ primitive of }W}\left( 1-\mathrm{sgn}(z)\cdot q_{V*}\left[\hat{\gamma}_z\right]\right).
	$$
	Here the (possibly infinite) product is taken over all the primitive dynamical cycles $z$ of $W$,
	and $\hat{\gamma}_z$ is defined in Notation \ref{dc-apt-pt}.
	The \emph{cluster reciprocal characteristic polynomial} for $Q_V$ is defined
	using the cluster transition graph $V$, denoted particularly as	
	$$\kappa^\fabcover_{f,\mathcal{R}}\left( Q_V\right)=\kappa^\fabcover_{f,\mathcal{R}}\left(V;Q_V\right).$$
	(See Definition \ref{cluster}).
	\end{definition}
	
	The following description of $\kappa^\fabcover_{f,\mathcal{R}}(W;Q_V)$
	and its proof explain the name.
	
	\begin{lemma}\label{rcp-finiteness}
	Adopting the notations of Definition \ref{rcp},
	the reciprocal characteristic polynomial $\kappa^\fabcover_{f,\mathcal{R}}(W;Q_V)$ 
	is of the form	$1+\sum_u a_uu$
	where the summation is taken for all $u\in H_1(Q_V;\Integral)_{\mathtt{free}}$ 
	with the property $\Real u\in\mathcal{D}(Q_V,\phi_V)$ and $\phi_V(u)>0$,
	and where at most finitely many coefficients $a_u\in\Complex$ are nonzero.
	In particular, $\kappa^\fabcover_{f,\mathcal{R}}(W;Q_V)$ 
	can be identified as an element of $\Complex H_1(Q_V;\Integral)_{\mathtt{free}}$.
	\end{lemma}
	
	\begin{proof}
		We make the following auxiliary choices.
		Choose an orientation for the distinguished fiber ${S}$, which is preserved under ${f}$.
		For each birectangle ${R}_i$ of the Markov partition $\mathcal{R}=\{R_1,\cdots,R_k\}$,
		choose a Cartesian chart $(x_i,y_i)\colon {R}_i\to \Real\times\Real$,
		such that $\ud x_i\wedge\ud y_i$ agrees with the chosen orientation of $S$.
		We also require $|\ud x_i|={\mu}^{\mathtt{s}}$ and $|\ud y_i|={\mu}^{\mathtt{u}}$.
		(So $R_i$ is parametrized as a product rectangle in the Cartesian plane.)
		Note that ${f}^*\ud x_j=\pm\lambda\ud x_i$ and ${f}^*\ud y_j=\pm\lambda^{-1}\ud y_i$
		hold in any nonempty intersection $\interior(R_i)\cap {f}^{-1}(\interior(R_j))$,
		and that the signs are either both positive or both negative.
		Therefore, the auxiliary data 
		determines a function on the set of directed edges of the transition graph
		$\mathrm{sgn}\colon \mathrm{Edge}(\tdigraph_{{f},{\mathcal{R}}})\to\{\pm1\}$,
		as follows:
		For any nonempty intersection $\interior(R_i)\cap f^{-1}(\interior(R_j))$
		and the corresponding directed edge $e_{ij}$,
		$\mathrm{sgn}(e_{ij})=\pm1$	
		is assigned according to the sign in the above equations.
		
		Given any subgraph $W$ of the cluster transition graph $V$,
		denote by $C_1({W})$ the free abelian group of the cellular $1$--chains of ${W}$. 
		The commutative group algebra $\Complex C_1(W)$ can be naturally regarded as 
		a complex multivariable Laurent polynomial ring with indeterminates in $\mathrm{Edge}(W)$.
		Endow $\Complex C_1({W})$ with the natural $\Integral$--grading	
		where all the indeterminants $e\in\mathrm{Edge}(W)$ have degree $1$.
				
		Define a square matrix $\Phi(W)$ of size $|\mathcal{R}|$
		with entries in $\Complex C_1({W})$,	as follows:
		The $(i,j)$--entry of $\Phi(W)$ equals $\mathrm{sgn}(e_{ij})\cdot e_{ij}$ if 
		there is a directed edge $e_{ij}\in\mathrm{Edge}({W})$, or $0$ otherwise.
		Note that the nonzero entries of $\Phi(W)^m$ are all homogeneous of degree $m$,
		for any $m\in\Natural$.
		Moreover,
		the following identity holds in the positive-half graded completion of $\Complex C_1({W})$:
		\begin{equation}\label{identity_det_W}
		\mathrm{det}\left(\mathbf{1}-\Phi(W)\right)
		=
		\exp\left(-\sum_{m\in\Natural}\frac{\mathrm{tr}\left(\Phi(W)^m\right)}{m}\right),
		\end{equation}
		where $\mathbf{1}$	stands for the identity matrix of size $|\mathcal{R}|$ over $\Complex C_1(W)$.
		In fact, one way to see (\ref{identity_det_W}) is as follows:
		Treat the square matrix $\mathbf{1}-\Phi(W)$	
		as over an extended graded ring $\Complex C_1({W})[t,t^{-1}]$,	
		where the new indeterminant $t$ has degree $1$.
		The extended graded ring is also a Laurent polynomial ring $\Lambda[t,t^{-1}]$,
		over the subring $\Lambda$ that consists of all the degree--$0$ homogeneous elements.
		In any splitting field $K$ for 
		the characteristic polynomial of the matrix $\Phi(W)t^{-1}$ over $\Lambda$,
		the usual splitting arugment can be applied,
		so (\ref{identity_det_W}) follows from the identity 
		$1-h t=\exp\left(-\sum_{m\in\Natural} \frac{h^mt^m}{m}\right)$
		in $K[[t]]$ for all $h\in K$.
		
		Further manipulation of (\ref{identity_det_W}) yields the following identity,
		in the positive-half graded completion of $\Complex C_1({W})$:
		\begin{equation}\label{identity_exp_tr_W}
		\mathrm{det}\left(\mathbf{1}-\Phi(W)\right)
		=
		\prod_{z\textrm{ primitive of }W}\left(1-\mathrm{sgn}(z)\cdot z\right).
		\end{equation}
		The product in (\ref{identity_exp_tr_W}) is taken over all the primitive dynamical cycles $z$ of $W$.
		We adopt the multiplicative notation for the abelian group $C_1(W)$,
		and understand $z$ as a cellular $1$--cycle in the factors.
		In fact, one way to see (\ref{identity_exp_tr_W}) is as follows:
		Denote by $\mathrm{DC}_m(W)$ the set of dynamical cycles of $W$ that have combinatorial length $m$,
		and by $\mathrm{PDC}(W)$ the set of primitive dynamical cycles of $W$.
		For any $z\in\mathrm{DC}_m(W)$, denote by $\mathrm{div}(z)$ the divisibility of $z$,
		(that is, the degree of which $z$ is a cyclic cover of a primitive dynamical cycle).		
		Then the number of different closed dynamical paths
		that represent $z$ (by forgetting the basepoint) equals $m/\mathrm{div}(z)$.
		The traces of $\Phi(W)^m$ can be computed in $\Complex C_1(W)$ as:
		$$\mathrm{tr}(\Phi(W)^m)=\sum_{z\in\mathrm{DC}_m(W)} \frac{m\times\mathrm{sgn}(z)}{\mathrm{div}(z)}\cdot z.$$
		This is because each closed path that represents $z$ contributes $\mathrm{sgn}(z)\cdot z$ to the trace,
		by definition.
		Therefore, in the positive half completion of $\Complex C_1(W)$ with respect to the $\Integral$--grading,
		we compute:
		\begin{eqnarray*}
		\log\mathrm{det}\left(\mathbf{1}-\Phi(W)\right)&=&
		-\sum_{m\in\Natural}\sum_{z\in\mathrm{DC}_m(W)} \frac{\mathrm{sgn}(z)}{\mathrm{div}(z)}\cdot z\\
		&=&
		-\sum_{z\in\mathrm{PDC}(W)}\sum_{j\in\Natural} \frac{\mathrm{sgn}(z)^j\cdot z^j}{j}\\
		&=&
		\sum_{z\in\mathrm{PDC}(W)} \log(1-\mathrm{sgn}(z)\cdot z).
		\end{eqnarray*}
		Then (\ref{identity_exp_tr_W}) follows by taking exponential.
		Note that the equalities in the above computation are all understood in the graded formal series sense.
		This means that we compare the homogeneous parts of both sides degree by degree.
		In particular, there are only finitely many dynamical cycles involved in each time of comparison.
		
		From (\ref{identity_exp_tr_W}) we see	that $\mathrm{det}(\mathbf{1}-\Phi(W))$ 
		lies in the subalgebra $\Complex Z_1({W})$
		of $\Complex C_1({W})$, where $Z_1({W})$ is the abelian subgroup of $C_1({W})$
		that consists of all the cellular $1$--cycles of ${W}$.
		In fact, the left-hand side shows that it has degree at most $|\mathcal{R}|$,
		while the right-hand side shows that up to this degree
		it agrees with a $\Rational$--coefficient polynomial function
		of the primitive dynamical cycles in $W$ of length at most $|\mathcal{R}|$.
		
		As ${W}$ is a cell $1$--complex 
		homotopy equivalent to $X_{{f},{\mathcal{R}}}({W})$ (Lemma \ref{collapse-correspondence}),
		there are canonical isomorphisms of abelian groups
		\begin{equation}\label{isomorphism_Z}
		Z_1({W})\cong H_1({W};\Integral)\cong H_1({W};\Integral)_{\mathtt{free}}
		\cong H_1(X_{{f},{\mathcal{R}}}({W});\Integral)_{\mathtt{free}},
		\end{equation}
		under which $z$ is identified with $[\hat{\gamma}_z]$ by Notation \ref{dc-apt-pt}.
		Via the cluster zipping map $q_V$, we see that $\kappa^\fabcover_{f,\mathcal{R}}(W;Q_V)$
		lies in $\Complex H_1(Q_V;\Integral)_{\mathtt{free}}$.
		Then the asserted form follows the defining expression.
	\end{proof}
	
	\begin{remark}\label{remark_rcp-finiteness}
	The algebraic manipulations in our proof of Lemma \ref{rcp-finiteness} 
	is essentially equivalent to \cite[Section 2]{Fried-flowEqv}.
	In fact, computations with formal logarithms are done there
	for homological zeta functions associated with Axiom A flows,
	(see the arguments for Proposition 2 thereof).
	What becomes different here is only the context, 
	as our cluster subgraph $W$ does not arise  \textit{a priori} from an actual flow.
	\end{remark}
	
	For any cluster $Q_V$ subordinate to $M_f$ and for any subgraph $W$ of the cluster transition graph $V$,
	enumerate by $U_1,\cdots,U_r$ all the maximal irreducible subgraphs of $W$,
	(also known as the \emph{strong components} of $W$).
	Then we obtain the \emph{product formula}:
	\begin{equation}\label{kappa-product}
	\kappa^\fabcover_{f,\mathcal{R}}\left(W;Q_V\right)=
	\kappa^\fabcover_{f,\mathcal{R}}\left(U_1;Q_V\right)\times\cdots\times\kappa^\fabcover_{f,\mathcal{R}}\left(U_r;Q_V\right).
	\end{equation}
	For any other cluster $Q_U$ subordinate to $Q_V$, 
	suppose that a subgraph $W$ of $V$ is contained in the cluster transition graph $U$.
	Then we obtain the \emph{induction formula}:
	\begin{equation}\label{kappa-induce}
	\kappa^\fabcover_{f,\mathcal{R}}\left(W;Q_V\right)=j_{U,V*}\left(\kappa^\fabcover_{f,\mathcal{R}}\left(W;Q_U\right)\right)
	\end{equation}
	where $j_{U,V*}\colon \Complex H_1(Q_U;\Integral)_{\mathtt{free}}\to \Complex H_1(Q_V;\Integral)_{\mathtt{free}}$,
	is the group-algebra homomorphism induced by cluster subordination.
	Both (\ref{kappa-product}) and (\ref{kappa-induce}) are immediate consequences of Definition \ref{rcp}.
	
	For any cluster $Q_V$ subordinate to $M_f$, and any subgraph $W$ of the cluster transition graph $V$,
	write $\kappa^\fabcover_{f,\mathcal{R}}(W;Q_V)=1+\sum_u\,a_uu$ as in Lemma \ref{rcp-finiteness}.
	For any closed face $E$ of the cluster homology direction hull $\mathcal{D}(Q_V,\phi_V)$,
	we introduce the \emph{$E$--part} of 
	the reciprocal characteristic polynomial $\kappa^\fabcover_{f,\mathcal{R}}(W;Q_V)$,
	denoted as
	$\kappa^\fabcover_{f,\mathcal{R}}(W;Q_V)[E]\in\Complex H_1(Q_V;\Integral)_{\mathtt{free}}$,
	to be the sum 
	$1+\sum_{u\textrm{ over }E}\,a_u\cdot u$
	where the summation is taken over all $u\in H_1(Q_V;\Integral)_{\mathtt{free}}$
	in the $\phi_V$--positive linear cone over $E$.
	Then we obtain the \emph{face formula}:
	\begin{equation}\label{kappa-face}
	\kappa^\fabcover_{f,\mathcal{R}}\left(W;Q_V\right)[E]=\kappa^\fabcover_{f,\mathcal{R}}\left(W\cap V[E];Q_V\right)
	\end{equation}
	where $V[E]$ is the support subgraph for $E$.	
			
	\begin{example}\label{kappa-simple}
	For any cluster $Q_C$ associated to a simple dynamical cycle $C$
	of the transition graph $\tdigraph_{f,\mathcal{R}}$,
	the cluster reciprocal characteristic polynomial 
	$\kappa^\fabcover_{f,\mathcal{R}}(Q_C)$ equals $1- \mathrm{sgn}(C)\cdot u$,
	where $u$ is the generating element 
	of the infinite cyclic group $H_1(Q_V;\Integral)_{\mathtt{free}}$
	determined by $\phi_V(u)=|\mathrm{Edge}(C)|$,
	and where $\mathrm{sgn}(C)=\pm1$ is determined by the fundamental cycle of $C$.
	In particular, $\kappa^\fabcover_{f,\mathcal{R}}(Q_C)$ is not the constant polynomial $1$.
	\end{example}
		
	\begin{example}\label{kappa-zeta}
	For the cluster $M_f$ associated to the transition graph $\tdigraph_{f,\mathcal{R}}$ itself,
	the cluster reciprocal characteristic polynomial 
	$\kappa^\fabcover_{f,\mathcal{R}}(M_f)$ equals the multivariable Lefschetz zeta function $\zeta^\fabcover_f$
	up to finitely many cyclotomic polynomial factors determined
	by the exceptional periodic trajectories.
	To be precise, 
	by (\ref{mLzeta-product}), Lemma \ref{zipping-correspondence}, and Definition \ref{rcp},
	we obtain the formula
	$$\kappa^\fabcover_{f,\mathcal{R}}(M_f)=\zeta^\fabcover_f\times \prod_{\gamma\textrm{ primitive}} p_\gamma\left([\gamma]\right)$$
	where $p_\gamma$ are polynomials as indicated by Table \ref{correction-factor-table} below, 
	(see (\ref{pn-po}), Lemma \ref{zipping-correspondence}, and Table \ref{degree-table} for the notations).
	The product is taken over all the primitive periodic trajectories $\gamma$ of $M_f$,
	which is essentially finite as we ignore the trivial factors $1$.
	In particular, for any \emph{ordinary} closed face $E$ of $\mathcal{D}(M_f,\phi_f)$,
	we have an equality between the $E$--parts
	$$\kappa^\fabcover_{f,\mathcal{R}}(M_f)[E]=\zeta^\fabcover_f[E]$$
	by (\ref{kappa-face}), (see Definitions \ref{homology_directions} and \ref{support_subgraph}).
	\end{example}
	

	\begin{table}[h]
		{\footnotesize
		\begin{tabular}{l c c}
		\hline
		pre-zipping type combination & the polynomial $p_\gamma(a)$  & description of $\gamma$ \\
		\hline
		$\mathrm{I}^{\times1}$ & $1$  & ordinary\\
		$\mathrm{SH}^{\times m}$ or $\mathrm{SV}^{\times n}$ &  $1-a$ & purely side exceptional \\
		$\mathrm{KL}^{\times l}+\mathrm{KR}^{\times l}+\left(\mathrm{SH}^{\times m}\right)+\left(\mathrm{SV}^{\times n}\right)$
		&  $(1-a)\times \left(1-a^{\mathrm{po}(\gamma)}\right)^l$ & mixed or purely corner exceptional\\
		\hline
		\end{tabular}
		}
		\bigskip
		\caption{
		}\label{correction-factor-table}
	\end{table}

	\begin{example}\label{kappa-sequence}
	Consider a subordination sequence of clusters
	$$\xymatrix{
	Q_d \ar[r]^-{j_d} & \cdots \ar[r] & Q_1 \ar[r]^-{j_1} & Q_0
	}$$
	associated to an inclusion sequence of cluster transition graphs
	$V_d\subset\cdots\subset V_1\subset V_0$, 
	which are irreducible subgraphs of the transition graph $\tdigraph_{f,\mathcal{R}}$.
	Suppose the following conditions are all satisfied:
	\begin{itemize}	
	\item For $n=1,\cdots,d$, there is a closed face $E_n$ of the cluster homology hull
	$\mathcal{D}(Q_{n-1},\phi_{n-1})$ such that $V_n$ is a maximal irreducible subgraph of $V_{n-1}[E_n]$.
	\item For $n=1,\cdots,d$, the induced affine linear map
	$\mathcal{D}(j_{n})\colon \mathcal{D}(Q_n,\phi_n)\to \mathcal{D}(Q_{n-1},\phi_{n-1})$
	is an embedding.
	\item The subgraph $V_d$ is a simple dynamical cycle.
	\end{itemize}
	Therefore, we have a sequence of affine linear embeddings of affine linear polytopes
	$$\xymatrix{
	\mathcal{D}_d \ar[r] & E_{d} \ar[r] & \cdots \ar[r] 
	&	\mathcal{D}_1 \ar[r] & E_1 \ar[r] & \mathcal{D}_0
	}$$
	where $\mathcal{D}_n$ stands for $\mathcal{D}(Q_n,\phi_n)$ for $n=0,1,\cdots,d$
	and where $\mathcal{D}_d$ is a single vertex (Theorem \ref{cluster_D-description}).
	
	The third condition implies $\kappa^\fabcover_{f,\mathcal{R}}(Q_d)\neq1$ by Example \ref{kappa-simple}.
	Using the first two conditions,	we obtain $\kappa^\fabcover_{f,\mathcal{R}}(Q_{d-1})\neq1$,
	because $j_{d*}(\kappa^\fabcover_{f,\mathcal{R}}(Q_d))$ divides the $E_{d}$--part of
	$\kappa^\fabcover_{f,\mathcal{R}}(Q_{d-1})$ in $\Complex H_1(Q_{d-1};\Integral)_{\mathtt{free}}$
	by the formulas (\ref{kappa-product}), (\ref{kappa-induce}), and (\ref{kappa-face}).
	Repeating likewise, we obtain $\kappa^\fabcover_{f,\mathcal{R}}(Q_{d-2})\neq1$, 
	and $\kappa^\fabcover_{f,\mathcal{R}}(Q_{d-3})\neq1$, and so on.
	In the end, we see that the above conditions guarantee $\kappa^\fabcover_{f,\mathcal{R}}(Q_0)\neq1$, 
	and indeed, $\kappa^\fabcover_{f,\mathcal{R}}(Q_0)[E_1]\neq1$.
	\end{example}

\section{Behavior under finite covering}\label{Sec-covering}
	In this section, we extend our theory of clusters to the covering setting.
	Since all the former arguments work almost directly for covering mapping tori
	with respect to the lifted Markov partition of the distinguished fiber,
	we provide a summary about how the associated objects change under passage to
	a finite cover of the mapping torus.
	For supplementary details of the verification, see Remark \ref{remark-FLP}.
	
	Let $f$ be a pseudo-Anosov map of a connected closed orientable surface $S$.
	Denote by $M_f$ the mapping torus of $f$ and $\phi_f$ the distinguished cohomology class of $M_f$.
	For any finite cover $M'$ of $M_f$, 
	declare $M'$ as a covering mapping torus $M_{f'}$ over $M_f$ according to Convention \ref{covering-mapping-torus}.
		
	Given any Markov partition $\mathcal{R}$ of $S$ with respect to $f$,
	there is an induced Markov partition $\mathcal{R}'$ of $S'$ with respect to $f'$,
	namely, whose birectangles are precisely
	all the lifts of the birectangles from $\mathcal{R}$.
	We obtain derived objects in the same way as before.
	These include
	the flow-box complex $X_{f',\mathcal{R}'}$, 
	the transition graph $\tdigraph_{f',\mathcal{R}'}$,
	and the polytope of projective currents $\mathcal{P}_{f',\mathcal{R}'}$,
	(Definitions 
	\ref{abstract_flow-box_complex}, \ref{transition_digraph},
	and \ref{space_of_projective_currents}).
	The zipping map, and the collapse map,
	and the homology direction map are associated in the same way as before,
	(see (\ref{zipping-map}), (\ref{collapse-map}), and (\ref{hd-map})).
	The derived objects and the associated maps can be 
	described by the same statements as
	Lemmas \ref{X-description}, \ref{zipping-correspondence}, \ref{T-description}, \ref{collapse-correspondence},
	\ref{P-description}, and \ref{hd-description}.
	In particular, Theorem \ref{D-description} holds true for the covering mapping torus $M_{f'}$,
	so the homology direction hull $\mathcal{D}(M_{f'},\phi_{f'})$ 
	for the covering mapping torus (Definition \ref{homology_directions}) 
	is a polytope of codimension $0$ in $\mathbf{A}(M_{f'},\phi_{f'})$.
	Moreover, any deck transformation of $M_{f'}$ over $M_f$ preserves $\phi_{f'}$,
	and it induces isomorphic transformations	on the derived objects, 
	with respect to their own declared structures.
	The associated maps are equivariant with respect to the induced transformations.
			
	For any further finite cover $M''$ of $M'$, declare $M''$ as a covering mapping torus $M_{f''}$ over $M_f$.
	There is a commutative diagram of continuous maps
	\begin{equation}\label{TXM-covering}
	\xymatrix{
	\tdigraph_{f'',\mathcal{R}''} \ar[d] &
	X_{f'',\mathcal{R}''} \ar[l]_-{\mathrm{coll.}} \ar[d] \ar[r]^-{q_{f'',\mathcal{R}''}}& M_{f''} \ar[d]\\
	\tdigraph_{f',\mathcal{R}'} &
	X_{f',\mathcal{R}'} \ar[l]_-{\mathrm{coll.}}\ar[r]^-{q_{f',\mathcal{R}'}}& M_{f'}\\
	}
	\end{equation}
	where the vertical maps are induced as covering projections.
	There is a commutative diagram of affine linear maps
	\begin{equation}\label{PA-projection}
	\xymatrix{
	\mathcal{P}_{f'',\mathcal{R}''} \ar[r]^-{\mathrm{hd}_{f'',\mathcal{R}''}} \ar[d] & \mathbf{A}\left(M_{f''},\phi_{f''}\right) \ar[d]\\
	\mathcal{P}_{f',\mathcal{R}'} \ar[r]^-{\mathrm{hd}_{f',\mathcal{R}'}} & \mathbf{A}\left(M_{f'},\phi_{f'}\right),
	}
	\end{equation}
	where the vertical maps are induced and surjective.
	In particular, the induced affine linear map
	\begin{equation}\label{D-projection}
	\mathcal{D}\left(M_{f''},\phi_{f''}\right)\to \mathcal{D}\left(M_{f'},\phi_{f'}\right)
	\end{equation} 
	a (surjective) projection between polytopes.
	The preimage of any closed face of $\mathcal{D}(M_{f'},\phi_{f'})$
	is always a closed face of $\mathcal{D}(M_{f''},\phi_{f''})$
	with the same or higher codimension,
	but the preimage of an open face is often a union of several open faces.
	
	Clusters subordinate to a covering mapping torus $M_{f'}$ with respect to 
	the induced Markov partition $\mathcal{R}'$ are defined in the same way as
	Definition \ref{cluster}.
	Furthermore, Lemmas \ref{Q-description}, \ref{cluster_hd-description},	
	and Theorem \ref{cluster_D-description} 
	hold true for the cluster setting 
	with only change of notations.
	For clusters subordinate to a covering mapping torus,
	we have the same theory about reciprocal characteristic polynomials
	as with Section \ref{Sec-rcp}, 
	simply by dropping the connectedness assumption there.	
	
	When a covering mapping torus $M_{f''}$ covers $M_{f'}$ as above,
	some clusters subordinate to $M_{f''}$ cover
	clusters subordinate to $M_{f'}$ in a natural way.
	To be precise, 
	for any irreducible subgraph $V'$ 
	of the transition graph $\tdigraph_{f',\mathcal{R}'}$
	and any connected component $V''$
	of the preimage of $V'$ in $\tdigraph_{f'',\mathcal{R}''}$,
	there is an induced (homotopy)
	\emph{cluster covering projection} $Q_{V''}\to Q_{V'}$.
	Indeed, for any defining triple $(Q_{V'},q_{V'},i_{V'})$,
	some triple $(Q_{V''},q_{V''},i_{V''})$ and a covering projection	$Q_{V''}\to Q_{V'}$ 
	can be chosen to make the following diagram of maps commutative up to homotopy:
	\begin{equation}\label{Q-diagram}
	\xymatrix{
	V'' \ar[d] &
	X_{f'',\mathcal{R}''}(V'') \ar[l]_-{\mathrm{coll.}}\ar[d] \ar[r]^-{q_{V''}}&
	Q_{V''} \ar[d] \ar[r]^-{i_{V''}} & M_{f''} \ar[d]\\
	V' &
	X_{f',\mathcal{R}'}(V')  \ar[l]_-{\mathrm{coll.}}\ar[r]^-{q_{V'}}&
	Q_{V'}  \ar[r]^-{i_{V'}}& M_{f'}\\
	}
	\end{equation}
	(We may call any irreducible subgraph $V''$ as above 
	an \emph{elevation} of $V'$, and the cluster $Q_{V''}$ an \emph{elevation} of $Q_{V'}$,
	agreeing with customary glossary for covering spaces.)
	Cluster covering projections $Q_{V''}\to Q_{V'}$ behave 
	pretty much like covering projections between covering mapping tori $M_{f''}\to M_{f'}$.
	For example, it induces an affine linear projection between polytopes
	\begin{equation}\label{cluster_D-projection}
	\mathcal{D}\left(Q_{V''},\phi_{V''}\right)\to\mathcal{D}\left(Q_{V'},\phi_{V'}\right),
	\end{equation}
	which directly generalizes (\ref{D-projection}).

	\begin{remark}\label{remark-FLP}
		A careful inspection of the arguments of Sections \ref{Sec-perspective_of_Markov_partitions},
		\ref{Sec-cluster}, \ref{Sec-rcp} shows that connectedness of $S$ is only essentially
		used in the proof of Lemma \ref{X-description}.
		In fact, it is rather the ergodicity of $f$ on $S$ than the connectedness of $S$
		that implies the connectedness of the flow-box complex.
		Therefore, 
		Lemma \ref{X-description} holds for the covering setting as well,
		because the lifted pseudo-Anosov automorphism on the possibly disconnected distinguished fiber
		is evidently ergodic by Convention \ref{covering-mapping-torus}.
	\end{remark}

\section{Diversity of dominant virtual faces}\label{Sec-diversity_of_dominant_virtual_faces}
	In this section, we construct regular finite covers for closed pseudo-Anosov mapping tori
	so that they satisfy the hypothesis of Theorem \ref{criterion-enfeoffed}, (Theorem \ref{dominant-diversity}).
	This is done in two steps by Propositions \ref{vertices-diversity} and \ref{vertex-dominant},
	following the plan as explained after Theorem \ref{criterion-enfeoffed}.
	The proof for Theorem \ref{dominant-diversity} is summarized in Subsection \ref{Subsec-dominant-diversity}.

	\begin{theorem}\label{dominant-diversity}
		Let $f$ be a pseudo-Anosov automorphism of a connected closed orientable surface $S$.
		Then for any positive integer $N$, 
		there exists a connected regular finite cover $\tilde{M}$ of the mapping torus $M_f$,
		and moreover,
		the homology direction hull $\mathcal{D}(\tilde{M},\tilde{\phi})$ 
		contains  $N$ or more mutually distinct $\Gamma$--orbits 
		of mutually disjoint dominant closed faces.
		Here $\Gamma$ stands for the deck transformation group and
		$\tilde{M}$ is declared as a covering mapping torus 
		according to Convention \ref{covering-mapping-torus}.
		(See Definition \ref{homology_directions}, Theorem \ref{D-description}, and Section \ref{Sec-covering}.)
	\end{theorem}
	
	
	\subsection{Diversity of ordinary virtual vertices}\label{Subsec-vertices-diversity}
	
	\begin{proposition}\label{vertices-diversity}
		Let $f$ be a pseudo-Anosov automorphism of a connected closed orientable surface $S$
		and $\mathcal{R}$ be a Markov partition of $S$ for $f$.
		Then for any positive integer $N$, 
		there exists a connected regular finite cover $\tilde{M}$ of the mapping torus $M_f$,
		and moreover,
		the homology direction hull $\mathcal{D}(\tilde{M},\tilde{\phi})$ 
		contains  $N$ or more mutually distinct $\Gamma$--orbits of ordinary vertices.
		Here $\Gamma$ stands for the deck transformation group and
		$\tilde{M}$ is declared as a covering mapping torus 
		according to Convention \ref{covering-mapping-torus}.
		(See Definitions \ref{homology_directions}, \ref{support_subgraph}, 
		Theorem \ref{D-description}, and Section \ref{Sec-covering}.)
	\end{proposition}
	
	We prove Proposition \ref{vertices-diversity} in the rest of this subsection.

	In the argument, frequently we consider group-theoretic properties
	for the $\pi_1$--image of a cluster	under the subordination map.
	Both the fundamental group and the induced homomorphism 
	depend on a chosen basepoint of the cluster,
	but different choices do not matter if 
	the considered subgroup property is invariant under conjugation.
	So we simply omit mentioning the chosen basepoint whenever conjugacy invariance
	is clear from the context.
	
	\begin{lemma}\label{quasiconvex-Q}
		Suppose that $M'$ is a connected finite cover of $M_f$, 
		and that
		$E'$ is a positive codimensional closed face of the homology direction hull $\mathcal{D}(M',\phi')$.
		Then for any irreducible subgraph $V'$ of the support subgraph $\tdigraph_{f',\mathcal{R}'}[E']$, 
		the $\pi_1$--image for the subordination map of the assoicated cluster $i_{V'}\colon Q_{V'}\to M'$
		is quasiconvex of infinite index in the nonelementary word-hyperbolic group $\pi_1(M')$.
	\end{lemma}
	
	\begin{proof}				
		Denote by $L_{E'}$ the smallest linear subspace of $H_1(M';\Real)$ 
		that contains the linear cone over $E'\subset\mathbf{A}(M',\phi')\subset\mathbf{P}(H_1(M';\Real))$,
		so $L_{E'}$ has codimension at least $1$ in $H_1(M';\Real)$ by the assumption.
		Since $\mathcal{D}(M',\phi')$ is a polytope 
		with rationally defined faces in $\mathbf{A}(M',\phi')$, 
		there exists some primitive cohomology class $\psi\in H^1(M';\Integral)$
		which vanishes on $L_{E'}$ and which remains strictly positive 
		on the linear cone over $\mathcal{D}(M',\phi')\setminus E'$.
		In other words, $\psi$ lies on the point-set boundary of the open linear cone
		in $H^1(M';\Real)$ that is dual to the linear cone over $\mathcal{D}(M',\phi')$ in 
		$H_1(M';\Real)$.
		It is known that the dual open linear cone
		coincides with the fibered cone that contains $\phi'$,
		\cite[Expos\'e 14, Theorem 14.11]{FLP}, (see also Remark \ref{D-of-Fried}).
		In particular, it is implied that $\psi$ must not be a fibered class.
		In other words, it is impossible to represent $\psi$ by 
		a fiber-bundle projection $M'\to S^1$,
		under the natural isomorphism $H^1(M';\Integral)\cong [M',S^1]$.
								
		Fix a basepoint $x\in X_{f',\mathcal{R}'}(V')$.
		Denote by $p_{V'}\in Q_{V'}$ and $p\in M'$ the induced basepoint
		via $q_{V'}$ and $i_{V'}\circ q_{V'}$ (Definition \ref{cluster}).
		Denote by $\Pi_{V'}$ the image of 
		the homomorphism $\pi_1(i_{V'})\colon \pi_1(Q_{V'},p_{V'})\to \pi_1(M',p)$.
		It is clear that $\Pi_{V'}$ is finitely generated, (see Lemma \ref{Q-description}).
		We also observe that $H_1(X_{f',\mathcal{R}'}(V');\Integral)\cong H_1(V';\Integral)$
		is generated by all the abstract periodic trajectories of $X_{f',\mathcal{R}'}(V')$.
		It follows from Definitions \ref{cluster} and \ref{support_subgraph}
		that the image of $i_{V'*}\colon H_1(Q_{V'};\Integral)\to H_1(M';\Real)$ 
		is contained in the proper linear subspace $L_{E'}$.
		Therefore, $\Pi_{V'}$ is contained in the kernel of the surjective composite homomorphism
		of groups:
		\begin{equation}\label{psi-homomorphism}
		\xymatrix{\pi_1(M',p) \ar[r]^{\mathrm{abel.}} & H_1(M';\Integral) \ar[r]^-{\psi} & \Integral.
		}
		\end{equation}
		In particular, $\Pi_{V'}$ has infinite index in $\pi_1(M',p)$.
		
		For the rest of the proof, we need to recall 
		some hyperbolic geometric facts (restricted to the closed orientable case),
		which are	technically deep consequences of hyperbolization.
		Every aspherical and atoroidal connected orientable closed $3$--manifold
		admits a Riemannian metric of constant sectional curvature $-1$.
		The metric is unique up to homotopy by the Mostow Rigidity,
		and in fact,
		unique up to isotopy by the Smale conjecture,
		proved by D.~Gabai for hyperbolic $3$--manifolds \cite{Gabai-smale}.
		See Thurston \cite{Thurston-notes} and Marden \cite{Marden-book} for general references
		of hyperbolization.	
		In particular, the fundamental group of any such $3$--manifold
		is word-hyperbolic in the sense of M.~Gromov \cite{Gromov-hyperbolic}.
		It is also nonelementary, 
		meaning that there are no infinite cyclic subgroups of finite index.
		After choosing an orthonormal frame of reference at a basepoint,
		the fundamental group can be uniquely represented via the holonomy representation
		as a cocompact Kleinian group,		
		by which we mean a cocompact discrete subgroup of $\mathrm{PSL}(2,\Complex)$.
		(We identify $\mathrm{PSL}(2,\Complex)$ as the group of orientation-preserving isometries
		acting on a fixed model of hyperbolic $3$--space $\mathbb{H}^3$,
		where a preferred orthonormal frame at a basepoint is given.)
		By hyperbolic geometry, any finitely generated subgroup 
		of a cocompact Kleinian group is either \emph{geometrically finite} or \emph{geometrically infinite}, 
		(distinguished according to the covolume of its convex hull).
		Every geometrically finite subgroup is
		a quasiconvex subgroup of the nonelementary word-hyperbolic fundamental group,
		in terms of geometric group theory.
		By contrast, a finitely generated and geometrically infinite subgroup
		is never quasiconvex.
		In fact, it	always contains some (closed orientable) surface subgroup 
		of finite index (actually at most $2$). 
		Moreover, the surface subgroup is normal in some finite-index subgroup 
		of the fundamental group, and the quotient group is infinite cyclic quotient.
		We refer the reader to \cite[Sections 4.1 and 4.4]{AFW-group}
		for an informative survey on hyperbolic $3$--manifolds,
		which contains the facts that we mention above and their direct references.

		The covering mapping torus $M'$ is aspherical and atoroidal because of 
		the Nielsen--Thurston classification of surface automorphisms,
		(see \cite[Section 1.10]{AFW-group}).
		Hence $\pi_1(M',p)$ is word-hyperbolic and nonelementary,
		(see \cite[Theorem 4.4.2]{AFW-group}).
		If the finitely generated subgroup $\Pi_{V'}$ were geometrically infinite,
		there would be a finite-index subgroup $\Pi''$ of $\pi_1(M',p)$ such that
		$\Pi_{V'}\cap\Pi''$ is normal in $\Pi''$ and that	$\Pi''/(\Pi_{V'}\cap\Pi'')$ is infinite cyclic.
		(In fact, one can construct a finite cover of that corresponds to $\Pi''$,
		such that $\Pi_{V'}\cap\Pi''$ is the fundamental group of a fiber surface;
		see \cite[Theorem 4.1.2]{AFW-group}.)
		In this case,	it is easy to see that the kernel of (\ref{psi-homomorphism}) 
		intersects $\Pi''$ exactly in the subgroup $\Pi_{V'}\cap\Pi''$:
		In fact, the restriction of (\ref{psi-homomorphism}) to $\Pi''$ has to factor through
		the infinite cyclic quotient $\Pi''\to \Pi''/(\Pi_{V'}\cap\Pi'')$,
		inducing an embedding $\Pi''/(\Pi_{V'}\cap\Pi'')\to\Integral$.
		In particular, the kernel of (\ref{psi-homomorphism}) is finitely generated.
		However, finite generation of the kernel occurs if and only if $\psi$ is a fibered class,
		by a well-known fibering criterion in $3$--manifold topology due to J.~Stallings
		\cite{Stallings-fibering},	(see also \cite[Expos\'e 14, Theorem 14.2]{FLP}).
		Because $\psi$ is nonfibered by our construction,
		the subgroup $\Pi_{V'}$ must be geometrically finite,
		or equivalently, quasiconvex in the nonelementary word-hyperbolic group $\pi_1(M',p)$.
		Therefore, we have verified all the asserted properties about 
		the $\pi_1$--image $\Pi_{V'}$ for the subordination map $i_{V'}\colon Q_{V'}\to M'$.
	\end{proof}
	
	\begin{lemma}\label{filling-Q}
		Suppose that $M'$ is a connected finite cover of $M_f$. Then there exists
		an infinite and virtually abelian quotient group $G'$ 
		of the fundamental group $\pi_1(M')$ with the following property:
		For every positive codimensional closed face $E'$ of the homology direction hull $\mathcal{D}(M',\phi')$,
		and for every irreducible subgraph $V'$ of the support subgraph $\tdigraph_{f',\mathcal{R}'}[E']$, 
		the $\pi_1$--image 
		for the subordination map $i_{V'}\colon Q_{V'}\to M'$
		has finite quotient-image in $G'$.
	\end{lemma}
	
	\begin{proof}
		The key fact for this proof is that $\pi_1(M')$ is 
		nonelementary word-hyperbolic and virtually compact special \cite[Section 9]{Agol-VHC}.
		We recall that a finitely generated group 
		is said to be \emph{virtually compact special} 
		if some finite-index subgroup of the group
		is isomorphic to the fundamental group of a compact \emph{special cube complex}
		\cite[Chapter 4, Definition 4.2]{Wise-book},
		(compare \cite[Definition 3.2]{Haglund--Wise} and \cite[Definition 2.1]{AGM-MSQT}).
		Since there are only finitely many available irreducible subgraphs $V'$,
		and since the subordination maps $i_{V'}$ all have quasiconvex  $\pi_1$--images (Lemma \ref{quasiconvex-Q}), 
		Wise's Special Quotient Theorem \cite[Theorem 12.7]{Wise-book} 
		applies to $\pi_1(M')$ with respect to those quasiconvex $\pi_1$--images.
		This will yield a virtually compact special word-hyperbolic quotient group $G^*$ of $\pi_1(M')$,
		and moreover, for each $V'$ as assumed, 
		the $\pi_1$--image for $i_{V'}$ projects onto a finite subgroup in $G^*$.
		If $G^*$ happens to be infinite, 
		an infinite and virtually abelian quotient group $G'$ as asserted
		can be obtained	by taking some quotient $G^*$.
		It seems plausible that 
		one may produce an infinite $G^*$
		by performing Wise's original construction with some mild control.		
		To avoid citation of technical proofs, however,
		we outline an alternate argument
		based on more efficiently stated results available in the literature.
		These include the Malnormal Special Quotient Theorem 
		and peripheral hyperbolic Dehn fillings.
		
		Choose a basepoint $p$ for $M'$.
		Denote by $\Pi$ the fundamental group $\pi_1(M',p)$. 
		Let $\mathscr{V}'$ be the set of 
		all the irreducible subgraphs $V'$ of $\tdigraph_{f',\mathcal{R}'}$
		with the property that 
		the image of the affine linear map 
		$\mathcal{D}(i_{V'})\colon \mathcal{D}(Q_{V'},\phi_{V'})\to \mathcal{D}(M',\phi')$
		is contained in $\partial\mathcal{D}(M',\phi')$.
		For each $V'\in\mathscr{V}'$,
		choose a basepoint $p_{V'}$ for the associated cluster $Q_{V'}$, 
		and a path $\alpha_{V'}$ in $M'$ from $p$ to $i_{V'}(p_{V'})$. 
		Then there is an induced homomorphism $\pi_1(i_{V'})\colon \pi_1(Q_{V'},p_{V'})\to \pi_1(M',i_{V'}(p_{V'}))$.
		Using the path $\alpha_{V'}$,
		we can identify the image of $\pi_1(i_{V'})$ as a subgroup of $\Pi$, denoted as $\Pi_{V'}$.
		Adjusting the path $\alpha_{V'}$ is equivalent to conjugating $\Pi_{V'}$ by some element of $\Pi$.
		As we assume that $\mathcal{D}(M',\phi')$ has a positive codimensional face $E'$,
		the subgroup $\Pi_{V'}$ must be quasi-convex of infinite index in $\Pi$, (Lemma \ref{quasiconvex-Q}).
		Moreover, if we adjust the paths $\{\alpha_{V'}\colon V'\in\mathscr{V}'\}$ suitably,
		the subgroup $\Pi_{\mathscr{V}'}$	of $\Pi$ generated by the subgroups $\{\Pi_{V'}\colon V'\in\mathscr{V}'\}$
		will be isomorphic to the free amalgamation of those subgroups.
		
		Such adjustment is quite well-known to be available.
		For example, it can be done as follows:
		Enumerate the not yet adjusted subgroups as $H_0,H_1\cdots,H_s$.
		We show that there are conjugates $H_0,H_1^{g_1},\cdots,H_s^{g_s}$ 
		which generate a free-amalgam subgroup $H_0*H_1^{g_1}*\cdots*H_s^{g_s}$ of $\Pi$.
		Here we adopt the notation $H^g=g^{-1}Hg$.
		The limit set $\Lambda(\Pi)$ of $\Pi$ itself 
		is the sphere at infinity $\mathbb{S}^2_\infty$ 
		of the hyperbolic $3$--space $\mathbb{H}^3$.		
		For any infinite-index quasiconvex subgroup $H$ of $\Pi$,
		the limit set $\Lambda(H)$ is a proper closed subset of $\mathbb{S}^2_\infty$
		and has Lebesgue measure zero \cite[Theorem 8.4.2]{Thurston-notes},
		(see also \cite[Theorem 5.6.6]{Marden-book}).
		In particular, for any open subset $D\subset \mathbb{S}^2_\infty$,
		there is some $g\in\Pi$ such that $\Lambda(H^g)$ is contained in $D$.
		(For example, one may pick $g$ so that the ideal fixed points of $g$ lies in $D\setminus\Lambda(H)$, 
		and the translation distance of $g$ is sufficiently large.)
		Therefore, there is some $g_1\in \Pi$ such that	
		$\Lambda(H_0)$ and $\Lambda(H_1^{g_1})$ are disjoint in $\mathbb{S}^2_\infty$.
		By Klein's Combination Theorem \cite[Chapter VII, Theorem A.13]{Maskit-book},
		the subgroups $H_0$ and $H_1^{g_1}$ generate a free-amalgam subgroup 
		$H_0*H_1^{g_1}$ of $\Pi$, which is again quasiconvex.
		Note that $H_0*H_1^{g_1}$ still has infinite index,
		since finite-index subgroups of $\Pi$ are freely indecomposable.
		Working similarly with $H_0*H_1^{g_1}$ and $H_2$,
		we find some $g_2\in\Pi$ such that $H_0*H_1^{g_1}$ and $H_2^{g_2}$
		generate a (quasiconvex infinite-index) free-amalgam subgroup $H_0*H_1^{g_1}* H_2^{g_2}$ of $\Pi$,
		and so on. 
		Eventually we obtain the conjugates $H_i^{g_i}$ and the free-amalgam subgroup
		$H_0*H_1^{g_1}*\cdots*H_s^{g_s}$, as claimed.
		
		Let $\Pi_{\mathscr{V}'}$ be the free-amalgam subgroup of $\Pi$,
		constructed from $\{\Pi_{V'}\colon V'\in\mathscr{V}'\}$ as above.
		In particular, $\Pi_{\mathscr{V}'}$ is geometrically finite of infinite covolume.
		Therefore, $\Pi_{\mathscr{V}'}$ is quasiconvex of infinite index
		in the virtually compact special nonelementary word-hyperbolic group $\Pi$.
				
		In general, we prove the following statement:
		For any virtually compact special and nonelementary word-hyperbolic group $G$ 
		and for any infinite-index quasiconvex subgroup $H$ of $G$,
		there exists a virtually compact special, nonelementary,
		and word-hyperbolic quotient group
		$G^*$ of $G$, and moreover, $H$ has finite image in $G^*$.
		This is done by induction on the height of $H$,
		following the same argument as with \cite[Theorem A.1]{Agol-VHC}.
		Recall the \emph{height} of $H$ in $G$ is defined to be the smallest integer
		$k\geq 0$ such that for any mutually distinct $H$--cosets 
		$H,g_1H,\cdots,g_kH$ in $G$,
		the intersection $H\cap H^{g_1}\cap\cdots\cap H^{g_k}$ is a finite subgroup of $G$,
		(see \cite[Definitions 0.1 and 0.2]{GMRS} or \cite[Definition 3.23]{AGM-MSQT}).
		For any infinite-index quasiconvex subgroup of a word-hyperbolic group,
		the height always exists, \cite{GMRS}.
		
		If the height of $H$ equals $0$, $H$ is already finite,
		so the statement is trivially true, for $G^*=G$.
		If the height of $H$ is at least $1$,
		we take a malnormal core for $H$ 
		and the induced peripheral structure for $G$, (see \cite[Definition A.5]{Agol-VHC}).
		For our consideration, let us only mention that 
		the induced peripheral structure is 
		given as a finite set of subgroups $\{P_1,\cdots,P_m\}$ of $G$,
		which is \emph{almost malnormal} in the sense that for any $P_i,P_j$ and $g\in G$,
		the intersection $P_i^g\cap P_j$ is a finite subgroup of $G$
		unless $P_i$ equals $P_j$ and contains $g$.
		By the Malnormal Special Quotient Theorem
		\cite[Theorem 12.3]{Wise-book} (see also \cite[Corollary 2.8]{AGM-MSQT}),
		sufficiently long peripherally finite fillings are word-hyperbolic and virtually compact special.
		More precisely, this means that 
		there are finite-index normal subgroups $\dot{P}_i$ of $P_i$,
		for $i=1,\cdots,m$, with the following property:
		For any finite-index normal subgroups $N_i$ of $P_i$ contained in $\dot{P}_i$,		
		the quotient group $\bar{G}=G/\langle\!\langle N_1\cup\cdots\cup N_m\rangle\!\rangle$ 
		of $G$ by the normal closure of $N_1\cup\cdots\cup N_m$
		is word-hyperbolic and virtually compact special.
		We can require in addition that $\bar{G}$ is nonelementary.
		In fact, this follows from \cite[Theorems 7.2 and 11.12]{GM-filling} if $G$ is torsion-free.
		For the general case, 
		we take a finite-index torsion-free characteristic subgroup $\dot{G}$ of $G$,
		using residual finiteness of the virtually compact special group $G$ \cite[Theorem 4.4]{Wise-book}.
		We require the above asserted $\dot{P}_i$ to be contained in $\dot{G}$.
		The quotient image of $\dot{G}$ in $\bar{G}$ is a peripherally finite filling of $\dot{G}$,
		with respect to the induced peripheral structure \cite[Notation 2.9]{AGM-MSQT}.
		For sufficiently long fillings, the torsion-free case implies that the image of $\dot{G}$ in $\bar{G}$
		is nonelementary, so $\bar{G}$ is nonelementary as well.
		Furthermore, we can require that the quotient image $\bar{H}$ of $H$
		in $\bar{G}$ is quasiconvex of height strictly less than the height of $H$
		\cite[Theorem A.16]{Agol-VHC}.
		In particular, $\bar{H}$ has infinite index in $\bar{G}$, as $\bar{G}$ is infinite.
		Passing to the pair $(\bar{G},\bar{H})$ from $(G,H)$ decreases the height strictly 
		with the other hypotheses unaffected.
		Then the statement to be proved holds true by induction.
		
		Apply the above statement
		to $G=\Pi$ and $H=\Pi_{\mathscr{V}'}$. This yields a nonelementary word-hyperbolic quotient
		$G^*$ of $\Pi$ which is virtually compact special. In particular, $G^*$ contains finite-index subgroups
		with (arbitrarily large) positive	first Betti number \cite[Corollary 1.2]{Agol-VHC}.
		So $G^*$ admits an infinite and virtually abelian quotient $G'$.
		Under the composite quotient homomorphism $\Pi\to G^*\to G'$,
		the image of $\Pi_{V'}$ for any $V'\in\mathscr{V}'$
		is finite, because the image of $\Pi_{\mathscr{V}'}$ is finite.
		This complete the proof for Lemma \ref{filling-Q}. 
	\end{proof}
	
	\begin{lemma}\label{vertices-increase}
		For every positive integer $n\in\Natural$,
		there exists a connected regular finite cover $M'_n$ of $M_f$ 
		with the deck transformation group denoted by $\Gamma'_n$,
		and moreover,
		the homology direction hull $\mathcal{D}(M'_n,\phi'_n)$ is a positive-dimensional polytope
		with at least $n$ $\Gamma'_n$--orbits of vertices.
	\end{lemma}
	
	\begin{proof}
		For $n$ equal to $1$, we take a connected regular finite cover
		$M'_1$ of $M_f$ with the deck transformation group denoted by $\Gamma'_1$.
		We also require the homology direction hull $\mathcal{D}(M'_1,\phi'_1)$ not to be a single vertex.
		This is possible because of the deep fact that every pseudo-Anosov mapping torus
		admits connected finite covers with arbitrarily large first Betti number \cite[Theorem 9.2]{Agol-VHC}, 
		(see also \cite{AFW-group}).
		Then $M'_1$ satisfies the asserted property.
		
		Suppose by induction that 
		we have constructed $M'_n$ as asserted, for some $n\in\Natural$.
		To proceed to the next level, we apply Lemma \ref{filling-Q} to $M'_n$
		and obtain an infinite and virtually abelian quotient
		$$\pi_1(M'_n)\to G'_n.$$
		Take a finite-index free abelian subgroup $G''_n$ of $G'_n$.
		Choose a basepoint $p'_n\in M'_n$.
		Take a connected finite cover $M''_n$ of $M'_n$ with a lifted basepoint $p''_n$,
		such that $\pi_1(M''_n,p''_n)$ gives rise to the preimage of $G''_n$
		in $\pi_1(M'_n,p'_n)$.
		We take $M'_{n+1}$ 
		to be a connected finite cover of $M''_n$ which is regular over $M_f$. 
		
		We verify that $M'_{n+1}$ satisfies the asserted properties.
		Observe that there are induced affine linear projections of homology direction hulls
		$$\mathcal{D}(M'_{n+1},\phi'_{n+1})\to \mathcal{D}(M''_n,\phi''_n)\to \mathcal{D}(M'_n,\phi'_n),$$
		(Theorem \ref{D-description}	and Section \ref{Sec-covering}).
		In particular, $\mathcal{D}(M'_{n+1},\phi'_{n+1})$ has positive dimension, 
		as $\mathcal{D}(M'_{n},\phi'_{n})$ does.
		Over each $\Gamma'_n$--orbit of vertices in $\mathcal{D}(M'_n,\phi'_n)$,
		there is at least one $\Gamma'_{n+1}$--orbit of vertices in $\mathcal{D}(M'_{n+1},\phi'_{n+1})$,
		so $\mathcal{D}(M'_{n+1},\phi'_{n+1})$ already has at least $n$ distinct $\Gamma'_{n+1}$--orbits
		of vertices, by the induction hypothesis.
		To obtain at least one extra vertex orbit in $\mathcal{D}(M'_{n+1},\phi'_{n+1})$,
		it suffices to show that the convex hull of the preimage of $\partial \mathcal{D}(M'_n,\phi'_n)$ in
		$\mathcal{D}(M'_{n+1},\phi'_{n+1})$ is not the entire $\mathcal{D}(M'_{n+1},\phi'_{n+1})$.
		In fact, we show that this convex hull has positive codimension in $\mathcal{D}(M'_{n+1},\phi'_{n+1})$.
		To this end, 
		observe that the abelianized homomorphism 
		$$H_1\left(M''_n;\Integral\right)\to G''_n$$
		is surjective, by the construction of $M''_n$.
		If $E'$ is any positive-codimensional closed face of $\mathcal{D}(M'_n,\phi'_n)$,
		and if $E''$ is the closed face of $\mathcal{D}(M''_n,\phi''_n)$ given as the preimage of $E'$, 
		there is an induced covering projection between the support subgraphs 
		$\tdigraph_{f''_n,\mathcal{R}''_n}[E'']\to \tdigraph_{f'_n,\mathcal{R}'_n}[E']$.
		If $Q_{V''}$ is any cluster subordinate to $M''_n$ 
		which is associated to an irreducible subgraph $V''$ of $\tdigraph_{f''_n,\mathcal{R}''_n}[E'']$, 
		the construction of $M''_n$ guarantees that	the composite homomorphism 
		$$\xymatrix{
		H_1\left(Q_{V''};\Integral\right) \ar[r]^-{i_{V''*}} & H_1\left(M''_n;\Integral\right) \ar[r] & G''_n
		}$$
		is trivial.
		Therefore, 
		the subspace of $H_1(M''_n;\Real)$ 
		spanned by all the $i_{V''*}(H_1(Q_{V''};\Real))$	as above
		must have codimension at least the rank of $G''_n$.
		Passing to homology direction hulls,
		this means that 
		the preimage of $\partial\mathcal{D}(M'_n,\phi'_n)$ in $\mathcal{D}(M''_n,\phi'')$
		is contained in some affine linear subset of codimension at least the rank of $G''_n$.
		The rank of $G''_n$ is positive because $G''_n$ is a finite-index free abelian subgroup
		in the infinite virtually abelian group $G'_n$.
		Pulling back via $\mathcal{D}(M'_{n+1},\phi'_{n+1})\to \mathcal{D}(M''_n,\phi''_n)$,
		we see that
		the preimage of $\partial\mathcal{D}(M'_n,\phi'_n)$ in $\mathcal{D}(M'_{n+1},\phi'_{n+1})$
		is also contained in an affine linear subset of the same positive codimension.
		Therefore, the convex hull of the preimage of $\partial\mathcal{D}(M'_n,\phi'_n)$
		cannot be the entire $\mathcal{D}(M'_{n+1},\phi'_{n+1})$.
		As explained above, 
		$\mathcal{D}(M'_{n+1},\phi'_{n+1})$ has at least $(n+1)$ $\Gamma'_{n+1}$--orbits of vertices.
		This completes the induction.
	\end{proof}
	
	To prove Proposition \ref{vertices-diversity}, 
	we observe that there are only finitely many, say $K$, 
	primitive exceptional abstract periodic trajectories
	in the flow-box complex $X_{f,\mathcal{R}}$, (Lemma \ref{X-description}).
	For any regular finite cover $\tilde{M}$ of $M_f$ with the deck transformation group denoted by $\Gamma$,
	the primitive exceptional abstract periodic trajectories of $X_{\tilde{f},\tilde{\mathcal{R}}}$
	are precisely the preimage of those of $X_{f,\mathcal{R}}$,
	with respect to the induced covering projection $X_{\tilde{f},\tilde{\mathcal{R}}}\to X_{f,\mathcal{R}}$.
	So there are at most $K$ distinct $\Gamma$--orbits 
	of primitive exceptional abstract periodic trajectories in $X_{\tilde{f},\tilde{\mathcal{R}}}$.
	Given any positive integer $N\in\Natural$,
	take $\tilde{M}$ to be a regular finite cover $M'_{N+K}$ 
	as provided by Lemma \ref{vertices-increase}.
	We see that the homology direction hull
	$\mathcal{D}(\tilde{M},\tilde{\phi})$ has at least $(N+K)$ $\Gamma$--orbits
	of vertices, and that at most $K$ of them could be represented by exceptional abstract periodic trajectories.
	In other words, there remains to be $N$ or more $\Gamma$--orbits 
	of ordinary vertices (Definition \ref{support_subgraph}), as asserted.
	
	This completes the proof of Proposition \ref{vertices-diversity}.

	\subsection{Dominance over ordinary virtual vertices}\label{Subsec-vertex-dominant}
	
	\begin{proposition}\label{vertex-dominant}
		Let $f$ be a pseudo-Anosov automorphism of a connected closed orientable surface $S$
		and $\mathcal{R}$ be a Markov partition of $S$ for $f$.
		Then for any connected finite cover $M'$ of the mapping torus $M_f$ 
		and any ordinary vertex $v'$ of the homology direction hull for $M'$, 
		there exists a connected finite cover $\tilde{M}$ of $M'$
		with the property that the preimage of $v'$
		is a dominant closed face of the homology direction hull for $\tilde{M}$. 
		Here $M'$ and $\tilde{M}$ are declared as covering mapping tori over $M_f$
		according to Convention \ref{covering-mapping-torus}.
		(See Definitions \ref{homology_directions}, \ref{support_subgraph},
		Theorem \ref{D-description}, and (\ref{D-projection}).)
	\end{proposition}
	
	We prove Proposition \ref{vertex-dominant} in the rest of this subsection.

	\begin{lemma}\label{sequence-VW}
		Suppose that $M'$ is a connected finite cover of the mapping torus $M_f$ 
		and that $v'$ is an ordinary vertex of
		the homology direction hull $\mathcal{D}(M',\phi')$.
		Then there exist 
		a finite sequence of successively subordinate clusters 
		$$\xymatrix{
		Q'_d \ar[r] & \cdots \ar[r] & Q'_2 \ar[r] & Q'_1 \ar[r] & M',
		}$$
		and moreover, the following conditions are all satisified:
		\begin{itemize}
		\item The cluster homology direction hull $\mathcal{D}(Q'_1,\phi'_1)$ 
		projects onto the vertex $v'$,
		and the cluster $Q'_1$ is maximal subject to this property.
		\item For every $n=2,\cdots,d$, the cluster homology direction hull
		$\mathcal{D}(Q'_n,\phi'_n)$ 
		projects onto some vertex of $\mathcal{D}(Q'_{n-1},\phi'_{n-1})$,
		and the cluster $Q'_n$ is maximal subject to this property.
		\item The cluster homology direction hull
		$\mathcal{D}(Q'_d,\phi'_d)$ is a single vertex.
		\end{itemize}
		Here the projections of cluster homology direction hulls and
		the partial ordering of clusters are both understood as induced by subordination.
	\end{lemma}
	
	\begin{proof}
		Take a maximal irreducible subgraph $V'_1$ of the support subgraph $\tdigraph_{f',\mathcal{R}'}[v']$
		(Definition \ref{support_subgraph}).
		(Topologically $V'_1$ is nothing but a connected component, 
		since $\tdigraph_{f',\mathcal{R}'}[v']$ is nonwandering.)
		We construct $Q'_1$	as the cluster subordinate to $M'$ and assoicated to $V'_1$.
		Suppose by induction that 
		a sequence of clusters $Q'_1,\cdots,Q'_n$ have been constructed for some $n\in\Natural$,
		associated to a properly descending chain of irreducible subgraphs 
		$V'_1\supsetneq V'_2\supsetneq\cdots\supsetneq V'_n$.
		If $V'_n$ is a simple dynamical cycle, we terminate our construction.
		Otherwise, take a vertex $w'_{n+1}$
		of the cluster homology direction hull $\mathcal{D}(Q'_n,\phi'_n)$, 
		and take a maximal irreducible subgraph $V'_{n+1}$ of the cluster support subgraph
		$V'_n[w'_{n+1}]$. Construct $Q'_{n+1}$ as the cluster associated to $V'_{n+1}$.
		
		Since there is an ordinary vertex $v'$, 
		$\mathcal{D}(M',\phi')$ must have positive dimension.
		It follows that $\mathcal{D}(Q'_n,\phi'_n)$ has positive dimension
		unless $V'_n$ is a simple dynamical cycle
		(Theorem \ref{cluster_D-description}).
		Then every $V'_{n+1}$ is properly contained in $V'_n$ unless 
		the construction terminates at the level $n$.
		However, the construction has to terminate at some level $n=d$,
		because the transition graph $\tdigraph_{f',\mathcal{R}'}$ is finite.
		It is straightforward to check that the clusters $Q'_1,\cdots,Q'_d$ 
		satisfy the asserted properties, (see Theorem \ref{cluster_D-description} and Definition \ref{support_subgraph}).		
	\end{proof}

	\begin{lemma}\label{elevation-sequence-VW}
		Suppose that $M'$ is a connected finite cover of the mapping torus $M_f$ 
		and that $v'$ is an ordinary vertex of
		the homology direction hull $\mathcal{D}(M',\phi')$.
		Then for any sequence of clusters $Q'_1,\cdots,Q'_d$ subordinate to $M'$
		as described by Lemma \ref{sequence-VW},
		there exist a connected finite cover $\tilde{M}$ of $M'$ 
		and clusters $\tilde{Q}_1,\cdots,\tilde{Q}_d$ subordinate to $\tilde{M}$,
		and moreover, the following conditions are all satisfied:
		\begin{itemize}
		\item The following diagram of maps is commutative up to homotopy:
		$$\xymatrix{
		\tilde{Q}_d \ar[r] \ar[d]& \cdots \ar[r] & \tilde{Q}_2 \ar[r] \ar[d]& \tilde{Q}_1 \ar[r] \ar[d]& \tilde{M} \ar[d]\\
		Q'_d \ar[r] & \cdots \ar[r] & Q'_2 \ar[r] & Q'_1 \ar[r] & M'
		}$$
		In the diagram, the horizontal maps are all cluster subordination maps, 
		and the vertical maps are all cluster covering projections.
		\item The subordination maps for cluster homology direction hulls
		$\mathcal{D}(\tilde{Q}_1,\tilde{\phi}_1)\to\mathcal{D}(\tilde{M},\tilde{\phi})$,
		and		
		$\mathcal{D}(\tilde{Q}_n,\tilde{\phi}_n)\to\mathcal{D}(\tilde{Q}_{n-1},\tilde{\phi}_{n-1})$
		for $n=2,\cdots,d$,
		are all affine linear embeddings.
		\item The map $\tilde{Q}_d\to Q'_d$ is a homotopy equivalence.
		\end{itemize}
	\end{lemma}
	
	\begin{proof}
		For convenience we rewrite $M'$ as $Q'_0$. Choose a basepoint $p'_d$ for $Q'_d$, 
		and for $n=0,1,\cdots,d-1$, endow $Q'_{n}$ with the basedpoint $p'_n$ 
		which is the image of $p'_d$ under the cluster subordination map $Q'_d\to Q'_n$.
		Denote by $\Pi_n$ the fundamental group $\pi_1(Q'_n,p'_n)$, for $n=0,\cdots,d$.
		As the cluster subordination maps are all $\pi_1$--injective, 
		for $n=1,\cdots,d$, we identify each $\Pi_n$ as a subgroup of $\Pi_{n-1}$,
		so $\Pi_n$ are all subgroups of $\Pi_0$.
		We also identify $\Pi_0$ as a cocompact Kleinian group by hyperbolization.
		Observe $b_1(M')>1$ because $v'$ is an ordinary vertex, so $\Pi_1,\cdots,\Pi_d$
		are all geometrically finite in $\Pi_0$ by Lemmas \ref{quasiconvex-Q} and \ref{sequence-VW}.
		Therefore, for $n=1,\cdots,d$, $\Pi_{n-1}$ is word-hyperbolic
		and $\Pi_n$ is quasiconvex in $\Pi_{n-1}$ \cite{Swarup}, (see also \cite[Theorem 2]{Kapovich--Short}).
		Moreover, $\Pi_{n-1}$ is virtually compact special \cite[Theorem 16.6]{Wise-book}.
		
		For $n=0,\cdots,d$, we construct a finite-index subgroup $\dot{\Pi}_n$ of $\Pi_n$ as follows.
		First take $\dot{\Pi}_d=\Pi_d$. Since $\dot{\Pi}_d$ is quasiconvex
		in the virtually compact special word-hyperbolic group $\Pi_{d-1}$,
		$\dot{\Pi}_d$ is a virtual retract of $\Pi_{d-1}$ by \cite[Theorem 7.3]{Haglund--Wise},
		(see also \cite[Proof 2 of Theorem 4.13]{Wise-book}).
		In other words, there exists a finite index subgroup $\dot{\Pi}_{d-1}$ of $\Pi_{d-1}$
		which contains $\dot{\Pi}_d$, and moreover,
		there exists a homomorphism $\dot{\Pi}_{d-1}\to\dot{\Pi}_d$
		which extends the identity homomorphism of $\dot{\Pi}_d$.
		So we take such a $\dot{\Pi}_{d-1}$ for our construction.
		Proceeding likewise, we construct finite-index subgroups 
		$\dot{\Pi}_n$ of $\Pi_n$, for $n=d-2,\cdots,1,0$.
		The construction makes sure that 
		each $\dot{\Pi}_n$ is contained in $\dot{\Pi}_{n-1}$
		as a retract, for $n=1,\cdots,d$.
		
		For $n=0,\cdots,d$,
		take $\tilde{Q}_n$ to be the connected finite cover of $Q'_n$ 
		that corresponds to the subgroup $\dot{\Pi}_n$ of $\Pi_n$.
		We rewrite $\tilde{Q}_0$ as $\tilde{M}$, the asserted connected finite cover of $M'$.
		The finite covers $\tilde{Q}_n$ are naturally clusters subordinate to the covering mapping torus $\tilde{M}$.
		In fact, for $n=1,\cdots,d$,
		the cluster transition graph $\tilde{V}_n$ of $\tilde{Q}_n$ is a preimage component of
		the cluster transition graph $V'_n$ of $Q'_n$,
		under the induced covering projection of transition graph
		$\tdigraph_{\tilde{f},\tilde{\mathcal{R}}}\to\tdigraph_{f',\mathcal{R}'}$,
		(see Definition \ref{cluster} and (\ref{Q-diagram})).
		The successive retract property from the construction implies 
		that the sequence of homomorphisms induced by cluster subordination maps
		$$\xymatrix{
		H_1(\tilde{Q}_d;\Real) \ar[r] & \cdots \ar[r] & H_1(\tilde{Q}_2;\Real) \ar[r] & H_1(\tilde{Q}_1;\Real) \ar[r] & H_1(\tilde{M};\Real)
		}$$
		are all embeddings. This implies the asserted successive embeddings of cluster homology direction hulls
		$\mathcal{D}(\tilde{Q}_n,\tilde{\phi}_n)$. The remaining asserted properties are obvious from the construction.
	\end{proof}
		
	To prove Proposition \ref{vertex-dominant},
	we apply Lemmas \ref{sequence-VW} and \ref{elevation-sequence-VW}
	to construct a connected finite cover
	\begin{equation}\label{claimed-tilde-M}
	\tilde{M}\to M',
	\end{equation}
	with respect to the given ordinary vertex $v'$
	of the homology direction hull $\mathcal{D}(M',\phi')$.
	We retain 
	the clusters $Q'_1,\cdots,Q'_d$ subordinate to $M'$
	and the	clusters $\tilde{Q}_1,\cdots,\tilde{Q}_d$ subordinate to $\tilde{M}$
	in those lemmas for our record.
		
	Let us quickly analyze the result of our construction 
	on the homology direction level.
	For $n=1,\cdots,d$, 
	denote by $\mathcal{D}'_n$ the cluster homology direction hull for $Q'_n$,
	and $\tilde{\mathcal{D}}_n$ for $\tilde{Q}_n$.
	For $n=2,\cdots,d$, denote by $w'_n$ the vertex of $\mathcal{D}'_{n-1}$ 
	that is projected by $\mathcal{D}'_n$, and 
	denote by $\tilde{E}_n$ the closed face of $\tilde{\mathcal{D}}_n$
	that is the preimage of $w'_n$ under the cluster covering-induced projection
	(\ref{cluster_D-projection}).
	Denote by $\tilde{E}_{v'}$ the closed face of $\mathcal{D}(\tilde{M},\tilde{\phi})$
	that is the preimage of $v'$.
	Then the situation can be summarized by
	the following commutative diagram of affine linear maps of polytopes:
	\begin{equation}\label{diagram-DE}
	\xymatrix{
	\tilde{\mathcal{D}}_d \ar[r] \ar[d] & \tilde{E}_d \ar[r] \ar[d] & \cdots \ar[r] 
	&\tilde{\mathcal{D}}_2 \ar[r] \ar[d] & \tilde{E}_2 \ar[r] \ar[d]
	&\tilde{\mathcal{D}}_1 \ar[r] \ar[d] & \tilde{E}_{v'} \ar[r] \ar[d]	
	& \mathcal{D}\left(\tilde{M},\tilde{\phi}\right) \ar[d]\\
	\mathcal{D}'_d \ar[r] & w'_d \ar[r]  & \cdots \ar[r] 
	&\mathcal{D}'_2 \ar[r] & w'_2 \ar[r]
	&\mathcal{D}'_1 \ar[r] & v' \ar[r] & \mathcal{D}\left(M',\phi'\right) 
	}
	\end{equation}
	In the diagram (\ref{diagram-DE}), the vertical arrows are all surjective;
	the horizontal arrows in the upper row are embeddings; 
	the horizontal arrows in the lower row are
	projections onto singletons and embeddings as vertices, alternately.
	
	It remains to verify that $\tilde{E}_{v'}$ is a dominant closed face of  $\mathcal{D}(\tilde{M},\tilde{\phi})$.
	To this end, we apply the calculation of Example \ref{kappa-sequence}
	on the level of the covering mapping torus $\tilde{M}$ (see also Section \ref{Sec-covering}), 
	putting the cluster $\tilde{Q}_0$ as $\tilde{M}$ and 
	the closed face $\tilde{E}_1$ as $\tilde{E}_{v'}$.
	It follows that the $\tilde{E}_{v'}$--part of the reciprocal characteristic polynomial
	for $\tilde{M}$ satisfies
	$$\kappa^\fabcover_{\tilde{f},\tilde{\mathcal{R}}}\left(\tilde{M}\right)\left[\tilde{E}_{v'}\right]\neq1.$$
	As $v'$ is an ordinary vertex of $\mathcal{D}(M',\phi')$, 
	the closed face $\tilde{E}_{v'}$ of $\mathcal{D}(\tilde{M},\tilde{\phi})$ is also ordinary,
	by Definition \ref{support_subgraph}.
	Then the calculation of Example \ref{kappa-zeta} implies
	that the $\tilde{E}_{v'}$--part of the multivariable Lefschetz zeta function satisfies
	$$\zeta^\fabcover_{\tilde{f}}\left[\tilde{E}_{v'}\right]=\kappa^\fabcover_{\tilde{f},\tilde{\mathcal{R}}}\left(\tilde{M}\right)\left[\tilde{E}_{v'}\right]\neq1$$
	By Definition \ref{homology_directions}, this means exactly that $\tilde{E}_{v'}$
	is a dominant closed face of $\mathcal{D}(\tilde{M},\tilde{\phi})$.
	In other words, $\tilde{M}$ is a connected finite cover of $M'$
	which satisfies the asserted property of Proposition \ref{vertex-dominant}
	with respect to the given ordinary vertex $v'$.
	
	This completes the proof of Proposition \ref{vertex-dominant}.

	\subsection{{Proof of Theorem \ref{dominant-diversity}}}\label{Subsec-dominant-diversity}
	We obtain by Proposition \ref{vertices-diversity}
	a connected regular finite cover $M'$ of $M_f$,
	whose deck transformation group is denoted as $\Gamma'$.
	There are at least $N$ mutually distinct
	$\Gamma'$--orbits of ordinary vertices 
	in the homology direction hull $\mathcal{D}(M',\phi')$.
	We list them by some representative vertices $v'_1,\cdots,v'_N$.
	For each $v'_i$, obtain by Proposition \ref{vertex-dominant}
	a connected finite cover $M''_i$ over $M'$, so that the preimage of $v'_i$
	is a dominant closed face $E''_i$ of $\mathcal{D}(M''_i,\phi''_i)$.
	Take a common connected finite cover $\tilde{M}$ for all $M''_1,\cdots,M''_N$,
	and moreover, require $\tilde{M}$ to be regular over $M_f$.
	With respect to the induced affine linear projections
	$$\mathcal{D}\left(\tilde{M},\tilde{\phi}\right)\to \mathcal{D}\left(M''_i,\phi''_i\right),$$
	denote by $\tilde{E}_i$ the preimage of $E''_i$, 
	which are closed faces of $\mathcal{D}(\tilde{M},\tilde{\phi})$. 
	
	It is clear by the definition of dominant closed faces
	that $\tilde{E}_1,\cdots,\tilde{E}_N$ are also dominant closed faces of $\mathcal{D}(\tilde{M},\tilde{\phi})$,
	(see Definition \ref{homology_directions}).
	Denote by $\Gamma$ the deck transformation group of $\tilde{M}$ over $M_f$,
	according to the notation of Theorem \ref{dominant-diversity}.
	Note that the affine linear projection
	$$\mathcal{D}\left(\tilde{M},\tilde{\phi}\right)\to \mathcal{D}\left(M',\phi'\right)$$
	is equivariant with respect to the quotient homomorphism of deck transformation groups
	$$\Gamma\to\Gamma'.$$
	The closed faces $\tilde{E}_1,\cdots,\tilde{E}_N$ are mutually disjoint
	since they project the mutually distinct vertices $v'_1,\cdots,v'_N$.
	Moreover, the $\Gamma$--orbits of the closed faces are mutually disjoint
	(as $\Gamma$--invariant subsets), 
	since the vertices represent mutually distinct $\Gamma'$--orbits.
	Therefore, for the connected regular finite cover $\tilde{M}$ over $M_f$,
	the homology direction hull $\mathcal{D}(\tilde{M},\tilde{\phi})$ 
	possesses at least $N$ mutually distinct
	$\Gamma$--orbits	of mutually disjoint dominant closed faces,
	as asserted.
	
	This completes the proof of Theorem \ref{dominant-diversity}.

\section{Proof of the virtual homological spectral radius conjecture}\label{Sec-main_proof}
	In this section, we prove Theorem \ref{main-vhsr}, which confirms Conjecture \ref{conjecture-vhsr}.
	The essential case for pseudo-Anosov automorphisms
	of closed surfaces is summarized as the following Lemma \ref{vhsr-pA-closed},
	based on all the techniques that we have developed so far.
	The bounded pseudo-Anosov case and the partially pseudo-Anosov case 
	are derived by Lemmas \ref{vhsr-pA-bounded} and \ref{vhsr-pA-partial}.
	In the end, we complete the proof of Theorem \ref{main-vhsr}
	using the Nielsen--Thurston classification.
		
	\begin{lemma}\label{vhsr-pA-closed}
		For any pseudo-Anosov automorphism $f$ of a connected closed orientable surface $S$,
		some virtual homological spectral radius for $f$ is strictly greater than $1$.
	\end{lemma}
	
	\begin{proof}
		Denote by $M_f$ the mapping torus for $f$ with the distinguished cohomology class $\phi_f$.
		Take $N$ to be the positive integer $-\chi(S)+1$.
		We apply Theorem \ref{dominant-diversity}
		to obtain a connected regular finite cover $\tilde{M}$ of $M_f$
		with the deck transformation group denoted by $\Gamma$.
		Declare $\tilde{M}$ as a covering mapping torus 
		according to Convention \ref{covering-mapping-torus}.
		Then the homology direction hull $\mathcal{D}(\tilde{M},\tilde{\phi})$
		has at least $N$ mutually disjoint $\Gamma$--invariant dominant semi-extreme subsets,
		each given by a $\Gamma$--orbit of dominant closed faces as their union.
		So the Mahler measure of multivariable Alexander polynomial of $\tilde{M}$
		is strictly greater than $1$	by Theorem \ref{criterion-enfeoffed}.
		By \cite[Theorem 1.2]{Sun-vhsr},
		we conclude that for some connected finite cover $S'$ of $S$ and 
		some lift $f'\colon S'\to S'$ of $f$, 
		the homological spectral radius for $f'$ is strictly greater than $1$,
		as asserted.
	\end{proof}
		
	For a connected compact orientable surface $S$ with nonempty boundary,
	the interior of $S$ is a connected punctured orientable surface of finite type.
	In this setting, 
	an orientation-preserving self-homeomorphism $f\colon S\to S$
	is called a \emph{pseudo-Anosov automorphism}
	if $S$ has negative Euler characteristic,
	and if $f$ preserves projectively a pair of measured foliations 
	$(\mathscr{F}^{\mathtt{u}},\mu^{\mathtt{u}})$
	and $(\mathscr{F}^{\mathtt{s}},\mu^{\mathtt{s}})$	on the punctured surface $\interior(S)$.
	The underlying invariant foliations are transverse to each other 
	except at finitely many common singular points in the interior or at the punctures.
	For any singular point in the interior, 
	we require both of the invariant foliations have a $k$--prong singularity,
	for some and the same positive integer $k\geq3$;
	for any singular point at a puncture, we allow $1$--prong singularity in addition.
	(This agrees with the terminology of \cite[Expos\'e 11, Section 11.3]{FLP}.)
		
	\begin{lemma}\label{vhsr-pA-bounded}
		For any pseudo-Anosov automorphism $f$ of a connected compact orientable surface $S$ with nonempty boundary,
		some virtual homological spectral radius for $f$ is strictly greater than $1$.
	\end{lemma}
	
	\begin{proof}
		For any characteristic finite cover $\tilde{S}$ of $S$,
		$f$ admits a lift to $\tilde{S}$ as a pseudo-Anosov automorphism $\tilde{f}$,
		and the invariant foliations $\tilde{\mathscr{F}}^{\mathtt{s}}$ and $\tilde{\mathscr{F}}^{\mathtt{u}}$ 
		for $\tilde{f}$ are obtained via pull-back.
		We require that every boundary component of $\tilde{S}$
		covers a boundary component of $S$ of degree at least $2$.
		Then the pull-back invariant foliations have no $1$--prong singularities.
		In this case, we obtain a closed orientable surface $\tilde{S}_{\mathtt{fill}}$ by collapsing
		every boundary component of $\tilde{S}$ to a point, so $\tilde{f}$
		descends to be a pseudo-Anosov automorphism $\tilde{f}_{\mathtt{fill}}$
		of $\tilde{S}_{\mathtt{fill}}$.
		For any lift $f'_{\mathtt{fill}}$ of $\tilde{f}_{\mathtt{fill}}$ 
		to some finite cover $S'_{\mathtt{fill}}$ of $\tilde{S}_{\mathtt{fill}}$,
		we have an induced lift $f'$ of $\tilde{f}$ to the induced finite cover $S'$ of $\tilde{S}$.
		
		For any lift $f'_{\mathtt{fill}}$ of $\tilde{f}_{\mathtt{fill}}$ as above,
		it is easy to see that $f'$ and $f'_{\mathtt{fill}}$ have the same homological spectral radius,
		(by considering the homology long exact sequence	induced by filling the punctures for $S'$).
		Since some $f'_{\mathtt{fill}}$ has 
		homological spectral radius strictly greater than $1$ (Lemma \ref{vhsr-pA-bounded}), 
		the same property holds for $f'$, which lifts $f$, as asserted.
	\end{proof}
	
	\begin{lemma}\label{vhsr-pA-partial}
		Let $f$ be an automorphism of a connected compact orientable surface $S$.
		Suppose that $S_0$ is an essentially embedded connected subsurface of $S$ which is invariant under $f$.
		If the restriction of $f$ to $S_0$ is a pseudo-Anosov automorphism of $S_0$,
		then some virtual homological spectral radius for $f$ is strictly greater than $1$.
	\end{lemma}
	
	\begin{proof}
		We start by explaining a basic construction for finite covers of surfaces
		(sometimes known as \emph{completing a semi-cover}):
		For any finite cover $S'_0\to S_0$, 
		we can construct a finite cover $S'$ of $S$ and an embedding of $S'_0$ into $S'$,
		so that they fit into the following commutative diagram:
		$$\xymatrix{
		S'_0 \ar[r]^-{\mathrm{emb.}} \ar[d]_{\mathrm{cov.}} & S' \ar[d]^{\mathrm{cov.}}\\
		S_0 \ar[r]^-{\mathrm{incl.}} & S
		}$$
		
		The above construction follows from the fact that the fundamental group of $S$ is LERF, 
		due to P.~Scott \cite[Theorem 3.1]{Scott-LERF}. 
		To be more precise, we recall 
		the following topological characterization of the LERF property \cite[Lemma 1.1]{Scott-LERF}:
		Let $X$ be a Hausdorff topological space with a regular covering space $\widehat{X}$
		and deck transformation group $G$. 
		Then $G$ is LERF if and only if given a finitely generated subgroup 
		$H$ of $G$ and a compact subset $K$ of $\widehat{X}/H$, 
		there is a finite covering projection $X'\to X$ such that the covering projection
		$\widehat{X}/H\to X$ factors through $X'$ and $K$ projects homeomorphically into $X'$.
		For convenience, we choose an auxiliary basepoint $*$ of $S$ which lies in $S_0$,
		and choose a lifted basepoint $*'$ in $S'_0$.
		Then $H=\pi_1(S'_0,*')$ can be naturally identified as a finitely generated subgroup of $G=\pi_1(S,*)$.
		Let $\widehat{S}$ be a universal cover of $S$ with a lifted basepoint,
		so $G$ acts naturally on $\widehat{S}$ as deck transformations.
		The composite map $S'_0\to S_0\to S$ lifts to $\widehat{S}/H$ as an embedding.
		Since $G$ is LERF \cite[Theorem 3.1]{Scott-LERF}, the covering projection $\widehat{S}/H \to S$
		factors through some finite cover $S'$ of $S$, and the embedded image of $S'_0$
		projects homeomorphically into $S'$.
		Therefore, we obtain a finite covering map $S'\to S$ and 
		an embedding $S'_0\to \widetilde{S}/H\to S'$, as claimed.
				
		To prove Lemma \ref{vhsr-pA-partial}, 
		it suffices to assume that the boundary of $S_0$ is nonempty,
		otherwise we are done by Lemma \ref{vhsr-pA-closed}.
		Denote by $f_0\colon S_0\to S_0$ the restriction of $f$ to $S_0$.
		We apply Lemma \ref{vhsr-pA-bounded}
		to obtain a connected finite cover $S'_0$,
		and an automorphism $f'_0\colon S'_0\to S'_0$ that lifts $f_0$ with homological spectral radius $>1$.
		By the above construction,
		we can extend $S'_0$ to obtain a finite cover $S'$ of $S$.
		Take a characteristic finite cover $\tilde{S}$ of $S$ that factors through $S'$,
		and take an automorphism $\tilde{f}\colon \tilde{S}\to\tilde{S}$ that lifts $f$.
		It suffices to show that $\tilde{f}$ has homological spectral radius  $>1$.
		
		To this end,
		let $\tilde{S}_0\subset \tilde{S}$ be a connected component of the preimage of $S'_0$. 
		There exists some $m\in \Natural$
		such that $\tilde{f}^m$ preserves $\tilde{S}_0$ and commutes with the deck transformations of $\tilde{S}$ over $S$.
		In particular, $\tilde{f}^m$ descends to $S'$ as an automorphism $F'\colon S'\to S'$.
		The restricted automorphism $F'_0\colon S'_0\to S'_0$ is necessarily of the form $\sigma\circ (f'_0)^m$, 
		where $\sigma$ is some deck transformation of $S'_0$ over $S_0$.
		Possibly after raising $m$ to a positive multiple, 
		(for example, $[S'_0:S_0]$ times $m$,)
		we may assume that $F'_0$ equals $(f'_0)^m$.
		As $\tilde{S}_0\to S'_0$ has finite degree,
		every homological eigenvalue of $(f'_0)^m$ off the unit circle
		is also a homological eigenvalue of the restriction of $\tilde{f}^m$ to $\tilde{S}_0$,
		(by the surjectivity of $H_1(\tilde{S}_0;\Complex)\to H_1(S'_0;\Complex)$).
		As $\tilde{S}_0$ is essentially embedded in $\tilde{S}$,
		every homological eigenvalue of $\tilde{f}^m|_{\tilde{S}_0}$ off the unit circle
		is also a homological eigenvalue of $\tilde{f}^m$,
		(by considering the homology Mayer--Vietoris sequence for the pair 
		$(\tilde{S}_0,\tilde{S}\setminus\interior(\tilde{S}_0))$).
		Since $f'_0\colon S'_0\to S'_0$ has homological spectral radius  $>1$,
		the same holds for $(f'_0)^m$, and hence for $\tilde{f}^m$.
		It follows that $\tilde{f}$ must also have homological spectral radius $>1$.		
	\end{proof}

	For any automorphism $f$ of a connected compact orientable surface $S$ in general,
	the Nielsen--Thurston classification implies that there exists a possibly empty
	collection of mutually disjoint essential simple closed curves on $S$ as follows:
	Up to isotopy, $f$ preserves the union of the curves, and therefore,
	some positive power $f^m$ of $f$ preserves each of the complementary components.
	Moreover, for each of the complementary component,
	(path-compactified as a connected compact orientable surface),
	the restriction of $f^m$ to that component is isotopic to 
	either a periodic automorphism or a pseudo-Anosov automorphism.
	(See \cite[Expos\'e 11, Theorem 11.7]{FLP}.)
	It is known that the mapping-class entropy of $f$ is positive 
	if and only if there is at least one component as above of the pseudo-Anosov type,
	(see \cite[Corollary 10]{Kojima}).
	In this case, 
	$f^m$, and hence $f$, 
	must have some virtual homological spectral radius which is strictly greater than $1$,
	by Lemmas \ref{vhsr-pA-closed}, \ref{vhsr-pA-bounded}, and \ref{vhsr-pA-partial}.
	Otherwise, some positive power of $f$ is isotopic to
	a Dehn multitwist along the above decomposition curves,
	namely, a product of integral Dehn twists along the curves.
	In this case,
	obviously
	any virtual homological spectral radius of $f$ must be $1$.
	
	This completes the proof of Theorem \ref{main-vhsr}.

\bibliographystyle{amsalpha}


\end{document}